\documentclass[12pt]{amsart}

\usepackage{frontmatter}

\title[Periodic Traveling Waves in Dimer FPUT Without Symmetry]{Small-Amplitude Periodic Traveling Waves in Dimer Fermi--Pasta--Ulam--Tsingou Lattices Without Symmetry}

\author{Timothy E.\@ Faver}
\address{Department of Mathematics, Kennesaw State University, 850 Polytechnic Lane, Marietta, GA 30060 USA, {\tt{tfaver1@kennesaw.edu}}}

\author{Hermen Jan Hupkes}
\address{Mathematical Institute, Universiteit Leiden, P.O. Box 9512, 2300 RA Leiden, The Netherlands, {\tt{hhupkes@math.leidenuniv.nl}}}

\author{J.\@ Douglas Wright}
\address{Department of Mathematics, Drexel University, Korman Center, 33rd \& Market Streets, Philadelphia, PA 19104 USA {\tt{jdw66@drexel.edu}}}

\keywords{FPU, FPUT, dimer, periodic traveling wave, bifurcation, two-dimensional kernel}

\subjclass[2020]{Primary 37K40; Secondary 35C07, 37K50, 37K60}

\date{\today}

\begin{document}

\begin{abstract}
We prove the existence of small-amplitude periodic traveling waves in dimer Fermi--Pasta--Ulam--Tsingou (FPUT) lattices without assumptions of physical symmetry. 
Such lattices are infinite, one-dimensional chains of coupled particles in which the particle masses and/or the potentials of the coupling springs can alternate. 
Previously, periodic traveling waves were constructed in a variety of limiting regimes for the symmetric mass and spring dimers, in which only one kind of material data alternates. 
The new results discussed here remove the symmetry assumptions by exploiting the gradient structure and translation invariance of the traveling wave problem.  
Together, these features eliminate certain solvability conditions that symmetry would otherwise manage and facilitate a bifurcation argument involving a two-dimensional kernel.
\end{abstract}

\maketitle

\section{Introduction}

\subsection{The traveling wave problem}
A dimer Fermi--Pasta--Ulam--Tsingou (FPUT) lattice is a chain of infinitely many particles coupled to their nearest neighbors by springs, with motion restricted to the horizontal direction, and with at least one of the following material heterogeneities: either the particle masses alternate, or the spring potentials alternate, or both alternate.
A dimer with alternating particles and identical springs is called a mass dimer; one with alternating springs and identical masses is a spring dimer.
See Figure \ref{fig: mass and spring dimers}.
Dimers are among the simplest nontrivial generalizations of the classical monatomic FPUT lattice, in which all of the particles have the same mass and all of the springs have the same potential \cite{fput-original, dauxois, vainchtein-survey, pankov}
These lattices, and their many variants and generalizations, are prototypical models of wave dynamics in granular media \cite{cpkd, chong-kev-book}.

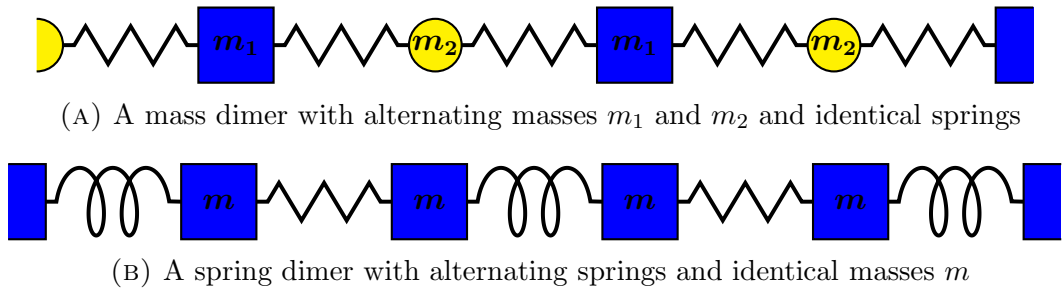
\begin{figure}[h]

\begin{subfigure}{\textwidth}

\centering

\begin{tikzpicture}[scale=1]

\def\arith#1#2#3{
#3*#1+#3*#2
};

\def\squaremass#1#2#3{
\draw[fill=blue, opacity=.4, draw opacity=1,thick]
(#1,#2) rectangle (#1+2*#2,-#2);
\node at (#1+#2,0){$\boldsymbol{#3}$};
};

\def\circlemass#1#2#3{
\draw[fill=yellow,thick] (#1,0)node{$\boldsymbol{#3}$} circle(#2);
};

\def\spring#1#2#3#4{
\draw[line width = 1.5pt] (#1,0)
--(#1 + #2,0)
--(#1 + #2 + #3, #4)
--(#1 + #2 + 3*#3,-#4)
--(#1 + #2 + 5*#3,#4)
--(#1 + #2 + 7*#3,-#4)
--(#1 + #2 + 9*#3,#4)
--(#1 + #2 + 10*#3,0)
--(#1 + 2*#2 + 10*#3,0);
};

\def\d{.15}; 
\def\mh{1}; 
\def\sh{.25}; 
\def\swl{.15}; 
\def\r{.35}; 
\def \sl#1{\arith{2*\d}{10*\swl}{#1}}; 

\def\labeldown#1#2{
\draw[densely dotted,thick] (#1,-\mh/2-\d)--(#1,-\mh/2-\d-.75)node[below]{#2}
};

\def\brace#1#2{
\draw[decoration={brace, amplitude = 10pt,mirror},decorate,thick] 
(#1,-\mh/2-\d)--node[midway,below=7pt]{#2}(#1+\sl{1}+\mh/2+\r,-\mh/2-\d)
}


\fill[yellow] (-\r,\r) arc(90:-90:\r)--cycle;
\draw[thick] (-\r,\r) arc(90:-90:\r);

\spring{0}{\d}{\swl}{\sh};


\squaremass{\sl{1}}{\mh/2}{m_1};


\spring{\sl{1}+\mh}{\d}{\swl}{\sh};


\circlemass{\sl{2}+\mh+\r}{\r}{m_2};


\spring{\sl{2}+\mh+2*\r}{\d}{\swl}{\sh};


\squaremass{\sl{3}+\mh+2*\r}{\mh/2}{m_1};


\spring{\sl{3}+2*\mh+2*\r}{\d}{\swl}{\sh};


\circlemass{\sl{4}+2*\mh+3*\r}{\r}{m_2};


\spring{\sl{4} + 2*\mh+4*\r}{\d}{\swl}{\sh};


\fill[blue,opacity=.4]
(\sl{5}+2*\mh+4*\r,\mh/2) rectangle (\sl{5}+2*\mh+4*\r+\mh/2,-\mh/2);

\draw[thick] (\sl{5}+2*\mh++4*\r+\mh/2,\mh/2)
--(\sl{5}+2*\mh+4*\r,\mh/2)
--(\sl{5}+2*\mh+4*\r,-\mh/2)
--(\sl{5}+2*\mh+4*\r+\mh/2,-\mh/2);


%

\end{tikzpicture}
\caption{A mass dimer with alternating masses $m_1$ and $m_2$ and identical springs}
\label{fig: mass dimer}

\end{subfigure}\\[10pt]

\begin{subfigure}{\textwidth}

\centering

\begin{tikzpicture}

\def\arith#1#2#3{
#3*#1+#3*#2
};

\def\squaremass#1#2#3{
\draw[fill=blue, opacity=.4, draw opacity=1,thick]
(#1,#2) rectangle (#1+2*#2,-#2);
\node at (#1+#2,0){$\boldsymbol{#3}$};
};

\def\circlemass#1#2#3{
\draw[fill=yellow,thick] (#1,0)node{$\boldsymbol{#3}$} circle(#2);
};

\def\spring#1#2#3#4{
\draw[line width = 1.5pt] (#1,0)
--(#1 + #2,0)
--(#1 + #2 + #3, #4)
--(#1 + #2 + 3*#3,-#4)
--(#1 + #2 + 5*#3,#4)
--(#1 + #2 + 7*#3,-#4)
--(#1 + #2 + 9*#3,#4)
--(#1 + #2 + 10*#3,0)
--(#1 + 2*#2 + 10*#3,0);
};

\def\d{.15}; 
\def\mh{1}; 
\def\sh{.25}; 
\def\swl{.15}; 
\def\r{.35}; 
\def \sl#1{\arith{2*\d}{10*\swl}{#1}}; 


\def\h{.45}; 
\def\M{.88}; 
\def\N{.3009}; 
\def\E{.2}; 


\def\coil#1{ 
{\N+\M*\t+\E*sin(4*\t*pi r)+#1},
{\h*cos(4*\t*pi r)}
}


\def\NLSPRING#1{
\draw[line width=1.5pt,domain={-.125:1.125},smooth,variable=\t,samples=100]
(#1,0)--plot(\coil{#1+\d})
--(#1+\sl{1},0);
}

\def\labeldown#1#2{
\draw[densely dotted,thick] (#1,-\mh/2-\d)--(#1,-\mh/2-\d-.75)node[below]{#2}
};

\def\brace#1#2{
\draw[decoration={brace, amplitude = 10pt,mirror},decorate,thick] 
(#1,-\mh/2-\d)--node[midway,below=7pt]{#2}(#1+\sl{1}+\mh,-\mh/2-\d)
}


\fill[blue,opacity=.4] 
(-\mh/2,\mh/2)
--(0,\mh/2)
--(0,-\mh/2)
--(-\mh/2,-\mh/2)
--cycle;

\draw[thick]
(-\mh/2,\mh/2)
--(0,\mh/2)
--(0,-\mh/2)
--(-\mh/2,-\mh/2);


\NLSPRING{0};


\squaremass{\sl{1}}{\mh/2}{m};


\spring{\sl{1}+\mh}{\d}{\swl}{\sh};


\squaremass{\sl{2}+\mh}{\mh/2}{m};


\NLSPRING{\sl{2}+2*\mh}:


\squaremass{\sl{3}+2*\mh}{\mh/2}{m};


\spring{\sl{3}+3*\mh}{\d}{\swl}{\sh};


\squaremass{\sl{4}+3*\mh}{\mh/2}{m};


\NLSPRING{\sl{4}+4*\mh};


\fill[blue,opacity=.4]
(\sl{5}+4*\mh,\mh/2) rectangle (\sl{5}+4*\mh+\mh/2,-\mh/2);

\draw[thick] (\sl{5}+4*\mh+\mh/2,\mh/2)
--(\sl{5}+4*\mh,\mh/2)
--(\sl{5}+4*\mh,-\mh/2)
--(\sl{5}+4*\mh+\mh/2,-\mh/2);


%

\end{tikzpicture}
\caption{A spring dimer with alternating springs and identical masses $m$}
\label{fig: spring dimer}

\end{subfigure}

\caption{The symmetric mass and spring dimers}
\label{fig: mass and spring dimers}

\end{figure}

Figures \ref{fig: mass dimer} and \ref{fig: spring dimer} suggest that the mass and spring dimers possess certain physical ``symmetries'' that a ``general'' dimer, in which both masses and springs alternate, does not.
We sketch such a general dimer, along with some notation for future use, in Figure \ref{fig: general dimer}.
Physically, the mass dimer is the same when ``reflected'' about a particle, as is the spring dimer when reflected about a spring.
Such symmetries manifest themselves mathematically in a variety of useful ways, as we elaborate in Section \ref{sec: with symmetry}, and these manifestations have been key to multiple prior analyses of mass and spring dimer dynamics.
Here we consider the general dimer, and one of the main novelties of our techniques is that we do not rely at all on symmetry.

\begin{figure}

\begin{tikzpicture}

\def\arith#1#2#3{
#3*#1+#3*#2
};

\def\squaremass#1#2#3{
\draw[fill=blue, opacity=.4, draw opacity=1,thick]
(#1,#2) rectangle (#1+2*#2,-#2);
\node at (#1+#2,0){$\boldsymbol{#3}$};
};

\def\circlemass#1#2#3{
\draw[fill=yellow,thick] (#1,0)node{$\boldsymbol{#3}$} circle(#2);
};

\def\spring#1#2#3#4{
\draw[line width = 1.5pt] (#1,0)
--(#1 + #2,0)
--(#1 + #2 + #3, #4)
--(#1 + #2 + 3*#3,-#4)
--(#1 + #2 + 5*#3,#4)
--(#1 + #2 + 7*#3,-#4)
--(#1 + #2 + 9*#3,#4)
--(#1 + #2 + 10*#3,0)
--(#1 + 2*#2 + 10*#3,0);
};

\def\d{.15}; 
\def\mh{1}; 
\def\sh{.25}; 
\def\swl{.15}; 
\def\r{.35}; 
\def \sl#1{\arith{2*\d}{10*\swl}{#1}}; 


\def\h{.45}; 
\def\M{.88}; 
\def\N{.3009}; 
\def\E{.2}; 


\def\coil#1{ 
{\N+\M*\t+\E*sin(4*\t*pi r)+#1},
{\h*cos(4*\t*pi r)}
}


\def\NLSPRING#1{
\draw[line width=1.5pt,domain={-.125:1.125},smooth,variable=\t,samples=100]
(#1,0)--plot(\coil{#1+\d})
--(#1+\sl{1},0);
}

\def\labeldown#1#2{
\draw[densely dotted,thick] (#1,-\mh/2-\d)--(#1,-\mh/2-\d-.75)node[below]{#2}
};

\def\brace#1#2{
\draw[decoration={brace, amplitude = 10pt,mirror},decorate,thick] 
(#1,-\mh/2-\d)--node[midway,below=7pt]{#2}(#1+\sl{1}+\mh/2+\r,-\mh/2-\d)
}


\fill[yellow] (-\r,\r) arc(90:-90:\r)--cycle;
\draw[thick] (-\r,\r) arc(90:-90:\r);


\NLSPRING{0};


\squaremass{\sl{1}}{\mh/2}{m_1};


\spring{\sl{1}+\mh}{\d}{\swl}{\sh};


\circlemass{\sl{2}+\mh+\r}{\r}{m_2};


\NLSPRING{\sl{2}+\mh+2*\r}:


\squaremass{\sl{3}+\mh+2*\r}{\mh/2}{m_1};


\spring{\sl{3}+2*\mh+2*\r}{\d}{\swl}{\sh};


\circlemass{\sl{4}+2*\mh+3*\r}{\r}{m_2};


\NLSPRING{\sl{4}+2*\mh+4*\r};


\fill[blue,opacity=.4]
(\sl{5}+2*\mh+4*\r,\mh/2) rectangle (\sl{5}+2*\mh+4*\r+\mh/2,-\mh/2);

\draw[thick] (\sl{5}+2*\mh+4*\r+\mh/2,\mh/2)
--(\sl{5}+2*\mh+4*\r,\mh/2)
--(\sl{5}+2*\mh+4*\r,-\mh/2)
--(\sl{5}+2*\mh+4*\r+\mh/2,-\mh/2);


\labeldown{\sl{1}+\mh/2}{$u_{j-1}$};
\labeldown{\sl{2}+\mh+\r}{$u_j$};
\labeldown{\sl{3}+3*\mh/2+2*\r}{$u_{j+1}$};
\labeldown{\sl{4}+2*\mh + 3*\r}{$u_{j+2}$};

\brace{\sl{1}+\mh/2}{$r_{j-1}$};
\brace{\sl{2}+\mh+\r}{$r_j$};
\brace{\sl{3}+3*\mh/2+2*\r}{$r_{j+1}$};

\end{tikzpicture}
\caption{A general dimer with alternating masses and springs}
\label{fig: general dimer}

\end{figure}
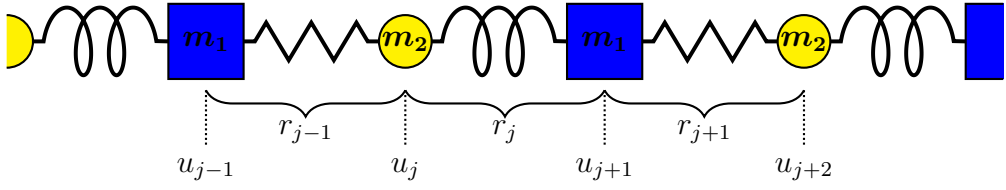

Specifically, we construct nontrivial periodic traveling waves for general dimers with wave speed close to a certain critical threshold called the lattice's ``speed of sound''---that is, periodic traveling waves in the long wave limit.
We state our precise results below in Theorem \ref{thm: main} and discuss the connection of these traveling waves to the long wave problem in dimers, and related problems, in Section \ref{sec: context}.

Here is our problem.
Index the particles by integers $j \in \Z$ and let $u_j$ be the displacement of the $j$th particle from its equilibrium position, let $m_j$ be the mass of the $j$th particle, and let $\V_j$ be the potential of the spring connecting the $j$th and $(j+1)$st particles.
To ensure a dimer structure, we assume
\[
m_{j+2} = m_j
\quadword{and}
\V_{j+2} = \V_j
\]
for all $j$.
More precisely, after nondimensionalization \cite[Sec.\@ 1.3, App.\@ F.5]{faver-dissertation}, we take
\begin{equation}\label{eqn: mat data dimer}
m_j
= \begin{cases}
1, \ j \text{ is odd} \\
m, \ j \text{ is even}
\end{cases}
\quadword{and}
\V_j'(r)
= \begin{cases}
r + r^2 + \O(r^3) \\
\kappa{r} + \beta{r}^2 + \O(r^3).
\end{cases}
\end{equation}
The minimum regularity required for our proofs is that $\V_j \in \Cal^7(\R)$ for precise technical reasons detailed in Appendix \ref{app: composition operator calculus}, but for broader applications to FPUT traveling wave problems, we may as well assume $\V_j \in \Cal^{\infty}(\R)$.
Our methods require that the heterogeneity appear at the linear level, so we will always assume
\begin{equation}\label{eqn: mat data hetero lin}
\frac{1}{m} > 1
\quadword{or}
\kappa > 1.
\end{equation}

We emphasize that $m$ and $\kappa$, and indeed all of the material data, are fixed throughout our analysis and that virtually all operators, quantities, and thresholds depend on at least these quantities; we do not indicate such dependence in our notation.
Additionally, beyond the regularity requirements on $\V_j$, the nonlinear terms, even the quadratic ones, play no important role.

Newton's second law requires that the displacements $u_j$ satisfy
\begin{equation}\label{eqn: eqn of motion}
m_j\ddot{u}_j
= \V_j'(u_{j+1}-u_j)-\V_{j-1}'(u_j-u_{j-1}).
\end{equation}
Under the traveling wave ansatz
\begin{equation}\label{eqn: position tw ansatz}
u_j(t)
= \begin{cases}
p_1(j-ct), \ j \text{ is odd} \\
p_2(j-ct), \ j \text{ is even},
\end{cases}
\qquad
\pb := \begin{pmatrix*}
p_1 \\ 
p_2
\end{pmatrix*},
\end{equation}
these equations of motion become the advance-delay problem
\begin{equation}\label{eqn: position tw system}
\begin{cases}
c^2p_1'' = \V_1'(S^1p_2-p_1) - \V_2'(p_1-S^{-1}p_2) \\
c^2mp_2'' = \V_2'(S^1p_1-p_2)-\V_1'(p_2-S^{-1}p_1).
\end{cases}
\end{equation}
Here, for $\theta \in \R$, $S^{\theta}$ is the shift operator
\[
(S^{\theta}p)(X)
:= p(X+\theta).
\]

The following is our main result for \eqref{eqn: position tw system}.
We use the notation for periodic Sobolev spaces developed in Appendix \ref{app: per sob space}.

\begin{theorem}\label{thm: main}
Suppose that the lattice's material data $m_j$ and $\V_j$ satisfy the dimer condition \eqref{eqn: mat data dimer} and the linear heterogeneity condition \eqref{eqn: mat data hetero lin}.
Let the wave speed $c$ in the ansatz \eqref{eqn: position tw ansatz} satisfy $|c| > c_{\star}$, where the lattice's ``speed of sound'' $c_{\star}$ is defined in \eqref{eqn: speed of sound}.
Then there exists $a_c > 0$ such that for $|a| \le a_c$, there is a traveling wave solution $\pb_c^a$ to \eqref{eqn: position tw system} of the form
\[
\pb_c^a(X)
= a\nub_1^c(\omega_c^aX)+a^2\psib_c^a(\omega_c^aX).
\]
The smooth, $2\pi$-periodic profile terms $\nub_c^a$ and $\psib_c^a$ and the frequency $\omega_c^a \in \R$ have the following properties.

\begin{enumerate}[label={\bf(\roman*)}]

\item
The leading order term $\nub_1^c$ has an exact formula given below by \eqref{eqn: nub1}.

\item
The remainder term $\psib_c^a$ is orthogonal to $\nub_1^c$ and uniformly bounded in $a$ in the sense that
\[
\ip{\nub_1^c}{\psib_c^a}_{L_{\per}^2} = 0
\quadword{and}
\sup_{|a| \le a_c} \norm{\psib_c^a}_{H_{\per}^r} < \infty, \ r \ge 0.
\]

\item
The frequency $\omega_c^a$ has the expansion 
\[
\omega_c^a
= \omega_c + a\xi_c^a,
\]
where $\omega_c > 0$ is the lattice's ``critical frequency,'' as developed in Theorem \ref{thm: eigenvalues}, and 
\[
\sup_{|a| \le a_c} |\xi_c^a| 
< \infty.
\]
\end{enumerate}
\end{theorem}

We provide effectively {\it{four}} proofs of this theorem.
Specifically, Sections \ref{sec: gradient formulation} and \ref{sec: LC} give proofs inspired by the techniques of Wright and Scheel \cite{wright-scheel} for constructing asymmetric solitary wave solutions to a system of coupled KdV equations; the lack of symmetry in their problem manifests itself mathematically in a complication very close to ours, as we discuss below in Remark \ref{rem: WS}.
Section \ref{sec: with symmetry} provides fresh perspectives on the role of symmetry when it {\it{is}} present in the case of mass and spring dimers.
And Section \ref{sec: quant} develops precise quantitative estimates for the solution components from Theorem \ref{thm: main} relative to the wave speed $c$.
In particular, the exact, but general, hypotheses of Theorem \ref{thm: quant} subsume all prior constructions of dimer periodics into one quantitative result.
For brevity, Theorem \ref{thm: main} above does not contain our results in the long wave scaling, and we discuss them instead in Section \ref{sec: app of abstract quant}.

\begin{remark}
The majority of traveling wave results (periodic or not) for lattices are stated in relative displacement coordinates: $r_j = u_{j+1}-u_j$.
See Figure \ref{fig: general dimer}.
We find it more convenient to work in the original equilibrium displacement coordinates $u_j$, from which relative displacement results can easily be obtained (though the converse is not necessarily true).
\end{remark}

We finally state the actual periodic traveling wave problem that we solve to prove Theorem \ref{thm: main}; the following notation was not strictly necessary above, but all of our subsequent work depends on it.
Since we are interested in periodic traveling waves, we adjust the original traveling wave ansatz \eqref{eqn: position tw ansatz} by decoupling the profile and frequency via the additional ansatz
\[
\pb(X) = \phib(\omega{X}),
\qquad
\phib := \begin{pmatrix*}
\phi_1 \\
\phi_2
\end{pmatrix*}.
\]
The new profiles $\phi_1$ and $\phi_2$ are now $2\pi$-periodic and $\omega \in \R$.
The traveling wave equations \eqref{eqn: position tw system} then become
\begin{equation}\label{eqn: periodic tw system}
\begin{cases}
c^2\omega^2\phi_1'' = \V_1'(S^{\omega}\phi_2-\phi_1) - \V_2'(\phi_1-S^{-\omega}\phi_2) \\
mc^2\omega^2\phi_2'' = \V_2'(S^{\omega}\phi_1-\phi_2) - \V_1'(\phi_2-S^{-\omega}\phi_1).
\end{cases}
\end{equation}
We compress \eqref{eqn: periodic tw system} in the form
\begin{equation}\label{eqn: periodic tw system Phib}
\Phib_c(\phib,\omega)
= 0,
\end{equation}
where
\begin{equation}\label{eqn: Phib}
\Phib_c(\phib,\omega)
:= c^2\omega^2M\phib'' 
+ \begin{pmatrix*}
\V_2'(\phi_1-S^{-\omega}\phi_2) - \V_1'(S^{\omega}\phi_2-\phi_1) \\
\V_1'(\phi_2-S^{-\omega}\phi_1) - \V_2'(S^{\omega}\phi_1-\phi_2)
\end{pmatrix*}
\end{equation}
and
\begin{equation}\label{eqn: M}
M 
:= \begin{bmatrix*}
1 &0 \\
0 &m
\end{bmatrix*}.
\end{equation}

The primary challenge that we confront is that the linearization $D_{\phib}\Phib_c(0,\omega_c)$ has a three-dimensional kernel and cokernel.
Translation invariance allows us to reduce both dimensions to two, but, in the absence of symmetry, we cannot get below that.
We now discuss the broader relevance of this periodic traveling wave problem and, in the process, why this dimension counting is so important.

\subsection{Motivation, context, and connections other FPUT traveling wave problems}\label{sec: context}
Our primary motivation in constructing these particular periodics is the long wave problem for dimers.
This limit looks for traveling waves whose profiles are close to a suitably scaled $\sech^2$-type solution of a KdV equation that acts as the ``continuum limit'' for the lattice.
(More precisely, one posits $\pb(X) = \ep^2\mathbf{h}(\ep{X})$ and $c^2 = c_{\star}^2 + \ep^2$, with $c_{\star}$ given by \eqref{eqn: speed of sound}. 
We discuss this further in Section \ref{sec: app of abstract quant}.)
The motivation for this ansatz is that in a ``polyatomic'' FPUT lattice, for which the material data repeats with some arbitrary period, long wave-scaled solutions to certain KdV equations (whose coefficients depend on the lattice's material data) are very good approximations to solutions to the equations of motion over very long time scales \cite{schneider-wayne, gmwz, chirilus-bruckner-etal}.

Faver and Wright constructed long wave solutions for the mass dimer \cite{faver-wright} and Faver treated the spring dimer \cite{faver-spring-dimer}.
Faver and Hupkes produced a different development of mass and spring dimer nanopterons via spatial dynamics in \cite{faver-hupkes-spatial-dynamics} and obtained results for equilibrium displacement coordinates as we do; Deng and Sun \cite{deng-sun-fput-md} performed a related spatial dynamics analysis to yield similar results.
These dimer traveling waves were not solitary waves, as Friesecke and Pego found for the monatomic lattice \cite{friesecke-pego1}, but rather nanopterons \cite{boyd}: the superposition of a leading-order localized (here, $\sech^2$-type) term, a higher-order localized remainder, and a high frequency periodic ``ripple'' of amplitude small beyond all algebraic orders of the long wave parameter.
Both constructions relied on lattice symmetries in two critical steps to adapt functional analytic techniques from Beale's work on capillary gravity water waves \cite{beale} and its later deployment by Amick and Toland \cite{amick-toland} for a model fourth-order KdV equation.

First, as mentioned above, the periodics in \cite{faver-wright, faver-spring-dimer} were constructed with a modified ``bifurcation from a simple eigenvalue'' argument in the style of Crandall and Rabinowitz and Zeidler \cite{crandall-rabinowitz, zeidler}.
We adapt further this bifurcation analysis in our arguments, and our preferred reference is \cite[Thm.\@ 1.5.1]{kielhofer}.
Symmetry permitted the restriction of the traveling wave problem $\Phib_c(\phib,\omega) = 0$ from \eqref{eqn: periodic tw system Phib} and \eqref{eqn: Phib} to function spaces on which the linearization $D_{\phib}\Phib_c(0,\omega)$ at $\phib = 0$ and $\omega = \omega_c$, with $\omega_c$ as the ``critical frequency'' from Theorem \ref{thm: eigenvalues}, had a one-dimensional kernel and cokernel.
This was the key to the modified bifurcation from a simple eigenvalue argument.

Up to a useful linear change of coordinates that diagonalized the traveling wave problem and the long wave scaling, the long wave periodics in \cite{faver-wright, faver-spring-dimer} have the same structure as ours from Theorem \ref{thm: main}.
However, the main technical accomplishment of our results here is that we manage a {\it{two}}-dimensional kernel and cokernel in the absence of symmetry via other inherent properties of the lattice---namely, the special ``orthogonality condition'' that $\ip{\Phib_c(\phib,\omega)}{\phib'}_{L_{\per}^2} = 0$, proved in Corollary \ref{cor: shift} and Lemma \ref{lem: deriv ortho}.
While there certainly exist other results on bifurcations with two-dimensional kernels, their hypotheses are inappropriate for our problem.
For example, \cite{kromer-healey-kielhofer} and \cite{liu-shi-wang} assume certain ``nondegeneracy'' conditions on what for us would be the second derivative $D_{\phib\phib}^2\Phib_c(0,\omega_c)$, over which we expect to have no control (beyond its existence), while \cite{akers-ambrose-sulon} and \cite{baldi-toland} assume some (non)resonance conditions among their critical frequencies.

The second use of symmetry in the full nanopteron constructions of \cite{faver-wright, faver-spring-dimer} was somewhat subtler and involved the actual need for periodics in the first place.
The obstacle was that attempting to solve the traveling wave problem \eqref{eqn: position tw system} by a merely localized perturbation from the $\sech^2$-type continuum limit resulted in an overdetermined system with two unknowns (the two components of the localized perturbation) but four equations.
These are the expected two components from \eqref{eqn: position tw system} and a surprising ``solvability condition'': the vanishing of the Fourier transform of a certain related operator at $\pm\omega_c$.
Symmetry ensured that the vanishing at $\omega_c$ implied the vanishing at $-\omega_c$, reducing the overdetermined problem to only three equations, which were managed by adding a third variable via the periodic amplitude---which is exactly why we seek periodics in Theorem \ref{thm: main} that are parametrized in amplitude.
While the full nanopteron problem in the general dimer without symmetry remains challenging and beyond the scope of our work here to address, the periodic solutions constructed here will be a fundamental component of the nanopteron ansatz for the general dimer's long wave problem.

Periodic traveling waves for lattices have been constructed in several other ``material limit'' regimes in addition to the long wave limit.
These include the small mass limit for mass dimers by Hoffman and Wright \cite[Thm.\@ 5.1]{hoffman-wright}, the equal mass limit for mass dimers by Faver and Hupkes \cite[Prop.\@ 3.3]{faver-hupkes}, and the small mass limit for the mass-in-mass variant of monatomic FPUT by Faver \cite[Thm.\@ 2]{faver-mim-nanopteron}.
Each of these limits views the heterogeneous lattice as a small material perturbation of a monatomic FPUT lattice, and the nanopteron is a small nonlocal perturbation of a monatomic solitary wave \cite{friesecke-pego1, friesecke-wattis}.
While each limit has a nontrivially different solvability condition that makes the traveling wave problem overdetermined, all of the periodic constructions are fundamentally alike and can be deduced from our Theorem \ref{thm: quant}.

These are not the only methods for producing periodic traveling waves for FPUT, and we give a brief, selected overview of others here.
Friesecke and Mikikits-Leitner \cite{fml} adapted the perturbative approach for monatomic solitary waves from \cite{friesecke-pego1} to prove the existence of long wave periodics in the monatomic lattice that were small perturbations of a KdV cnoidal profile.
Pankov constructed periodics in the monatomic lattice using variational methods \cite{pankov}, as did Qin for mass dimers \cite{qin} with spring force given by the FPUT $\beta$-model, i.e., roughly of the form $\V'(r) = r+\O(r^3)$.
Betti and Pelinovsky used an implicit function theorem argument to produce periodics in mass dimers with Hertzian spring potentials \cite{betti-pelinovsky}, and we note with interest that their proofs also relied on symmetry to reduce the dimension of a key linearization's kernel.
Finally, we mention that Lombardi's spatial dynamics method for nanopterons under very general hypotheses includes the full development of periodics from that point of view \cite{lombardi}, with the more stringent requirement that the spring potentials be real analytic.

\subsection{Notation}
We summarize several aspects of notation that we will use without further comment.

\begin{enumerate}[label=$\bullet$]

\item
If $\X$ is a vector space, then $\ind_{\X}$ is the identity operator on $\X$.

\item
If $\X$, $\Y$, and $\Zcal$ are normed spaces and $f \colon \U \subseteq \X \times \Y \to \Zcal$ is differentiable at some $(x_0,y_0) \in \U$, then we denote its partial derivative at $(x_0,y_0)$ with respect to $x$ by $D_xf(x_0,y_0)$.
Likewise, $D_yf(x_0,y_0)$ is the partial derivative with respect to $y$.
We reserve the notation $f' = \partial_xf$ for a function $f \colon I \subseteq \R \to \R$.

\item
If $\X$ and $\Y$ are sets and $f \colon \U \subseteq \X \to \Y$ is a function, then for any $\U_0 \subseteq \U$ we denote by $\restr{f}{\U_0}$ the restriction of $f$ to $\U_0$.

\item
If $\X$ and $\Y$ are Hilbert spaces and $\T \colon \X \to \Y$ is a bounded linear operator, then $\T^* \colon \Y \to \X$ is the adjoint of $\T$.

\item
If $\X$ is a normed space, $x \in \X$, and $r > 0$, then $\B(x;r)$ is the open ball
\[
\B(x;r) 
:= \set{y \in \X}{\norm{x-y}_{\X} < r}.
\]

\item
If $\X$ and $\Y$ are normed spaces, then $\b(\X,\Y)$ is the space of bounded linear operators from $\X$ to $\Y$ with operator norm $\norm{\T}_{\X \to \Y}$.
\end{enumerate}

\section{Linear Analysis}
We assume familiarity here with the notation and conventions of Appendix \ref{app: per sob space} on periodic Sobolev spaces and Fourier coefficients.
Briefly, $\hat{\phib}(k)$ is the $k$th Fourier coefficient of $\phib \in L_{\per}^2(\R^2)$, and $\ip{\cdot}{\cdot}$ is the $L_{\per}^2$-inner product (we no longer keep the subscript here).

Our bifurcation analysis naturally hinges on a careful understanding of the linearization
\begin{equation}\label{eqn: Lc}
\L_c[\omega]
:= D_{\phib}\Phib_c(0,\omega)
\end{equation}
of the problem $\Phib_c(\phib,\omega) = 0$ at $\phib = 0$, where $\Phib_c$ was defined in \eqref{eqn: Phib}.
Using that definition of $\Phib_c$ and recalling from the hypotheses \eqref{eqn: mat data dimer} that the spring potentials satisfy 
\[
\V_1'(r) = r + \O(r^2)
\quadword{and}
\V_2'(r) = \kappa{r} + \O(r^2),
\]
we have
\[
\L_c[\omega]\phib
= c^2\omega^2M\phib'' + \D[\omega]\phib,
\]
where
\begin{equation}\label{eqn: Domega}
\D[\omega]
:= \begin{bmatrix*}
(1+\kappa) &-(S^{\omega}+\kappa{S}^{-\omega}) \\
-(\kappa{S}^{\omega}+S^{-\omega}) &(1+\kappa).
\end{bmatrix*}.
\end{equation}

We follow the strategy of the existing bifurcation arguments \cite[App.\@ C]{faver-wright}, \cite[Sec.\@ 3]{faver-spring-dimer}, \cite[Sec.\@ 5]{hoffman-wright}, \cite[Sec.\@ 3]{faver-hupkes-equal-mass}, \cite[Sec.\@ 3]{faver-mim-nanopteron} and begin by considering the kernel of $\L_c[\omega]$.
We have $\L_c[\omega]\phib = 0$ if and only if 
\begin{equation}\label{eqn: Fourier side kernel}
\tL_c(\omega{k})\hat{\phib}(k) 
= 0
\end{equation}
for all $k \in \Z$, where 
\begin{equation}\label{eqn: tL}
\tL_c(K) 
:= -c^2K^2M + \tD(K)
\end{equation}
and
\begin{equation}\label{eqn: tD}
\tD(K)
:= \begin{bmatrix*}
(1+\kappa) &-(e^{iK}+\kappa{e}^{-iK}) \\
-(\kappa{e}^{iK}+e^{-iK}) &(1+\kappa)
\end{bmatrix*}.
\end{equation}
Then \eqref{eqn: Fourier side kernel} is equivalent to
\[
c^2(\omega{k})^2\hat{\phib}(k)
= M^{-1}\tD(\omega{k})\hat{\phib}(k),
\]
and so, if $\hat{\phib}(k) \ne 0$, then $c^2(\omega{k})^2$ must be an eigenvalue of $M^{-1}\tD(\omega{k})$.
Any eigenvalue $\lambda$ of $M^{-1}\tD(K)$ must satisfy the characteristic equation
\[
\lambda^2 - (1+w)(1+\kappa)\lambda + 2\kappa{w}(2\cos^2(K)-1)
= 0,
\]
and so the eigenvalues are
\begin{equation}\label{eqn: tlambda}
\tlambda_{\pm}(K)
:= \frac{(1+\kappa)(1+w)}{2} \pm \frac{\trho(K)}{2},
\end{equation}
where
\begin{equation}\label{eqn: trho}
\trho(K) := \sqrt{(1+w)^2(1-\kappa)^2 + 4\kappa((1-w)^2+4w\cos^2(K))}
\quadword{and}
w := \frac{1}{m}.
\end{equation}

The following is proved in \cite[Prop.\@ 2.2.1]{faver-dissertation} about these eigenvalues.

\begin{theorem}\label{thm: eigenvalues}
Suppose that at least one of the inequalities $\kappa > 1$ or $w > 1$ holds and define the ``speed of sound'' to be
\begin{equation}\label{eqn: speed of sound}
c_{\star}
:= \sqrt{\frac{4\kappa{w}}{(1+\kappa)(1+w)}}.
\end{equation}

\begin{enumerate}[label={\bf(\roman*)}]

\item
If $|c| > c_{\star}$, then $c^2K^2 = \tlambda_-(K)$ if and only if $K = 0$.

\item
If $|c| > c_{\star}$, then there exists $\omega_c > 0$ such that $c^2K^2 = \tlambda_+(K)$ if and only if $K = \pm\omega_c$.

\item
This ``critical frequency'' $\omega_c$ satisfies the estimates
\begin{equation}\label{eqn: omegac-ineq}
\frac{\sqrt{\tlambda_+(\pi/2)}}{c}
\le \omega_c
\le \frac{\sqrt{(1+\kappa)(1+w)}}{c}.
\end{equation}
and
\begin{equation}\label{eqn: omegac transversality}
\inf_{|c| > c_{\star}} 2c^2\omega_c-\tlambda_+'(\omega_c) 
> 0.
\end{equation}
\end{enumerate}
\end{theorem}

We will ultimately run several bifurcation arguments from the critical frequency $\omega_c$.
To do that, we need a good understanding of the kernel and cokernel of $\L_c[\omega_c]$.
We carefully compute the following in Appendix \ref{app: ker coker Lc-omegac}.

\begin{corollary}\label{cor: eigenfunctions}
The kernels of both 
\[
\L_c[\omega_c] \colon H_{\per}^2(\R^2) \to L_{\per}^2(\R^2)
\quadword{and}
\L_c[\omega_c]^* \colon L_{\per}^2(\R^2) \to H_{\per}^2(\R^2)
\]
are spanned by the orthonormal vectors $\nub_0$, $\nub_1^c$, and $\nub_2^c$ defined by
\begin{equation}\label{eqn: nub0}
\nub_0 
:= \frac{1}{\sqrt{2}}\begin{pmatrix*} 
1 \\ 
1 
\end{pmatrix*},
\end{equation}
\begin{equation}\label{eqn: nub1}
\nub_1^c(x) 
:= \frac{e^{-ix}}{\Nu_c}\begin{pmatrix*}
e^{-i\omega_c}+\kappa{e}^{i\omega_c} \\
1+\kappa-c^2\omega_c^2
\end{pmatrix*}
+ \frac{e^{ix}}{\Nu_c}\begin{pmatrix*}
e^{i\omega_c}+\kappa{e}^{-i\omega_c} \\
1+\kappa-c^2\omega_c^2
\end{pmatrix*}
\end{equation}
and
\begin{equation}\label{eqn: nub2}
\nub_2^c(x)
:= \frac{e^{-ix}}{\Nu_c}\begin{pmatrix*}
i(e^{-i\omega_c}+\kappa{e}^{i\omega_c}) \\
i(1+\kappa-c^2\omega_c^2)
\end{pmatrix*}
+ \frac{e^{ix}}{\Nu_c}\begin{pmatrix*}
-i(e^{i\omega_c}+\kappa{e}^{-i\omega_c}) \\
-i(1+\kappa-c^2\omega_c^2)
\end{pmatrix*},
\end{equation}
where
\begin{equation}\label{eqn: Nu-c}
\Nu_c
:= \sqrt{2}\left(\big[(1-\kappa)^2+4\kappa\cos(\omega_c)\big]^2 + \left[\frac{(1+\kappa)(1-w)+\trho(\omega_c)}{2}\right]^2\right)^{1/2}.
\end{equation}
The eigenfunctions $\nub_1^c$ and $\nub_2^c$ satisfy the derivative identities
\begin{equation}\label{eqn: nub derivatives}
\partial_x\nub_1^c = -\nub_2^c
\quadword{and}
\partial_x\nub_2^c = \nub_1^c
\end{equation}
and the shift identity
\begin{equation}\label{eqn: nub shifts}
\nub_2^c
= S^{-\pi/2}\nub_1^c.
\end{equation}
A function $\phib \in L_{\per}^2(\R^2)$ satisfies
\begin{equation}\label{eqn: nub ortho equiv}
\ip{\phib}{\nub_1^c} = \ip{\phib}{\nub_2^c} = 0
\quadword{if and only if}
\hat{\phib}(1)\cdot\hat{\nub_1^c}(1) = 0.
\end{equation}
\end{corollary}

Thus the kernel and the cokernel of the linearization of the traveling wave problem \eqref{eqn: periodic tw system Phib} are ostensibly three-dimensional.
Translation invariance (Corollary \ref{cor: trans invar}) will allow us to rule out $\nub_0$, and so we are down to two dimensions.

We also need to understand the interaction of the mixed partial derivative $\L_c'[\omega_c] := D_{\phib\omega}\Phib_c(0,\omega_c)$ with the eigenfunction $\nub_1$.
It follows from Appendix \ref{app: shift operator calculus}, specifically the identity \eqref{eqn: shift derivative}, that $\L_c'[\omega_c]$ is the Fourier multiplier given by 
\begin{equation}\label{eqn: Lc-prime}
\hat{\L_c'[\omega_c]\phib}(k)
= k\tL_c'(\omega_ck)\hat{\phib}(k)
\end{equation}
with $\tL_c'$ as the componentwise derivative of the matrix $\tL_c$ from \eqref{eqn: tL}.
We prove the following estimate in Appendix \ref{app: transversality}.
This is the direct analogue of the classical Crandall--Rabinowitz--Zeidler transversality condition \cite[Eqn.\@ (I.5.3)]{kielhofer} for our approach.

\begin{corollary}\label{cor: transversality}
$\inf_{|c| > c_{\star}} |\ip{\L_c'[\omega_c]\nub_1^c}{\nub_1^c}| > 0$.
\end{corollary}

Last, we will need the following coercive estimate on $\L_c[\omega_c]$, proved in Appendix \ref{app: coercive}.

\begin{corollary}\label{cor: coercive}
There is $C > 0$ such that the following holds for all $c$ with $|c| > c_{\star}$ and all $r \ge 0$.
If $\L_c[\omega_c]\psib = \etab$ for $\psib \in H_{\per}^{r+2}(\R^2)$ and $\etab \in H_{\per}^r(\R^2)$ with
\[
\ip{\psib}{\nub_0} = \ip{\psib}{\nub_1^c} = \ip{\psib}{\nub_2^c} = 0
\quadword{and}
\ip{\etab}{\nub_0} = \ip{\etab}{\nub_1^c} = \ip{\etab}{\nub_2^c} = 0,
\]
then
\[
\norm{\psib}_{H_{\per}^{r+2}}
\le C\norm{\etab}_{H_{\per}^r}.
\]
\end{corollary}

\section{The Gradient Formulation}\label{sec: gradient formulation}

\subsection{The gradient structure of the traveling wave problem}
We rewrite the traveling wave operator $\Phib_c$ from \eqref{eqn: Phib} as the $L_{\per}^2$-gradient of a certain ``kinetic + potential energy'' functional on $H_{\per}^2(\R^2)$.
This formulation yields transparent proofs of certain properties of $\Phib_c$ from shift invariance, and from these properties follow our first existence proof in Sections \ref{sec: LS ID} and \ref{sec: LS FD}.

First, we need some new notation; all of the consequences below of this notation are straightforward calculations, which we omit.
For $\omega \in \R$, put 
\begin{equation}\label{eqn: Delta-pm}
\Delta_+(\omega) := \begin{bmatrix*}[r]
-1 &S^{\omega} \\
1 &-S^{-\omega}
\end{bmatrix*}
\quadword{and}
\Delta_-(\omega) = \begin{bmatrix*}[r]
1 &-1 \\
-S^{-\omega} &S^{\omega}
\end{bmatrix*}.
\end{equation}
We then have the adjoint relationship
\begin{equation}\label{eqn: Delta pm adjoint}
\ip{\Delta_+(\omega)\phib}{\etab}
= -\ip{\phib}{\Delta_-(\omega)\etab}
\end{equation}
for any $\phib$, $\etab \in L_{\per}^2(\R^2)$.

Next, let
\[
\Vb(\pb)
:= \begin{pmatrix*}
\V_1(p_1) \\
\V_2(p_2)
\end{pmatrix*}
\quadword{and}
\Vb'(\pb)
:= \begin{pmatrix*}
\V_1'(p_1) \\
\V_2'(p_2)
\end{pmatrix*},
\]
where $\V_1$ and $\V_2$ are the spring potentials from \eqref{eqn: mat data dimer}, and $\pb = (p_1,p_2) \in L_{\per}^2(\R^2)$.
For $\vb = (v_1,v_2)$, $\wb = (v_1,v_2) \in \R^2$, define componentwise multiplication as
\[
\vb.\wb
:= \begin{pmatrix*}
v_1v_1 \\
v_2v_2
\end{pmatrix*}.
\]
We then have the derivative formula
\begin{equation}\label{eqn: Vb derivative}
D_{\pb}\Vb(\pb)\grave{\pb}
= \Vb'(\pb).\grave{\pb}
\end{equation}
for any $\grave{\pb} \in L_{\per}^2(\R^2)$.
Define $\oneb(x) := (1,1)$; then since $\vb$, $\wb \in \R^2$ have real entries, the useful identity
\begin{equation}\label{eqn: dot mult oneb}
\ip{\vb.\wb}{\oneb}
= \ip{\vb}{\wb}
\end{equation}
is true.

Last, we have
\[
\Delta_-(\omega)\Vb'(\Delta_+(\omega)\phib)
= \begin{pmatrix*}
\V_1'(S^{\omega}\phi_2-\phi_1)-\V_2'(\phi_1-S^{-\omega}\phi_2) \\
\V_2'(S^{\omega}\phi_1-\phi_2)-\V_1'(\phi_2-S^{-\omega}\phi_1)
\end{pmatrix*}.
\]
Comparing this to the second term in $\Phib_c$ from \eqref{eqn: Phib}, we conclude
\begin{equation}\label{eqn: Phib rewritten}
\Phib_c(\phib,\omega)
= c^2\omega^2M\phib''-\Delta_-(\omega)\Vb'(\Delta_+(\omega)\phib).
\end{equation}
This version of $\Phib_c$ allows us to recognize it as a gradient; similar calculations for the monatomic lattice appear in \cite[Prop.\@ 3.2]{pankov} and for mass dimers with the FPUT $\beta$-model in \cite[Lem.\@ 3.1]{qin}.

\begin{theorem}\label{thm: G}
Let $c \in \R$.
Define
\begin{equation}\label{eqn: T}
\T
\colon H_{\per}^2(\R^2) \times \R \to \R
\colon (\phib,\omega) \mapsto \frac{\omega^2}{2}\ip{M\phib''}{\phib} 
\end{equation}
and, with $\oneb(x) := (1,1)$,
\begin{equation}\label{eqn: P}
\P 
\colon H_{\per}^2(\R^2) \times \R \to \R
\colon (\phib,\omega) \mapsto \ip{\Vb(\Delta_+(\omega)\phib)}{\oneb} 
\end{equation}

Put 
\begin{equation}\label{eqn: Gc}
\G_c 
:= c^2\T + \P.
\end{equation}
Then 
\[
\Phib_c(\phib,\omega) 
= \nabla\G_c(\phib,\omega)
\]
in the sense that 
\begin{equation}\label{eqn: Phib Frechet}
D_{\phib}\G_c(\phib,\omega)\etab
= \ip{\Phib_c(\phib,\omega)}{\etab}
\end{equation}
for all $\phib$, $\etab \in H_{\per}^2(\R^2)$ and $\omega \in \R$.
\end{theorem}

\begin{proof}
The proof is just a careful calculation using the definition of the derivative and the inner product $\ip{\cdot}{\cdot}$ and the various identities stated above.
More precisely, we compute the following.

First, for $\phib$, $\etab \in H_{\per}^2(\R^2)$ and $\omega$, $h \in \R$, we use the definition of $\T$ in \eqref{eqn: T} to compute
\[
\T(\phib + h\etab,\omega)-\T(\phib,\omega)
= h\ip{\omega^2M\phib''}{\etab} + h^2\frac{\omega^2\ip{M\etab''}{\etab}}{2}.
\]
This uses two applications of the integration by parts identity \eqref{eqn: integration by parts} to compute $\ip{\phib''}{\etab} = \ip{\phib}{\etab''}$, the symmetry of $M$, and the assumption that $\phib$ and $\etab$ are $\R^2$-valued.
It follows that 
\[
D_{\phib}\T(\phib,\omega)\etab 
= \ip{\omega^2M\phib''}{\etab}.
\]

Next, we use the definition of $\P$ in \eqref{eqn: P} to compute
\begin{align*}
D_{\phib}\P(\phib,\omega)\etab
&= \ip{D_{\phib}\Vb(\Delta_+(\omega)\phib)\Delta_+(\omega)\etab}{\oneb} \text{ by the chain rule} \\
&= \ip{\Vb'(\Delta_+(\omega)\phib).\Delta_+(\omega)\etab}{\oneb} \text{ by \eqref{eqn: Vb derivative}} \\
&= \ip{\Vb'(\Delta_+(\omega)\phib}{\Delta_+(\omega)\etab} \text{ by \eqref{eqn: dot mult oneb}} \\
&= -\ip{\Delta_-(\omega)\Vb'(\Delta_+(\omega)\phib)}{\etab} \text{ by \eqref{eqn: Delta pm adjoint}}.
\end{align*}

All together, we have
\[
D_{\phib}\G_c(\phib,\omega)\etab
= D_{\phib}\T(\phib,\omega)\etab + D_{\phib}\P(\phib,\omega)\etab
= \ip{c^2\omega^2M\phib''}{\etab} - \ip{\Delta_-(\omega)\Vb'(\Delta_+(\omega)\phib)}{\etab}.
\]
By our rewritten formula for $\Phib_c$ in \eqref{eqn: Phib rewritten}, this proves \eqref{eqn: Phib Frechet}.
\end{proof}

We collect two families of properties of $\G_c$ and $\Phib_c$.
The proofs for $\G_c$ are easy consequences of its definition in Theorem \ref{thm: G}, while those for $\Phib_c$ are also straightforward and could in fact be done (somewhat more laboriously) using just the definition of $\Phib_c$ in \eqref{eqn: Phib}.
However, we present proofs for $\Phib_c$ here as a consequence of the gradient formulation to emphasize the utility and efficiency of this formulation.

\begin{corollary}[Shift invariance]\label{cor: shift}
The following hold for all $\phib \in H_{\per}^2(\R^2)$ and $\omega \in \R$.

\begin{enumerate}[label={\bf(\roman*)},ref={(\roman*)}]

\item
The functional $\G_c$ is shift-invariant:
\begin{equation}\label{eqn: Gc shift invariance}
\G_c(S^{\theta}\phib,\omega)
= \G_c(\phib,\omega)
\end{equation}
for all $\theta \in \R$.

\item\label{part: Phib shift invariant}
The operator $\Phib_c$ is also shift-invariant:
\begin{equation}\label{eqn: Phib shift invariance}
\Phib_c(S^{\theta}\phib,\omega) 
= S^{\theta}\Phib_c(\phib,\omega)
\end{equation}
for all $\theta \in \R$.

\item\label{part: deriv ortho cond}
The operator $\Phib_c$ has the ``derivative orthogonality property''
\begin{equation}\label{eqn: Phib derivative ortho}
\ip{\Phib_c(\phib,\omega)}{\phib'} 
= 0.
\end{equation}

\end{enumerate}
\end{corollary}

\begin{proof}

\begin{enumerate}[label={\bf(\roman*)}]

\item
We use the identities 
\[
\ip{\phib''}{\nub_0} = 0
\quadword{and}
\Delta_+(\omega)\nub_0 = 0
\]
to obtain $\T(\phib+\alpha\nub_0,\omega) = \T(\phib,\omega)$ and $\P(\phib+\alpha\nub_0,\omega) = \P(\phib,\omega)$, respectively.
Since $\G_c = c^2\T + \P$, the identity \eqref{eqn: Gc translation invariance} follows.

\item
The chain rule and the identity \eqref{eqn: Gc shift invariance} imply
\[
D_{\phib}\G_c(\phib,\omega)\etab
= D_{\phib}\G_c(S^{\theta}\phib,\omega)S^{\theta}\etab
\]
for all $\etab \in H_{\per}^2(\R^2)$.
At the level of gradients, this reads
\[
\ip{\Phib_c(\phib,\omega)}{\etab}
= \ip{\Phib_c(S^{\theta}\phib,\omega)}{S^{\theta}\etab}.
\]
On the right, we use the adjoint relation \eqref{eqn: adjoint of shift} for shifts to rewrite
\[
\ip{\Phib_c(S^{\theta}\phib,\omega)}{S^{\theta}\etab}
= \ip{S^{-\theta}\Phib_c(S^{\theta}\phib,\omega)}{\etab}.
\]
It follows that
\[
\ip{\Phib_c(\phib,\omega)}{\etab}
= \ip{S^{-\theta}\Phib_c(S^{\theta}\phib,\omega)}{\etab}
\]
for all $\etab \in H_{\per}^2(\R^2)$, and so $\Phib_c(\phib,\omega) = S^{-\theta}\Phib_c(S^{\theta}\phib,\omega)$.
Applying $S^{\theta}$ to both sides yields \eqref{eqn: Phib shift invariance}.

\item
Now we differentiate the identity $\G_c(S^{\theta}\phib,\omega) = \G_c(\phib,\omega)$ from \eqref{eqn: Gc shift invariance} with respect to $\theta$ and evaluate the result at $\theta = 0$.
This yields
\[
D_{\phib}\G_c(S^0\phib,\omega)\left(\frac{\partial}{\partial\theta}[S^{\theta}\phib]\bigg|_{\theta = 0}\right)
= 0.
\]
Differentiating the shift operator yields the identity
\[
\frac{\partial}{\partial\theta}[S^{\theta}\phib]\bigg|_{\theta = 0}
= \phib',
\]
which is valid in the $L_{\per}^2$-norm since $\phib \in H_{\per}^2$.

Thus
\[
D_{\phib}\G_c(\phib,\omega)\phib'
= 0,
\]
and in the language of the gradient formulation, this says
\[
\ip{\Phib_c(\phib,\omega)}{\phib'}
= 0.
\qedhere
\]

\end{enumerate}
\end{proof}

\begin{corollary}[Translation invariance]\label{cor: trans invar}
The following hold for all $\phib \in H_{\per}^2(\R^2)$ and $\omega \in \R$.

\begin{enumerate}[label={\bf(\roman*)}]

\item
The functional $\G_c$ is translation-invariant in the sense that
\begin{equation}\label{eqn: Gc translation invariance}
\G_c(\phib+\alpha\nub_0,\omega)
= \G_c(\phib,\omega)
\end{equation}
for all $\alpha \in \R$, where $\nub_0$ is defined in \eqref{eqn: nub0}.

\item
The operator $\Phib_c$ is also translation-invariant:
\begin{equation}\label{eqn: Phib translation invariance}
\Phib_c(\phib+\alpha\nub_0,\omega) 
= \Phib_c(\phib,\omega)
\end{equation}
for all $\alpha \in \R$.

\item
The range of $\Phib_c$ is orthogonal to $\nub_0$:
\begin{equation}\label{eqn: Phib range ortho}
\ip{\Phib_c(\phib,\omega)}{\nub_0} 
= 0.
\end{equation}

\end{enumerate}
\end{corollary}

\begin{proof}
\begin{enumerate}[label={\bf(\roman*)}]

\item
This follows from the integral structure of $\G_c = c^2\T + \P$ from Theorem \ref{thm: G}, the $2\pi$-periodicity of $\phib$, and the identity
\[
\int_{-\pi}^{\pi} f(x+\theta) \dx
= \int_{-\pi}^{\pi} f(x) \dx,
\]
which is valid for all $\theta \in \R$ and all integrable, $2\pi$-periodic $f \colon [-\pi,\pi] \to \C$.

\item
We differentiate the identity $\G_c(\phib+\alpha\nub_0,\omega) = \G_c(\phib,\omega)$ from \eqref{eqn: Gc translation invariance} with respect to $\phib$ to find
\[
D_{\phib}\G_c(\phib+\alpha\nub_0,\omega)\etab
= D_{\phib}\G_c(\phib,\omega)\etab
\]
for all $\etab \in H_{\per}^2(\R^2)$ and thus
\[
\ip{\Phib_c(\phib+\alpha\nub_0,\omega)}{\etab}
= \ip{\Phib_c(\phib,\omega)}{\etab}.
\]
Since this holds for all $\etab$, we obtain \eqref{eqn: Phib translation invariance}.

\item
Now we differentiate the identity $\G_c(\phib+\alpha\nub_0,\omega) = \G_c(\phib,\omega)$ with respect to $\alpha$ and evaluate the result at $\alpha = 0$.
This yields
\[
0
= D_{\phib}\G_c(\phib+(0\cdot\nub_0),\omega)\left(\frac{\partial}{\partial\alpha}[\phib+\alpha\nub_0]\bigg|_{\alpha=0}\right)
= \ip{\Phib_c(\phib,\omega)}{\nub_0}.
\qedhere
\]

\end{enumerate}
\end{proof}

\begin{remark}\label{rem: WS}
The derivative orthogonality property \eqref{eqn: Phib derivative ortho} of $\Phib_c$ is the key to resolving the overdetermined periodic problem.
While this property follows quickly from the shift invariance of $\G_c$, as proved above, it is not quite as easy to prove directly from the definition of $\Phib_c$ as are all the other consequences of shift and translation invariance.
We discuss that direction of proof further in Lemma \ref{lem: deriv ortho}.
A similar derivative orthogonality property, deployed in somewhat different language, enabled Wright and Scheel \cite[Sec.\@ 4, p.\@ 548]{wright-scheel} to complete a Lyapunov--Schmidt analysis in which the linearization also had a two-dimensional kernel that, in the absence of symmetry, could not be reduced in dimension.
\end{remark}

\subsection{Function spaces and projection operators}\label{sec: function spaces}
The translation invariance identities \eqref{eqn: Phib translation invariance} and \eqref{eqn: Phib range ortho} mean that we can effectively ignore the contributions of $\nub_0$ to the problem $\Phib_c(\phib,\omega) = 0$.
So, we put
\begin{equation}\label{eqn: Y}
\Y
:= \set{\etab \in L_{\per}^2(\R^2)}{\ip{\etab}{\nub_0} = 0},
\end{equation}
\begin{equation}\label{eqn: X}
\X
:= H_{\per}^2(\R^2) \cap \Y,
\end{equation}
and 
\begin{equation}\label{eqn: Zcal-c}
\Zcal_c
:= \spn(\nub_1^c,\nub_2^c).
\end{equation}
It follows that $\Phib_c(\phib,\omega) \in \Y$ for all $\phib \in \X$ and $\omega \in \R$, and also
\[
\Zcal_c
= \ker(\L_c[\omega_c]) \cap \X = \ker(\L_c[\omega_c]^*) \cap \Y.
\]
Define
\begin{equation}\label{eqn: Pi}
\Pi_c
\colon \Y \to \Zcal_c
\colon \phib \mapsto \ip{\phib}{\nub_1^c}\nub_1^c + \ip{\phib}{\nub_2^c}\nub_2^c.
\end{equation}
Since $\nub_1^c$ and $\nub_2^c$ are orthogonal, per Corollary \ref{cor: eigenfunctions}, the operator $\Pi_c$ is the orthogonal projection of $\Y$ (and $\X$) onto $\Zcal_c$.
In particular,
\begin{equation}\label{eqn: Pi adjoint}
\ip{\Pi_c\phib}{\etab}
= \ip{\phib}{\Pi_c\etab}
\end{equation}
for all $\phib$, $\psib \in \Y$.

It turns out to be quite useful for us that the projection $\Pi_c$ and the first derivative $\partial_x$ commute.

\begin{lemma}\label{lem: Pi partial commute}
Let $\phib \in H_{\per}^1(\R^2)$.
Then
\begin{equation}\label{eqn: Pi partial commute}
\Pi_c\partial_x\phib
= \partial_x\Pi_c\phib.
\end{equation}
\end{lemma}

\begin{proof}
We use the integration by parts identity $\ip{\phib'}{\etab} = -\ip{\phib}{\etab'}$ from \eqref{eqn: integration by parts} and the derivative identities \eqref{eqn: nub derivatives} to compute
\begin{align*}
\Pi_c\partial_x\phib
&= \ip{\phib'}{\nub_1^c}\nub_1^c + \ip{\phib'}{\nub_2^c}\nub_2^c \\
&= -\ip{\phib}{\partial_x\nub_1^c}\nub_1^c - \ip{\phib}{\partial_x\nub_2^c}\nub_2^c \\
&= \ip{\phib}{\nub_2^c}\nub_1^c - \ip{\phib}{\nub_1^c}\nub_2^c \\
&= \ip{\phib}{\nub_2^c}\partial_x\nub_2^c + \ip{\phib}{\nub_1^c}\partial_x\nub_1^c \\
&= \partial_x\Pi_c\phib. \qedhere
\end{align*}
\end{proof}

Last, we state precisely the regularity of $\Phib_c$ and some of its derivatives on periodic Sobolev spaces.
The technical challenge here is that $\Phib_c$ is infinitely differentiable from $H_{\per}^2(\R^2)$ to $L_{\per}^2(\R^2)$ with respect to $\phib$, but any order derivative with respect to $\phib$ is only once continuously differentiable with respect to $\omega$.
This is ultimately a consequence of the limited differentiability of shift operators between periodic Sobolev spaces, as we discuss in Appendix \ref{app: shift operator calculus}.
We prove the next lemma in Appendix \ref{app: proof of lem Phib regularity}.

\begin{lemma}\label{lem: Phib regularity}
$\Phib_c \in \Cal^1(H_{\per}^2(\R^2)\times\R,L_{\per}^2(\R^2))$ and $D_{\phib}\Phib_c \in \Cal^1(H_{\per}^2(\R^2)\times\R,L_{\per}^2(\R^2))$.
\end{lemma}

\subsection{The Lyapunov--Schmidt decomposition: infinite-dimensional analysis}\label{sec: LS ID}
The approach here is classical and follows, for example, the proof of the Crandall--Rabinowitz--Zeidler theorem in \cite[Thm.\@ 1.5.1]{kielhofer}.
The difference appears in the following section, when we manage the two-dimensional kernel.

We use the projection operator $\Pi_c$ from \eqref{eqn: Pi} to make a Lyapunov--Schmidt decomposition for our problem $\Phib_c(\phib,\omega) = 0$.
First, with the spaces $\X$ and $\Y$ defined in \eqref{eqn: X} and \eqref{eqn: Y}, let
\begin{equation}\label{eqn: X Y infty}
\X_c^{\infty} := (\ind_{\X}-\Pi_c)(\X)
\quadword{and}
\Y_c^{\infty} := (\ind_{\Y} -\Pi_c)(\Y),
\end{equation}
where $\ind_{\X}$ and $\ind_{\Y}$ are the identity operators on $\X$ and $\Y$, respectively.
Consequently,
\begin{equation}\label{eqn: Zcal cap}
\Zcal_c \cap \X_c^{\infty} 
= \Zcal_c \cap \Y_c^{\infty}
= \{0\}.
\end{equation}

Next, write $\phib = \nub+\psib$, where $\nub \in \Zcal_c$ and $\psib \in \X_c^{\infty}$.
Then $\Phib_c(\phib,\omega) = 0$ if and only if
\begin{equation}\label{eqn: LS 1}
\begin{cases}
(\ind_{\Y} -\Pi_c)\Phib_c(\nub+\psib,\omega) = 0 \\
\Pi_c\Phib_c(\nub+\psib,\omega) = 0.
\end{cases}
\end{equation}
We solve the first equation in \eqref{eqn: LS 1} quickly with a direct application of the implicit function theorem \cite[Thm.\@ I.1.1]{kielhofer}.

Define
\begin{equation}\label{eqn: Fc-infty}
\F_c^{\infty} 
\colon \X_c^{\infty} \times \Zcal_c \times \R \to \Y_c^{\infty}
\colon (\psib,\nub,\omega) \mapsto (\ind_{\Y} -\Pi_c)\Phib_c(\nub+\psib,\omega).
\end{equation}
Certainly $\F_c^{\infty}(0,0,\omega) = 0$ for all $\omega$, and we have
\[
D_{\psib}\F_c^{\infty}(0,0,\omega_c)
= (\ind_{\Y} -\Pi_c)\restr{\L_c[\omega_c]}{\X_c^{\infty}}.
\]
This operator has trivial kernel in $\X_c^{\infty}$ and trivial cokernel in $\Y_c^{\infty}$ by \eqref{eqn: Zcal cap}, and so it is invertible.
(More precisely, the closure of its range is the orthogonal complement of its cokernel, which is all of $\Y_c^{\infty}$.
But the range is closed by the coercive estimate in Corollary \ref{cor: coercive}.)
Since $\Phib_c \in \Cal^1(H_{\per}^2(\R^2),L_{\per}^2(\R^2))$ by Lemma \ref{lem: Phib regularity}, with $\Phib_c(0,\omega) = 0$ for all $\omega$, the implicit function theorem yields $\delta_c$, $\ep_c > 0$ and a map $\Psib_c \in \Cal^1\big(\B_{\Zcal_c \times \R}((0,\omega_c);\delta_c), \B_{\X_c^{\infty}}(0;\ep_c)\big)$
such that
\begin{equation}\label{eqn: what Psib does}
\F_c^{\infty}(\Psib_c(\nub,\omega),\nub,\omega)
= 0.
\end{equation}
Moreover, if $\F_c^{\infty}(\psib,\nub,\omega) = 0$ for some $\psib \in \B_{\X_c^{\infty}}(0;\ep_c)$ and $(\nub,\omega) \in \B_{\Zcal_c \times \R}((0,\omega_c);\delta_c)$, then $\psib = \Psib_c(\nub,\omega)$.
(Recall that $\B_{\X}(x_0;r) = \set{x \in \X}{\norm{x-x_0}_{\X} < r}$ for $x_0 \in \X$ and $r > 0$.)
We pause to collect some useful properties of this map $\Psib_c$.

\begin{lemma}\label{lem: Psib}
Let $\nub \in \Zcal_c$ and $\omega \in \R$ with $\norm{\nub}_{H_{\per}^2} + |\omega-\omega_c| < \delta_c$.
Then the following identities hold.

\begin{enumerate}[label={\bf(\roman*)}, ref={(\roman*)}]

\item\label{part: Psib1}
$\ip{\Phib_c(\nub+\Psib_c(\nub,\omega),\omega)}{\partial_x\nub} = 0$ for $(\nub,\omega) \in \B_{\Zcal_c\times\R}((0,\omega_c);\delta_c)$.

\item
$\Psib_c(0,\omega) = 0$ for $(0,\omega) \in \B_{\Zcal_c\times\R}((0,\omega_c);\delta_c)$.

\item\label{part: Psib3}
$D_{\nub}\Psib_c(0,\omega_c) = 0$.

\item
$D_{\nub}\Psib_c \in \Cal^1(\B_{\Zcal_c \times \R}((0,\omega_c);\delta_c), \b(\Zcal_c,\X_c^{\infty}))$.
\end{enumerate}
\end{lemma}

\begin{proof}

\begin{enumerate}[label={\bf(\roman*)}]

\item
The derivative orthogonality property of $\Phib_c$ from \eqref{eqn: Phib derivative ortho} implies
\begin{equation}\label{eqn: Psib ortho 1}
\ip{\Phib_c(\nub+\Psib_c(\nub,\omega),\omega)}{\partial_x[\nub+\Psib_c(\nub,\omega)]}
= 0,
\end{equation}
and by \eqref{eqn: what Psib does}, we have
\begin{equation}\label{eqn: Psib ortho 2}
(\ind_{\Y}-\Pi_c)\Phib_c(\nub+\Psib_c(\nub,\omega),\omega)
= 0.
\end{equation}
Then we compute
\begin{equation}\label{eqn: the miracle calculation}
\begin{aligned}
0
&= \ip{\Pi_c\Phib_c(\nub+\Psib_c(\nub,\omega),\omega)}{\partial_x[\nub+\Psib_c(\nub,\omega)]} \text{ using \eqref{eqn: Psib ortho 2} in \eqref{eqn: Psib ortho 1}} \\
&= \ip{\Phib_c(\nub+\Psib_c(\nub,\omega),\omega)}{\Pi_c\partial_x[\nub+\Psib_c(\nub,\omega)]} \text{ by \eqref{eqn: Pi adjoint}} \\
&= \ip{\Phib_c(\nub+\Psib_c(\nub,\omega),\omega)}{\partial_x\Pi_c[\nub+\Psib_c(\nub,\omega)]} \text{ since } \partial_x \text{ and } \Pi_c \text{ commute} \\
&=  \ip{\Phib_c(\nub+\Psib_c(\nub,\omega),\omega)}{\partial_x\nub} \text{ since } \Pi_c\nub = \nub \text{ and } \Pi_c\Psib_c(\nub,\omega) = 0. 
\end{aligned}
\end{equation}

\item
By definition of $\F_c^{\infty}$ in \eqref{eqn: Fc-infty}, we have 
\[
\F_c(0,0,\omega) 
= (\ind_{\Y} -\Pi_c)\Phib_c(0,\omega)
= 0
\]
for all $\omega$.
By the uniqueness property of $\Psib_c$, we have $\Psib(0,\omega) = 0$ for all $\omega$.

\item
We differentiate \eqref{eqn: Psib ortho 2} with respect to $\nub \in \Zcal_c$ to find the operator-valued identity
\begin{equation}\label{eqn: impl diff 0}
(\ind_{\Y} -\Pi_c)D_{\phib}\Phib_c(\nub+\Psib_c(\nub,\omega),\omega)\big(\ind_{\Zcal_c} + D_{\nub}\Psib_c(\nub,\omega)\big)
= 0.
\end{equation}
Here $\ind_{\Zcal_c}$ is the identity operator on $\Zcal_c = \ker(\L_c[\omega_c])$.
Taking $\nub = 0$ and $\omega = \omega_c$ collapses \eqref{eqn: impl diff 0} to
\[
(\ind_{\Y} -\Pi_c)\L_c[\omega_c]\big(\ind_{\Zcal_c}+D_{\nub}\Psib_c(0,\omega_c)\big)
= 0,
\]
recalling $\L_c[\omega_c] = D_{\phib}\Phib_c(0,\omega_c)$ from \eqref{eqn: Lc}.
Since $\restr{\L_c[\omega_c]}{\Zcal_c} = 0$, this further reduces to
\[
\L_c[\omega_c]D_{\nub}\Psib_c(0,\omega_c)-\Pi_c\L_c[\omega_c]D_{\nub}\Psib_c(0,\omega_c)
= 0.
\]
Because $\ker(\L_c[\omega_c]) = \ker(\L_c[\omega_c]^*)$, it follows from the definition of $\Pi_c$ in \eqref{eqn: Pi} that $\Pi_c\L_c[\omega_c] = 0$.
Thus 
\[
\L_c[\omega_c]D_{\nub}\Psib_c(0,\omega_c)
= 0,
\]
and so the range of $D_{\nub}\Psib_c(0,\omega_c)$ is contained in $\ker(\L_c[\omega_c]) = \Zcal_c$.
But the range of $D_{\nub}\Psib_c(0,\omega_c)$ is also contained in $\Y_c^{\infty}$, and $\Y_c^{\infty} \cap \Zcal_c = \{0\}$ by \eqref{eqn: Zcal cap}.
Thus the range of $D_{\nub}\Psib_c(0,\omega_c)$ is trivial.

\item
This follows from the implicit function theorem, which guarantees that $\Psib_c$ is as regular as $\Phib_c$.
Since $D_{\phib}\Phib_c \in \Cal^1(H_{\per}^2(\R^2),L_{\per}^2(\R^2))$ by Lemma \ref{lem: Phib regularity}, $D_{\nub}\Psib_c$ inherits this regularity on $\B_{\Zcal_c\times\R}((0,\omega_c);\delta_c)$.
\qedhere
\end{enumerate}
\end{proof}

\subsection{The Lyapunov--Schmidt decomposition: finite-dimensional analysis}\label{sec: LS FD}
Now we solve the second equation in the Lyapunov--Schmidt decomposition \eqref{eqn: LS 2} with $\psib = \Psib_c(\nub,\omega)$.
This amounts to solving the pair of equations
\begin{equation}\label{eqn: LS FD}
\begin{cases}
\ip{\Phib_c(\nub+\Psib_c(\nub,\omega),\omega)}{\nub_1^c} = 0 \\
\ip{\Phib_c(\nub+\Psib_c(\nub,\omega),\omega)}{\nub_2^c} = 0,
\end{cases}
\end{equation}
where $\nub$ is a linear combination of the two linearly independent eigenfunctions $\nub_1^c$ and $\nub_2^c$.
The apparent quandary is that we want solutions parametrized in amplitude, so formally this suggests $\nub + \Psib_c(\nub,\omega) = \O(a)$.
This leads to our taking $\nub = a\nub_1^c$ below, which may appear to remove a degree of freedom from the ansatz.
In turn, this could appear to be problematic, given that we have two equations to solve above in \eqref{eqn: LS FD}.
Ostensibly we could have stayed with $\nub$ as a combination of $\nub_1^c$ and $\nub_2^c$.
None of this, however, is a problem, and we discuss at length in Section \ref{sec: the choice} why.

With the choice of 
\begin{equation}\label{eqn: the choice}
\nub 
= a\nub_1^c
\end{equation}
for $a \in \R$ sufficiently small, we need
\begin{subnumcases}{}
\ip{\Phib_c(a\nub_1^c+\Psib_c(a\nub_1^c,\omega),\omega)}{\nub_1^c} = 0 \label{eqn: LS 2} \\
\ip{\Phib_c(a\nub_1^c+\Psib_c(a\nub_1^c,\omega),\omega)}{\nub_2^c} = 0. \label{eqn: LS 3}
\end{subnumcases}
We claim that \eqref{eqn: LS 3} always holds.
Indeed, for $a = 0$, it is trivially true, since $\Phib_c(0,\omega) = 0$, while for $a \ne 0$ we have
\begin{align*}
\ip{\Phib_c(a\nub_1^c+\Psib_c(a\nub_1^c,\omega),\omega)}{\nub_2^c}
&= a^{-1}\ip{\Phib_c(a\nub_1^c+\Psib_c(a\nub_1^c,\omega),\omega)}{a\nub_2^c} \\ 
&= -a^{-1}\ip{\Phib_c(a\nub_1^c+\Psib_c(a\nub_1^c,\omega),\omega)}{\partial_x[a\nub_1^c]} \text{ by \eqref{eqn: nub derivatives}}\\
&= 0 \text{ by part \ref{part: Psib1} of Lemma \ref{lem: Psib}}.
\end{align*}
We emphasize that our success here traces back to the derivative orthogonality property \eqref{eqn: Phib derivative ortho}.

We conclude by solving \eqref{eqn: LS 2} with another application of the implicit function theorem, and this, again, is effectively the remainder of the proof of the Crandall--Rabinowitz--Zeidler theorem \cite[Thm.\@ 1.5.1]{kielhofer}.
Define
\[
\F_c^0 
\colon \B_{\R^2}((\omega_c,0); \delta_c/2)  \to \R
\colon (\omega,a) \mapsto \ip{\Phib_c(a\nub_1^c+\Psib_c(a\nub_1^c,\omega),\omega)}{\nub_1^c}.
\]
The threshold $\delta_c$ arose from the infinite-dimensional implicit function theorem argument in Section \ref{sec: LS ID}.

Since $\F_c^0(\omega,0) = 0$ for all $\omega$, we have
\[
\F_c^0(\omega,a) = a\H_c(\omega,a),
\qquad
\H_c(\omega,a) := \int_0^1 D_a\F_c^0(\omega,a\alpha) \dalpha.
\]
It therefore suffices to solve $\H_c(\omega,a) = 0$ by selecting $\omega$ as a function of $a$, and we do this by checking $\H_c(\omega_c,0) = 0$ and $D_{\omega}\H_c(\omega_c,0) \ne 0$.

Toward this end, we first differentiate
\[
D_a\F_c^0(\omega,a)
= \ip{D_{\phib}\Phib_c(a\nub_1^c+\Psib_c(a\nub_1^c,\omega))\big(\nub_1^c+D_{\nub}\Psib_c(a\nub_1^c,\omega)\nub_1^c\big)}{\nub_1^c}.
\]
We put $a = 0$ and use $\Psib_c(0,\omega) = 0$ to find for any $\omega$ that
\begin{equation}\label{eqn: Hc-omega}
\H_c(\omega,0)
= \int_0^1 D_a\F_c^0(\omega,0) \dalpha
= D_a\F_c(\omega,0)
= \ip{\L_c[\omega]\big(\nub_1^c+D_{\nub}\Psib_c(0,\omega)\nub_1^c\big)}{\nub_1^c}.
\end{equation}
In the special case of $\omega=\omega_c$, we can use either $D_{\nub}\Psib_c(0,\omega) = 0$ from part \ref{part: Psib3} of Lemma \ref{lem: Psib} or the condition $\L_c[\omega_c]^*\nub_1^c = 0$ to reduce \eqref{eqn: Hc-omega} to
\[
\H_c(\omega_c,0)
= \ip{\L_c[\omega_c]\nub_1^c}{\nub_1^c}
= 0.
\]

Next, with $\omega$ arbitrary, we differentiate \eqref{eqn: Hc-omega} with respect to $\omega$ and use the product rule to find
\[
D_{\omega}\H_c(\omega,0)
= \ip{\L_c'[\omega]\nub_1^c}{\nub_1^c}
+ \ip{\L_c'[\omega]D_{\nub}\Psib_c(0,\omega)\nub_1^c}{\nub_1^c}
+ \ip{\L_c[\omega]D_{\nub\omega}\Psib_c(0,\omega)\nub_1^c}{\nub_1^c}.
\]
Here we are using the shorter notation from \eqref{eqn: Lc-prime} of $\L_c'[\omega] = D_{\phib\omega}\Phib_c(0,\omega)$.
Taking $\omega = \omega_c$, we use $D_{\nub}\Psib_c(0,\omega_c) = 0$ to find that the second term is $0$.
At $\omega_c$, the third term is $0$ since $\L_c[\omega_c]^*\nub_1^c = 0$.
And so
\[
D_{\omega}\H_c(\omega_c,0)
= \ip{\L_c'[\omega_c]\nub_1^c}{\nub_1^c}
\ne 0,
\]
by Corollary \ref{cor: transversality}.

We are now in position to invoke the implicit function theorem once more, and we find $a_c$, $b_c > 0$ and a map $\Omega_c \colon (-a_c,a_c) \to \R$ such that 
\[
\H_c(\Omega_c(a),a)
= 0
\]
for $|a| < a_c$, while if $|\omega-\omega_c| < b_c$ and $|a| < a_c$ and $\H_c(\omega,a) = 0$, then $\omega = \Omega_c(a)$.
In particular, $\Omega_c(0) = \omega_c$.

In short, taking
\[
\phib_c^a := a\nub_1 + \Psib_c(a\nub_1,\Omega_c(a))
\quadword{and}
\omega_c^a := \Omega_c(a)
\]
solves our original problem $\Phib_c(\phib_c^a,\omega_c^a) = 0$.
We can expose uniformly the ``amplitude'' parameter of $a$ in $\phib_c^a$ by setting
\[
\psib_c(a)
:= \Psib_c(a\nub_1,\Omega_c(a)),
\]
and computing
\[
\psib_c(0) = \Psib_c(0,\omega_c) = 0,
\quad
\psib_c(a) = a\int_0^1 D_a\psib_c(a\alpha) \dalpha,
\quad\text{and}\quad
\phib = a\left(\nub_1+\int_0^1 D_a\psib_c(a\alpha) \dalpha\right).
\]
Likewise, we can write
\[
\omega_c^a
= \omega_c + a\xi_c^a,
\qquad
\xi_c^a := \int_0^1 D_a\Omega_c(a\alpha) \dalpha.
\]
This concludes our first proof of Theorem \ref{thm: main}.

\subsection{Remarks on the choice \eqref{eqn: the choice}}\label{sec: the choice}
We discuss our decision at the start of Section \ref{sec: LS FD} to specialize the finite-dimensional component $\nub$ to $\nub = a\nub_1^c$.
We consider two aspects of this choice to allay any concerns about its peculiarity or restrictiveness.

First, up to a shift, any solution $\phib$ to $\Phib_c(\phib,\omega) = 0$ has this form $\phib = a\nub_1^c + \psib$, with $a \in \R$ and $\psib$ orthogonal to $\nub_0$, $\nub_1^c$, and $\nub_2^c$.

\begin{lemma}\label{lem: soln structure 1}
Let $\omega \in \R$.
If $\phib \in \X$ solves $\Phib_c(\phib,\omega) = 0$, then there exist $a$, $\theta \in \R$ and $\psib \in \X_c^{\infty}$ such that 
\begin{equation}\label{eqn: shifted phib}
\phib 
= S^{\theta}(a\nub_1^c + \psib).
\end{equation}
In particular, $\Phib_c(a\nub_1^c+\psib,\omega) = 0$, as well.
\end{lemma}

\begin{proof}
We prove the last sentence first.
If a solution $\phib$ to $\Phib_c(\phib,\omega) = 0$ has the form \eqref{eqn: shifted phib}, then the shift invariance of $\Phib_c$ from \eqref{eqn: Phib shift invariance} implies
\[
\Phib_c(a\nub_1^c+\psib,\omega)
= S^{-\theta}\Phib_c(S^{\theta}(a\nub_1^c+\psib),\omega)
= S^{-\theta}\Phib_c(\phib,\omega)
= 0.
\]

It remains to prove the decomposition \eqref{eqn: shifted phib}.
We can always write 
\begin{equation}\label{eqn: phib with tilde}
\phib
= a_1\nub_1^c + a_2\nub_2^c + \tilde{\psib}
\end{equation}
for some $a_1$, $a_2 \in \R$ and $\tilde{\psib} \in \X_c^{\infty}$.
Write
\[
a_1-ia_2
= ae^{i\theta}
\]
in polar coordinates, where $a$, $\theta \in \R$.
(If $a_1 = a_2 = 0$, just take $a = 0$ and $\psib = \tilde{\psib}$.)
The identities $\nub_1^c(x) = 2\re[e^{ix}\hat{\nub_1^c}(1)]$ from \eqref{eqn: nub1} and $\nub_2^c = S^{-\pi/2}\nub_1^c$ from \eqref{eqn: nub shifts} then imply
\[
a_1\nub_1^c(x) +a_2\nub_2^c(x)
= a_1\nub_1^c(x)+a_2S^{-\pi/2}\nub_1^c(x)
= 2\re[(a_1-ia_2)e^{ix}\hat{\nub_1^c}(1)]
= a(S^{\theta}\nub_1^c)(x).
\]
Returning to \eqref{eqn: phib with tilde}, we have
\[
\phib = S^{\theta}(a\nub_1^c+\psib),
\qquad
\psib := S^{-\theta}\tilde{\psib}.
\]

We conclude by checking that if $\tilde{\psib} \in \X_c^{\infty}$, then $\psib = S^{-\theta}\tilde{\psib} \in \X_c^{\infty}$.
That is, we assume
\begin{equation}\label{eqn: original membership in Xcinfty1}
\ip{\tilde{\psib}}{\nub_1^c} 
= \ip{\tilde{\psib}}{\nub_2^c}
= 0,
\end{equation}
and we want to show
\begin{equation}\label{eqn: shifted membership in Xcinfty1}
\ip{S^{-\theta}\tilde{\psib}}{\nub_1^c}
= \ip{S^{-\theta}\tilde{\psib}}{\nub_2^c}
= 0.
\end{equation}
By the orthogonality condition \eqref{eqn: nub ortho equiv}, our assumption \eqref{eqn: original membership in Xcinfty1} is equivalent to
\begin{equation}\label{eqn: original membership in Xcinfty2}
\hat{\tilde{\psib}}(1)\cdot\hat{\nub_1^c}(1)
= 0,
\end{equation}
and our desired conclusion \eqref{eqn: shifted membership in Xcinfty1} is equivalent to
\begin{equation}\label{eqn: shifted membership in Xcinfty2}
(e^{-i\theta}\hat{\tilde{\psib}}(1))\cdot\hat{\nub_1^c}(1)
= 0.
\end{equation}
Certainly \eqref{eqn: original membership in Xcinfty2} implies \eqref{eqn: shifted membership in Xcinfty2}.
\end{proof}

Consequently, trying an ansatz of the form $\phib = a\nub_1^c + b\nub_2^c + \psib$ in the hope that $a$ and $b$ would be enough to manage the two equations in \eqref{eqn: LS FD} will not be effective; informally, the problem simply does not ``see'' the two unknowns $a$ and $b$ simultaneously.
And such an ansatz would not expose the single uniform amplitude parameter that we desire, anyway.

There is nothing special about $\nub_1^c$ here, and we could just as easily show that any solution to $\Phib_c(\phib,\omega)$ is a shifted version of a solution of the form $a\nub_2^c + \psib$.
In fact, we could have run the bifurcation argument above using $\nub_2^c$ throughout in place of $\nub_1^c$.
This hinges on expressing the transversality inequality of Corollary \ref{cor: transversality} in terms of $\nub_2^c$, which is possible because of the calculation
\begin{equation}\label{eqn: transverse on nub2}
\ip{\L_c'[\omega_c]\nub_2^c}{\nub_2^c}
= \ip{\L_c'[\omega_c]\partial_x\nub_1^c}{\partial_x\nub_1^c}
= \ip{\partial_x\L_c'[\omega_c]\nub_1^c}{\partial_x\nub_1^c}
= -\ip{\L_c'[\omega_c]\nub_1^c}{\partial_x^2\nub_1^c}
= \ip{\L_c'[\omega_c]\nub_1^c}{\nub_1^c}.
\end{equation}
The second equality relies on the commutativity of the Fourier multipliers $\partial_x$ and $\L_c'[\omega_c]$ on $H_{\per}^3(\R^2)$.
However, since we can write any solution in the form
\[
S^{\theta}(a\nub_1^c+\psib)
= S^{\theta}(aS^{\pi/2}\nub_2^c+\psib)
= S^{\theta+\pi/2}(a\nub_2^c+S^{-\pi/2}\psib)
\]
with $\psib$, and thus (by the end of the proof of Lemma \ref{lem: soln structure 1}) $S^{-\pi/2}\psib$, orthogonal to $\nub_1^c$ and $\nub_2^c$, there is not much point to this line of inquiry.

Next, we consider further the special form of solutions to $\Phib_c(\phib,\omega) = 0$ as given in Theorem \ref{thm: main}: they are $\phib = a(\nub_1^c + \psib)$ with $\psib$ again orthogonal to $\nub_1^c$ and $\nub_2^c$.
This may appear to be less general than the result of Lemma \ref{lem: soln structure 1}, which says that, up to a shift, any solution has the form $a\nub_1^c + \tilde{\psib}$, with $\tilde{\psib}$ satisfying the perennial orthogonality conditions.
Of course, if $a \ne 0$, then this solution factors as $a(\nub_1^c + a^{-1}\tilde{\psib})$, and that has the form given by Theorem \ref{thm: main}.
It turns out that all {\it{small}} nontrivial solutions to $\Phib_c(\phib,\omega) = 0$ have this special factored form (again, up to a shift from Lemma \ref{lem: soln structure 1}).

We prove a negative version of this result, which says that if the shifted solution from Lemma \ref{lem: soln structure 1} has the form $\phib = \tilde{\psib}$ alone, i.e., if $a = 0$, and if this solution is sufficiently small, then it is trivial.

\begin{lemma}\label{lem: soln structure 2}
There exists $\delta_c^{\infty} > 0$ such that if $\psib \in \X_c^{\infty}$ and $\omega \in \R$ with $\norm{\psib}_{H_{\per}^2} + |\omega-\omega_c| < \delta_c^{\infty}$, and if $\Phib_c(\psib,\omega) = 0$, then $\psib = 0$.
\end{lemma}

\begin{proof}
Define
\[
\tilde{\F}_c^{\infty} 
\colon \X_c^{\infty} \times \R \to \Y_c^{\infty}
\colon (\psib,\omega) \mapsto (\ind_{\Y} -\Pi_c)\Phib_c(\psib,\omega).
\]
The notation and structure of this map are intentionally similar to those of $\F_c^{\infty}$ in \eqref{eqn: Fc-infty}, and the spaces $\X_c^{\infty}$ and $\Y_c^{\infty}$ are defined in \eqref{eqn: X Y infty}.
Then $\tilde{\F}_c^{\infty}(0,\omega) = 0$ for all $\omega$ and
\[
D_{\psib}\tilde{\F}_c^{\infty}(0,\omega_c)
= (\ind_{\Y} -\Pi_c)\restr{\L_c[\omega_c]}{\X_c^{\infty}}
\]

As with the analogous linearization in Section \ref{sec: LS ID}, this operator has trivial kernel and cokernel and therefore is invertible.
The implicit function theorem gives $\delta_c^{\infty}$, $\ep_c^{\infty} > 0$ and a map
\[
\Psib_c^{\infty} 
\colon (\omega_c-\delta_c^{\infty},\omega_c+\delta_c^{\infty}) \to \X_c^{\infty}
\]
such that $\tilde{\F}_c^{\infty}(\Psib_c^{\infty}(\omega),\omega) = 0$ for $|\omega-\omega_c| < \delta_c^{\infty}$.
Moreover, if $|\omega-\omega_c| < \delta_c^{\infty}$ and $\norm{\psib}_{H_{\per}^2} < \ep_c^{\infty}$, and if $\tilde{\F}_c^{\infty}(\psib,\omega) = 0$, then $\psib = \Psib_c^{\infty}(\omega)$.
But $\tilde{\F}_c^{\infty}(0,\omega) = 0$ for all $\omega$, and so we must have $\Psib_c^{\infty}(\omega) = 0$ for all $\omega$.
Conversely, if $\Phib_c(\psib,\omega) = 0$ then $\tilde{\F}_c^{\infty}(\psib,\omega) = 0$, too, and so if $|\omega-\omega_c| < \delta_c^{\infty}$ and $\norm{\psib}_{H_{\per}^2} < \ep_c^{\infty}$, then $\psib = \Psib_c^{\infty}(\omega) = 0$.
\end{proof}

Lemmas \ref{lem: soln structure 1} and \ref{lem: soln structure 2} together effectively tell us that the only worthwhile form of solutions to $\Phib_c(\phib,\omega)$ is $\phib = a(\nub_1^c+\psib)$.
When we study this problem quantitatively in Section \ref{sec: quantitative}, we will start directly with an ansatz of this form.

\section{The Lyapunov Center Formulation}\label{sec: LC}
We solve the problem
\begin{equation}\label{eqn: LC}
\Phib_c(\phib,\omega) + \gamma\phib'
= 0
\end{equation}
for $\phib \in H_{\per}^2(\R^2)$ and $\omega$, $\gamma \in \R$.
The extra unknown $\gamma$ closes the overdetermined system that results from the two solvability conditions induced by the two-dimensional cokernel; we show momentarily that any solution to \eqref{eqn: LC} necessarily has $\gamma=0$, and so solving \eqref{eqn: LC} really returns solutions to our original problem $\Phib_c(\phib,\omega) = 0$.
This strategy is based on the work of Wright and Scheel in \cite[Sec.\@ 8]{wright-scheel}; in their words, ``[t]he idea is to augment the Hamiltonian equation with a dissipation term, for instance $\gamma\nabla{H}$, so that for $\gamma\ne0$, the system is gradient-like and does not possess any small non-equilibrium solutions.''
In turn, the proof there was motivated by a proof of the Lyapunov center theorem \cite[Thm.\@ 3.2]{ambrosetti-prodi}.

Some (though not all) of the implicit function theorem arguments are quite similar to those in Sections \ref{sec: LS ID} and \ref{sec: LS FD}, so we move rather more briskly here.
We emphasize that while the existence proof developed in this section is not strictly necessary for logical completeness of our argument, we see it as a potentially useful alternative to the first proof in that it is completely independent of the gradient structure. 

First we show that any nonconstant solution to \eqref{eqn: LC} has $\gamma = 0$; the following calculation is similar to \cite[Lem.\@ 3.1]{ambrosetti-prodi}, which was done in preparation for their proof of the Lyapunov center theorem.
If \eqref{eqn: LC} holds, then
\begin{equation}\label{eqn: why LC works}
0
= \ip{\Phib_c(\phib,\omega)+\gamma\phib'}{\phib'}
= \ip{\Phib_c(\phib,\omega)}{\phib'} + \gamma\norm{\phib'}^2
= \gamma\norm{\phib'}^2.
\end{equation}
Since $\phib$ is nonconstant, we must have $\gamma = 0$.

In \eqref{eqn: why LC works} we used the derivative orthogonality property
\[
\ip{\Phib_c(\phib,\omega)}{\phib'}
= 0,
\]
as established in Corollary \ref{cor: shift} using the gradient formulation.
However, with some more work, this can be checked directly from the definition of $\Phib_c$.

\begin{lemma}\label{lem: deriv ortho}
Let $\phib \in H_{\per}^1(\R^2)$ and define
\begin{equation}\label{eqn: Jc}
\J_c(\phib,\omega)
:= c^2\omega^2\frac{(\phi_1')^2}{2}
+ c^2\omega^2\frac{(\phi_2')^2}{2w} 
+ \V_1(\phi_2-S^{-\omega}\phi_1)
+ \V_2(\phi_1-S^{-\omega}\phi_2)
+ \I(\phib,\omega),
\end{equation}
where
\[
\I(\phib,\omega)(x)
:= 
\int_x^{x-\omega} \V_1'(S^{\omega}\phi_2-\phi_1)\phi_1'
+ \int_x^{x-\omega} \V_2'(S^{\omega}\phi_1-\phi_2)\phi_2'.
\]
Then
\begin{equation}\label{eqn: ip Phib phib J}
\ip{\Phib_c(\phib,\omega)}{\phib'}
= \int_{-\pi}^{\pi} \partial_x\J_c(\phib,\omega).
\end{equation}
In particular, since $\J_c(\phib,\omega)$ is $2\pi$-periodic, 
\[
\ip{\Phib_c(\phib,\omega)}{\phib'}
= 0.
\]
\end{lemma}

\begin{proof}
The proof of \eqref{eqn: ip Phib phib J} is a direct calculation using the definition of $\J_c$ above and the definition of $\Phib_c$ in \eqref{eqn: Phib}, but, for clarity, we provide some details as to how $\J_c$ naturally arises.
First, computing the dot product yields
\begin{multline*}
\Phib_c(\phib,\omega)\cdot\phib'
= c^2\omega^2\phi_1''\phi_1'
+ \frac{c^2}{w}\omega^2\phi_2''\phi_2'
+\V_1'(\phi_2-S^{-\omega}\phi_1)\phi_2'
-\V_1'(S^{\omega}\phi_2-\phi_1)\phi_1' \\
+ \V_2'(\phi_1-S^{-\omega}\phi_2)\phi_1'
-\V_2'(S^{\omega}\phi_1-\phi_2)\phi_2'.
\end{multline*}
The first two terms are perfect derivatives, but the others involving $\V_1'$ and $\V_2'$ need some modification.
We work with just the $\V_1'$ terms to show the origin of the first of the two integrals in $\I$.
Adding zero, we have
\begin{align*}
\V_1'(\phi_2-S^{-\omega}\phi_1)\phi_2'
-\V_1'(S^{\omega}\phi_2-\phi_1)\phi_1'
&= \V_1'(\phi_2-S^{-\omega}\phi_1)\phi_2' - \V_1'(\phi_2-S^{-\omega}\phi_1)S^{-\omega}\phi_1' \\
&+ \V_1'(\phi_2-S^{-\omega}\phi_1)S^{-\omega}\phi_1'-\V_1'(S^{\omega}\phi_2-\phi_1)\phi_1' \\
\\
&= \V_1'(\phi_2-S^{-\omega}\phi_1)(\phi_2'-S^{-\omega}\phi_1') \\
&+ (S^{-\omega}-1)[\V_1'(S^{\omega}\phi_2-\phi_1)\phi_1'].
\end{align*}
Here we have factored
\[
\V_1'(\phi_2-S^{-\omega}\phi_1)S^{-\omega}\phi_1'
= S^{-\omega}[\V_1'(S^{\omega}\phi_2-\phi_1)\phi_1'].
\]
to get the second term in the second equality above.
In the first term of that second equality, we immediately recognize the perfect derivative
\[
\V_1'(\phi_2-S^{-\omega}\phi_1)(\phi_2'-S^{-\omega}\phi_1')
= \partial_x[\V_1(\phi_2-S^{-\omega}\phi_1)].
\]
Finally, we use the identity
\[
\partial_x\left[\int_x^{x-\omega} f\right]
= f(x-\omega)-f(x)
= [(S^{-\omega}-1)f](x)
\]
to rewrite
\[
(S^{-\omega}-1)[\V_1'(S^{\omega}\phi_2-\phi_1)\phi_1']
= \partial_x\left[\int_x^{x-\omega} \V_1'(S^{\omega}\phi_2-\phi_1)\phi_1'\right].
\]
Repeating these calculations on the $\V_2'$ terms shows $\Phib_c(\phib,\omega)\cdot\phib' = \partial_x\J_c(\phib,\omega)$, and that is \eqref{eqn: ip Phib phib J}.
\end{proof}

\begin{remark}
The structure of the operator $\J_c$ in \eqref{eqn: Jc} bears some resemblance to the first integral in \cite[Prop.\@ 3.10]{faver-hupkes-spatial-dynamics} for the spatial dynamics formulation of the traveling wave problem.
Indeed, the existence of that conserved quantity from the spatial dynamics viewpoint inspired us to search for a related conserved quantity in this traveling wave framework, and $\J_c$ naturally emerged.
Moreover, $\J_c$ is constant on solutions to $\Phib_c(\phib,\omega) = 0$ in the sense that if $\phib$ and $\omega$ satisfy this equation, it can be checked that $\partial_x\J_c(\phib,\omega) = 0$.
This leads to another (related) proof that the existence of a nonconstant solution to \eqref{eqn: LC} forces $\gamma = 0$: if $\phib$ and $\omega$ meet \eqref{eqn: LC}, it follows that 
\[
\partial_x\J_c(\phib,\omega)
= -\gamma\left((\phi_1')^2+\frac{(\phi_2')^2}{w}\right).
\]
If $\gamma > 0$, then $\partial_x\J_c(\phib,\omega)$ is nonpositive and not identically zero; since $\J_c(\phib,\omega)$ is periodic, this is impossible.
A similar contradiction results if $\gamma < 0$.
\end{remark}

Now we study the problem \eqref{eqn: LC} with a Lyapunov--Schmidt decomposition as in Sections \ref{sec: LS ID} and \ref{sec: LS FD}.
Using the projection operator $\Pi_c$ and the function spaces $\X_c^{\infty}$, $\Y_c^{\infty}$, and $\Zcal_c$ from Section \ref{sec: function spaces}, we split \eqref{eqn: LC} into the pair of equations
\[
\begin{cases}
(\ind_{\Y} -\Pi_c)\Phib_c(\nub+\psib,\omega) + \gamma(\ind_{\X}-\Pi_c)(\nub'+\psib') = 0 \\
\Pi_c\Phib_c(\nub+\psib,\omega) + \gamma\Pi_c(\nub'+\psib') = 0,
\end{cases}
\]
where $\phib = \nub+\psib$ and $\nub \in \Zcal_c$, $\psib \in \X_c^{\infty}$.
We can simplify the terms involving $\gamma$:
\[
\Pi_c\nub' = \partial_x\Pi_c\nub = \nub'
\quadword{and}
\Pi_c\psib' = \partial_x\Pi_c\psib = 0,
\]
since $\Pi_c$ and $\partial_x$ commute by Lemma \ref{lem: Pi partial commute} and since $\Pi_c\nub = \nub$ while $\Pi_c\psib = 0$.
Then the decomposition reads
\begin{subnumcases}{}
(\ind_{\X}-\Pi_c)\Phib_c(\nub+\psib,\omega) + \gamma\psib' = 0 \label{eqn: LC ID} \\
\Pi_c\Phib_c(\nub+\psib,\omega) + \gamma\nub' = 0, \label{eqn: LC FD} 
\end{subnumcases}
and this is the problem that we will solve here.

First we address the infinite-dimensional equation \eqref{eqn: LC ID}.
Using the same notation as in Section \ref{sec: LS ID}, define
\[
\F_c^{\infty} 
\colon \X_c^{\infty} \times \Zcal_c \times \R^2 \to \Y_c^{\infty}
\colon (\psib,\nub,\omega,\gamma) \mapsto (\ind_{\X}-\Pi_c)\Phib_c(\nub+\psib,\omega) + \gamma\psib'.
\]
Since $\Pi_c\psib' = 0$ as computed above, and since 
\[
\ip{\psib'}{\nub_0}
= -\ip{\psib}{\nub_0'}
= 0
\]
by integration by parts and the identity $\nub_0'=0$ from \eqref{eqn: nub0}, we do indeed have $\psib' \in \Y_c^{\infty}$ for $\psib \in \X_c^{\infty}$.
That is, $\F_c^{\infty}$ does indeed map into $\Y_c^{\infty}$.
Next, $\F_c^{\infty}(0,0,\omega,\gamma) = 0$ for all $\omega$ and $\gamma$, and $D_{\psib}\F_c^{\infty}(0,\omega_c,0,0) = (\ind_{\X}-\Pi_c)\restr{\L_c[\omega_c]}{\X_c^{\infty}}$ is invertible.
Consequently, by the implicit function theorem, all suitably small solutions to $\F_c^{\infty}(\psib,\nub,\omega,\gamma) = 0$ have the form $\psib = \Psib_c(\nub,\omega,\gamma)$ with $\Psib_c(0,\omega,\gamma) = 0$ for all $\omega$ and $\gamma$.
Before proceeding, we note that the same proof as for part \ref{part: Psib3} of Lemma \ref{lem: Psib} (which did not rely on the gradient structure at all) yields
\begin{equation}\label{eqn: Dnub Psib again}
D_{\nub}\Psib(0,\omega_c,0)
= 0.
\end{equation}

Now we specialize to $\nub = a\nub_1^c$ and solve the finite-dimensional equation \eqref{eqn: LC FD} by studying
\[
\F_c^0(\omega,\gamma,a)
:= \Pi_c\Phib_c(a\nub_1+\Psib_c(a\nub_1^c,\omega,\gamma),\omega)-\gamma{a}\nub_2^c
= 0.
\]
Here we have used $\partial_x\nub_1^c = -\nub_2^c$.
Since $\F_c^0(\omega,\gamma,0) = 0$, as in Section \ref{sec: LS FD} we have the factorization
\[
\F_c^0(\omega,\gamma,a) = a\H_c(\omega,\gamma,a),
\qquad
\H_c(\omega,\gamma,a) := \int_0^1 D_a\F_c^0(\omega,\gamma,a\alpha) \dalpha.
\]
We solve $\H_c(\omega,\gamma,a) = 0$.

First, we compute
\[
D_a\F_c^0(\omega,\gamma,a)
= \Pi_cD_{\phib}\Phib_c(a\nub_1+\Psib_c(a\nub_1^c,\omega,\gamma),\omega)\big(\nub_1^c+D_{\nub}\Psib_c(a\nub_1^c,\omega,\gamma)\nub_1^c\big)-\gamma\nub_2^c.
\]
This, together with \eqref{eqn: Dnub Psib again} and $\Pi_c\L_c[\omega_c] = 0$, implies $\H_c(\omega_c,0,0) = 0$.
Next, differentiate with respect to $(\omega,\gamma) \in \R^2$ and write this derivative as a linear combination of partial derivatives:
\begin{equation}\label{eqn: D-omega-gamma total}
D_{(\omega,\gamma)}\H_c(\omega_c,0,0)(\omega,\gamma)
= \omega{D}_{\omega}\H_c(\omega_c,0,0) + \gamma{D}_{\gamma}\H_c(\omega_c,0,0).
\end{equation}
We compute each of these partial derivatives separately.

For $D_{\omega}\H_c$, we calculate
\[
\H_c(\omega,0,0)
= \Pi_cD_{\phib}\Phib_c(0,\omega)\big(\nub_1^c+D_{\nub}\Psib_c(0,\omega,0)\nub_1^c\big),
\]
and use the product rule and the identities $\Pi_c\L_c[\omega_c] = 0$ and $D_{\nub}\Psib_c(0,\omega_c,0) = 0$ to obtain
\begin{equation}\label{eqn: D-omega only}
D_{\omega}\H_c(\omega_c,0,0)
= \Pi_c\L_c'[\omega_c]\nub_1^c
= \ip{\L_c'[\omega_c]\nub_1^c}{\nub_1^c}\nub_1^c.
\end{equation}

Above we used the following lemma to simplify the projection calculation.

\begin{lemma}\label{lem: other transverse}
$\ip{\L_c'[\omega_c]\nub_1^c}{\nub_2^c} = 0$.
\end{lemma}

We give two proofs of this lemma in Appendix \ref{app: other transverse}, one using the gradient formulation, and one using directly the definitions of $\L_c'[\omega_c]$, $\nub_1^c$, and $\nub_2^c$.

Next, we work on $D_{\gamma}\H_c$.
Since $\Psib_c(0,\omega,\gamma) = 0$ for all $\omega$ and $\gamma$, we have
\[
\H_c(\omega_c,\gamma,0)
= \Pi_c\L_c[\omega_c]\big(\nub_1^c+D_{\nub}\Psib_c(0,\omega_c,\gamma)\nub_1^c\big)-\gamma\nub_2
= -\gamma\nub_2,
\]
thanks to $\Pi_c\L_c[\omega_c] = 0$ once again.
Thus
\begin{equation}\label{eqn: D-gamma only}
D_{\gamma}\H_c(\omega_c,0,0)
= -\nub_2^c.
\end{equation}

We combine \eqref{eqn: D-omega-gamma total}, \eqref{eqn: D-omega only}, and \eqref{eqn: D-gamma only} to find
\[
D_{(\omega,\gamma)}\H_c(\omega_c,0,0)(\omega,\gamma)
= \omega\ip{\L_c'[\omega_c]\nub_1^c}{\nub_1^c}\nub_1^c - \gamma\nub_2^c.
\]
Since $\ip{\L_c'[\omega_c]\nub_1^c}{\nub_1^c} \ne 0$ by Corollary \ref{cor: transversality}, and since $\nub_1^c$ and $\nub_2^c$ form a basis for $\Zcal_c$, we conclude that $D_{(\omega,\gamma)}\H_c(\omega_c,0,0)$ is an invertible linear operator from $\R^2$ to $\Zcal_c$.
By the implicit function theorem, for suitably small $\omega$, $\gamma$, and $a$, we can solve $\H_c(\omega,\gamma,a) = 0$ with $\omega = \Omega_c(a)$ and $\gamma = \Gamma_c(a)$ for some maps $\Omega_c$ and $\Gamma_c$ with $\Omega_c(0) = \omega_c$ and $\Gamma_c(0) = 0$.

It follows that taking
\[
\phib_c^a := a\nub_1^c + \Psib_c(a\nub_1^c,\Omega_c(a),\Gamma_c(a))
\quadword{and}
\omega_c^a := \Omega_c(a)
\]
solves $\Phib_c(\phib_c^a,\omega_c^a) + \Gamma_c(a)\phib' = 0$.
Since $\hat{\phib_c^a}(\pm1) \ne 0$, $\phib_c^a$ is nonconstant, and so by the calculation in \eqref{eqn: why LC works} we really have $\Gamma_c(a) = 0$ for all $a$.
Additionally, if we put $\psib_c(a) = \Psib_c(a\nub_1^c,\Omega_c(a),0)$, then $\psib_c(0) = 0$, and so
\[
\phib_c^a
= a\left(\nub_1^c + \int_0^1 D_a\psib_c(a\alpha) \dalpha\right),
\]
which is the representation that we want.
This concludes our second proof of Theorem \ref{thm: main}.

\section{Periodic Solutions with Symmetry}\label{sec: with symmetry}

We first work out an abstract notion of symmetry and quickly show how bifurcation unfolds in its presence.
Then we prove that mass and spring dimers actually possess such symmetries.
The point of this analysis is that when the lattice has a symmetry, the periodic traveling wave solutions can be chosen to respect that symmetry.

\subsection{Symmetry operators and their properties}

\begin{definition}\label{defn: S}
A bounded linear operator $\Scal \colon L_{\per}^2(\R^2) \to L_{\per}^2(\R^2)$ is a \defn{symmetry}  if the following hold.

\begin{enumerate}[label={\bf(\roman*)}, ref={(\roman*)}]

\item\label{part: S1}
$\G_c(\Scal\phib,\omega) = \G_c(\phib,\omega)$ for all $\phib \in H_{\per}^2(\R^2)$ and any $c \in \R$, where $\G_c$ is defined in \eqref{eqn: Gc}.

\item\label{part: S2}
$\Scal^2\phib = \phib$ for all $\phib \in L_{\per}^2(\R^2)$.

\item\label{part: S3}
$\ip{\Scal\phib}{\etab} = \ip{\phib}{\Scal\etab}$ for all $\phib$, $\etab \in L_{\per}^2(\R^2)$.

\item\label{part: S4}
$\partial_x\Scal\phib = -\Scal\phib'$ for all $\phib \in H_{\per}^2(\R^2)$.
\end{enumerate}
\end{definition}

We point out that while shift operators $S^d$ do satisfy the invariance property \ref{part: S1} above, and while $S^{\pm\pi}$ also satisfies \ref{part: S2}, shifts in general do not meet \ref{part: S3} and \ref{part: S4}.
The symmetries that we construct will not rely on shift operators.

Here are some useful properties of symmetries for our problem.

\begin{lemma}\label{lem: Scal}
Let $\Scal$ be a symmetry.

\begin{enumerate}[label={\bf(\roman*)}, ref={(\roman*)}]

\item\label{part: Scal1}
$\Phib_c(\Scal\phib,\omega) = \Scal\Phib_c(\phib,\omega)$ for all $\phib \in H_{\per}^2(\R^2)$ and $\omega \in \R$.

\item\label{part: Scal2}
$\L_c[\omega]\Scal = \Scal\L_c[\omega]$ and $\L_c'[\omega]\Scal = \Scal\L_c'[\omega]$ for all $\omega$.

\item\label{part: Scal3}
$\Scal\nub_1^c = \pm\nub_1^c$ if and only if $\Scal\nub_2^c = \mp\nub_2^c$.
\end{enumerate}
\end{lemma}

\begin{proof}

\begin{enumerate}[label={\bf(\roman*)}]

\item
Since $\G_c(\Scal\phib,\omega) = \G_c(\phib,\omega)$ for all $\phib \in H_{\per}^2(\R^2)$ and $\omega \in \R$, we differentiate with respect to $\phib$ and use the chain rule (much as we did in the proof of part \ref{part: Phib shift invariant} of Corollary \ref{cor: shift}) to find
\[
D_{\phib}\G_c(\phib,\omega)\etab
= D_{\phib}\G_c(\Scal\phib,\omega)\Scal\etab
\]
for all $\etab \in H_{\per}^2(\R^2)$.
Using the gradient formulation, this reads
\[
\ip{\Phib_c(\phib,\omega)}{\etab}
= \ip{\Phib_c(\Scal\phib,\omega)}{\Scal\etab}
= \ip{\Scal\Phib_c(\Scal\phib,\omega)}{\etab},
\]
where the second equality is the adjoint property of $\Scal$.
Since this is true for all $\etab \in H_{\per}^2(\R^2)$, we have $\Scal\Phib_c(\Scal\phib,\omega) = \Phib_c(\phib,\omega)$.

\item
This follows from part \ref{part: Scal1} and the chain rule.

\item
We use the relations $\partial_x\nub_1^c = -\nub_2^c$ and $\partial_x\nub_2^c = \nub_1^c$ from Corollary \ref{cor: eigenfunctions}.
If $\Scal\nub_1^c = \pm\nub_1^c$, then
\[
\Scal\nub_2^c
= -\Scal\partial_x\nub_1^c
= \partial_x\Scal\nub_1^c
= \pm\partial_x\nub_1^c
= \mp\partial_x\nub_2^c.
\]
Conversely, if $\Scal\nub_2^c = \mp\nub_2^c$, then
\[
\Scal\nub_1^c
= \Scal\partial_x\nub_2^c
= -\partial_x\Scal\nub_2^c
= -(\mp\partial_x\nub_2^c)
= \pm\partial_x\nub_2^c
= \pm\nub_1^c.
\qedhere
\]
\end{enumerate}
\end{proof}

Now we adapt the nonconstant eigenfunctions $\nub_1^c$ and $\nub_2^c$ from Corollary \ref{cor: eigenfunctions} so that they respect symmetry.

\begin{lemma}\label{lem: nub-pm}
Let $\Scal$ be a symmetry and define
\[
\nub_+^c 
:= \begin{cases}
\nub_1^c, \ \Scal\nub_1^c = \nub_1^c \\
\nub_2^c, \ \Scal\nub_1^c = -\nub_1^c \\
(\nub_1^c + \Scal\nub_1^c)/\norm{\nub_1^c+\Scal\nub_1^c}_{L_{\per}^2}, \ \Scal\nub_1^c \ne \pm\nub_1^c 
\end{cases}
\]
and
\[
\nub_-^c 
:= \begin{cases}
\nub_2^c, \ \Scal\nub_1^c = \nub_1^c \\
\nub_1^c, \ \Scal\nub_1^c = -\nub_1^c \\
(\nub_2^c - \Scal\nub_2^c)/\norm{\nub_2^c-\Scal\nub_2^c}_{L_{\per}^2}, \ \Scal\nub_1^c \ne \pm\nub_1^c.
\end{cases}
\]

\begin{enumerate}[label={\bf(\roman*)}, ref={(\roman*)}]

\item
$\Scal\nub_+^c = \nub_+^c$ and $\Scal\nub_-^c = \nub_-^c$.

\item
The vectors $\nub_+^c$ and $\nub_-^c$ form an orthonormal basis for $\Zcal_c$ as defined in \eqref{eqn: Zcal-c}.

\item\label{part: nub pm3}
$\inf_{|c| > c_{\star}} \ip{\L_c'[\omega_c]\nub_+^c}{\nub_+^c} > 0$.
\end{enumerate}

\end{lemma}

\begin{proof}
We first remark that part \ref{part: Scal3} of Lemma \ref{lem: Scal} ensures that $\nub_{\pm}^c$ is defined in the third case of $\Scal\nub_1^c \ne \pm \nub_1^c$: if $\Scal\nub_1^c \ne \pm\nub_1^c$, then also $\Scal\nub_2^c \ne \pm\nub_2^c$, and so both $\nub_1^c + \Scal\nub_1^c$ and $\nub_2^c - \Scal\nub_2^c$ are nonzero.

\begin{enumerate}[label={\bf(\roman*)}]

\item
This is a direct calculation.

\item
This is obvious in the cases $\Scal\nub_1^c = \pm\nub_1^c$.
In the third case, we use part \ref{part: Scal2} of Lemma \ref{lem: Scal} to compute
\[
\L_c[\omega_c]\Scal\nub_1^c
= \Scal\L_c[\omega_c]\nub_1^c
= 0
\]
and likewise $\L_c[\omega_c]\Scal\nub_2^c = 0$.
This shows $\nub_{\pm}^c \in \ker(\L_c[\omega_c])$.
Next,
\[
\ip{\Scal\nub_1^c}{\nub_0}
= \ip{\Scal\partial_x\nub_2^c}{\nub_0}
= -\ip{\partial_x\Scal\nub_2^c}{\nub_0}
= \ip{\Scal\nub_2^c}{\partial_x\nub_0}
= 0
\]
and likewise $\ip{\Scal\nub_2^c}{\nub_0} = 0$.
This shows $\nub_{\pm}^c \in \Zcal_c$.

For orthogonality, we compute
\begin{equation}\label{eqn: nub pm aux}
\ip{\nub_1^c+\Scal\nub_1^c}{\nub_2^c-\Scal\nub_2^c}
= \ip{\nub_1^c}{\nub_2^c} - \ip{\nub_1^c}{\Scal\nub_2^c} + \ip{\Scal\nub_1^c}{\nub_2^c}-\ip{\Scal\nub_1^c}{\Scal\nub_2^c}.
\end{equation}
Now we use properties of $\Scal$ to rewrite
\[
\ip{\nub_1^c}{\Scal\nub_2^c} = \ip{\Scal\nub_1^c}{\nub_2^c}
\quadword{and}
\ip{\Scal\nub_1^c}{\Scal\nub_2^c} = \ip{\Scal^2\nub_1^c}{\nub_2^c} = \ip{\nub_1^c}{\nub_2^c}.
\]
From this and \eqref{eqn: nub pm aux}, we obtain $\ip{\nub_1^c+\Scal\nub_1^c}{\nub_2^c-\Scal\nub_2^c} = 0$.
Since $\Zcal_c$ is already two-dimensional, it follows from orthogonality and linear independence that $\nub_+^c$ and $\nub_-^c$ are a basis.

\item
The first case that $\Scal\nub_1^c=\nub_1^c$ is Corollary \ref{cor: transversality}.
The second case that $\Scal\nub_1^c = -\nub_1^c$ is equivalent to $\Scal\nub_2^c = \nub_2^c$ by part \ref{part: Scal3} of Lemma \ref{lem: Scal}, and then we can use the calculation in \eqref{eqn: transverse on nub2}.
For the third case that $\Scal\nub_1^c \ne \pm \nub_1^c$, we start by taking $\nub \in \Zcal_c = \spn(\nub_1^c,\nub_2^c)$ and then computing via the orthonormality of $\nub_1^c$ and $\nub_2^c$, \eqref{eqn: transverse on nub2}, and Lemma \ref{lem: other transverse} that
\[
\ip{\L_c'[\omega_c]\nub}{\nub}
= \norm{\nub}_{L_{\per}^2}^2\ip{\L_c'[\omega_c]\nub_1^c}{\nub_1^c}.
\]
With $\nub = \nub_+^c$, we have $\nub_+^c \in \spn(\nub_1^c,\nub_2^c)$ and $\norm{\nub_+^c}_{L_{\per}^2} = 1$, so
\[
\ip{\L_c'[\omega_c]\nub_+^c}{\nub_+^c}
= \ip{\L_c'[\omega_c]\nub_1^c}{\nub_1^c},
\]
from which the positive infimum follows.
\qedhere
\end{enumerate}
\end{proof}

\subsection{Bifurcation in the presence of symmetry}
Let $\Scal$ be a symmetry and define
\[
\Y_{\Scal} := \set{\phib \in L_{\per}^2(\R^2)}{\ip{\phib}{\nub_0} = 0 \text{ and } \Scal\phib = \phib}
\quadword{and}
\X_{\Scal} := H_{\per}^2(\R^2) \cap \Y_{\Scal}.
\]
Part \ref{part: Scal1} of Lemma \ref{lem: Scal} shows that $\Phib_c(\phib,\omega) \in \Y_{\Scal}$ for each $\phib \in \X_{\Scal}$ and $\omega \in \R$.
The effect of restricting $\Phib_c$ to map from $\X_{\Scal} \times \R$ to $\Y_{\Scal}$ is that the restriction $\restr{\L_c[\omega_c]}{\X_{\Scal}}$ now has a one-dimensional kernel and cokernel.
This, along with the transversality condition from part \ref{part: nub pm3} of Lemma \ref{lem: nub-pm}, puts us in a position to use the classical Crandall--Rabinowitz--Zeidler theorem directly, without the work in Sections \ref{sec: LS FD} or \ref{sec: LC} to manage the extra finite-dimensional equation.

\begin{remark}
While the Crandall--Rabinowitz--Zeidler theorem is often used to solve a problem of the form $F(x,\lambda) = 0$ with $F$ twice continuously differentiable, this regularity is not strictly necessary; the proof in \cite[Thm.\@ 1.5.1]{kielhofer} really hinges on having $F$ and $F_x$ once continuously differentiable.
This allows us to avoid the annoying insufficient regularity in the frequency parameter $\omega$ in our problem; recall Lemma \ref{lem: Phib regularity}.
\end{remark}

More precisely, we know that the three vectors $\nub_0$, $\nub_+^c$, and $\nub_-^c$ form an orthonormal basis for $\ker(\L_c[\omega_c])$ and $\ker(\L_c[\omega_c]^*)$; now suppose that $\phib \in \Y_{\Scal} \cap \spn(\nub_0,\nub_+^c,\nub_-^c)$.
Then by orthonormality
\[
\phib
= \ip{\phib}{\nub_0}\nub_0 + \ip{\phib}{\nub_+^c}\nub_+^c + \ip{\phib}{\nub_-^c}\nub_-^c.
\]
By definition of $\Y_{\Scal}$, we already have $\ip{\phib}{\nub_0} = 0$, and now we compute
\begin{equation}\label{eqn: symm ortho to nub-m}
\ip{\phib}{\nub_-^c}
= \ip{\Scal\phib}{\nub_-^c}
= \ip{\phib}{\Scal\nub_-^c}
= -\ip{\phib}{\nub_-^c}.
\end{equation}
Thus $\ip{\phib}{\nub_-^c} = 0$, and so $\phib \in \spn(\nub_+^c)$.
This proves our claim above that $\nub_+^c$ spans both the kernel and cokernel of $\L_c[\omega_c]$.

Alternatively, we could follow the bifurcation argument in Sections \ref{sec: LS ID} and \ref{sec: LS FD} and replace $\nub_1^c$ with $\nub_+^c$ and $\nub_2^c$ with $\nub_-^c$.
The only change would be the new version of the finite-dimensional problem \eqref{eqn: LS FD}
\begin{subnumcases}{}
\ip{\Phib_c(\phib,\omega)}{\nub_+^c} = 0 \label{eqn: ip nub+} \\
\ip{\Phib_c(\phib,\omega)}{\nub_-^c} = 0. \label{eqn: ip nub-}
\end{subnumcases}
By a calculation similar to \eqref{eqn: symm ortho to nub-m}, we always have \eqref{eqn: ip nub-}.
Specifically, for $\phib \in \X_{\Scal}$, we have
\[
\ip{\Phib_c(\phib,\omega)}{\nub_-^c}
= \ip{\Phib_c(\Scal\phib,\omega)}{\nub_-^c}
= \ip{\Scal\Phib_c(\phib,\omega)}{\nub_-^c}
= \ip{\Phib_c(\phib,\omega)}{\Scal\nub_-^c}
= -\ip{\Phib_c(\phib,\omega)}{\nub_-^c},
\]
thus $\ip{\Phib_c(\phib,\omega)}{\nub_-^c} = 0$ regardless of the form of $\phib \in \X_{\Scal}$.
(This is actually a stronger result than our managing of the second finite-dimensional equation \eqref{eqn: LS 2} in Section \ref{sec: LS FD}, as here there are no restrictions on the form of $\phib$.)
Last, we can solve \eqref{eqn: ip nub+} using the transversality condition from part \ref{part: nub pm3} of Lemma \ref{lem: nub-pm}, exactly as we did \eqref{eqn: LS 2} in Section \ref{sec: LS FD}.
The major difference in the results here is that the solutions $\phib$ now respect the symmetry.

\subsection{Existence of symmetries for the mass and spring dimers}
We will build the symmetries primarily on  a ``reflection'' operator and a ``flip'' operator.

\begin{lemma}\label{lem: R}
The operator
\begin{equation}\label{eqn: R}
(R\phib)(x)
:= \phib(-x)
\end{equation}
has the following properties.

\begin{enumerate}[label={\bf(\roman*)}, ref={(\roman*)}]

\item\label{part: R1}
$R\partial_x = -\partial_xR$.

\item\label{part: R2}
$RS^{\theta} = S^{-\theta}R$ for all $\theta \in \R$.

\item\label{part: R3}
$\ip{R\etab}{R\phib} = \ip{\phib}{\etab}$ for all $\phib$, $\etab \in L_{\per}^2(\R^2)$.
\end{enumerate}
\end{lemma}

\begin{proof}

\begin{enumerate}[label={\bf(\roman*)}]

\item
This follows from the chain rule.

\item
We compute
\[
(RS^{\theta}\phib)(x)
= (S^{\theta}\phib)(-x)
= \phib(-x+\theta)
= \phib(-(x-\theta))
= (R\phib)(x-\theta)
= (S^{-\theta}R\phib)(x).
\]

\item
This follows from substitution.
\qedhere
\end{enumerate}
\end{proof}

We will also use the ``flip'' operator
\begin{equation}\label{eqn: J}
J
:= \begin{bmatrix*}
0 &1 \\
1 &0
\end{bmatrix*},
\end{equation}
which commutes with $R$.
These reflection and flip operators also appeared in the manifestation of symmetries for spatial dynamics coordinates \cite[Sec.\@ 3.2]{faver-hupkes}.

\subsubsection{Symmetry in the mass dimer}
The mass dimer symmetry is 
\[
\Scal_{\Mb}\phib
:= -R\phib.
\]
The subscript here is meant to emphasize the role of the mass ratio $m = w^{-1}$ in the mass dimer analysis.

We show that $\Scal_{\Mb}$ satisfies property \ref{part: S1} of Definition \ref{defn: S}; all of the other properties of symmetries are quick and direct calculations.
That is, we show
\[
\G_c(\Scal_{\Mb}\phib,\omega)
= \G_c(\phib,\omega)
\]
with $\G_c = c^2\T + \P$ as defined in \eqref{eqn: Gc}.
The operator $\T$ is defined in \eqref{eqn: T} and $\P$ in \eqref{eqn: P}, and it is important here that in \eqref{eqn: P} we are assuming $\V_1 = \V_2 =: \V$.
In particular, we take $\kappa = 1$.

First we compute
\[
\frac{2}{\omega^2}\T(\Scal_{\Mb}\phib,\omega)
= \ip{-M(R\phib)''}{-R\phib}
= \ip{RM\phib''}{R\phib}
= \ip{M\phib''}{\phib}
= \frac{2}{\omega^2}\T(\phib,\omega).
\]
Here we have used $\partial_x^2R = R\partial_x^2$, which follows from part \ref{part: R1} of Lemma \ref{lem: R}, and also part \ref{part: R3} of that lemma to get the penultimate equality.

Next,
\[
\P(\Scal_{\Mb}\phib,\omega)
= \ip{\Vb(\Delta_+(\omega)\Scal_{\Mb}\phib)}{\oneb},
\]
where $\Vb(\pb) = (\V(p_1),\V(p_2))$ for $\pb = (p_1,p_2)$, $\oneb = (1,1)$, and $\Delta_+(\omega)$ is defined in \eqref{eqn: Delta-pm}.
Since $S^{\pm\omega}R = RS^{\mp\omega}$ by part \ref{part: R2} of Lemma \ref{lem: R}, we have
\begin{equation}\label{eqn: Delta+-1}
\Delta_+(\omega)\Scal_{\Mb}\phib
= -\Delta_+(\omega)R\phib
= -R\Delta_+(-\omega)\phib
= RJ\Delta_+(\omega)\phib.
\end{equation}
Here we used the property that
\begin{equation}\label{eqn: Delta+-2}
-\Delta_+(-\omega) 
= J\Delta_+(\omega)
\end{equation}
with $J$ from \eqref{eqn: J}.

Thus
\[
\P(\Scal_{\Mb}\phib,\omega)
= \ip{\Vb(RJ\Delta_+(\omega)\phib)}{\oneb}
= \ip{R\Vb(J\Delta_+(\omega)\phib)}{\oneb}.
\]
Since $\oneb$ is constant, $\oneb = R\oneb$, and so part \ref{part: R3} of Lemma \ref{lem: R} implies
\[
\P(\Scal_{\Mb}\phib,\omega)
= \ip{R\Vb(J\Delta_+(\omega)\phib)}{R\oneb}
= \ip{\Vb(J\Delta_+(\omega)\phib)}{\oneb}.
\]
Last, since $\V(\pb) = (\V(p_1),\V(p_2))$ for $\pb = (p_1,p_2)$, we have
\[
\ip{\Vb(J\pb)}{\oneb}
= \int_{-\pi}^{\pi} \big(\V(p_2)+\V(p_1)\big)
= \ip{\V(\pb)}{\oneb}.
\]
We conclude
\[
\P(\Scal_{\Mb}\phib,\omega)
= \ip{\Vb(\Delta_+(\omega)\phib)}{\oneb}
= \P(\phib,\omega).
\]

Last, we use the definitions of $\nub_1^c$ and $\nub_2^c$ from \eqref{eqn: nub1} and \eqref{eqn: nub2} to compute, assuming $\kappa=1$,
\[
\nub_1^c(x) = \frac{2\cos(x)}{\Nu_c}\begin{pmatrix*} 2\cos(\omega_c) \\ 2-c^2\omega_c^2 \end{pmatrix*}
\quadword{and}
\nub_2^c(x) = \frac{2\sin(x)}{\Nu_c}\begin{pmatrix*} 2\cos(\omega_c) \\ 2-c^2\omega_c&2 \end{pmatrix*}.
\]
This shows $\Scal_{\Mb}\nub_1^c = \nub_1^c$ and $\Scal_{\Mb}\nub_2^c = \nub_2^c$ directly.
Consequently, when we run the symmetric bifurcation argument for the mass dimer, we can just use the first case for $\nub_{\pm}^c$ in Lemma \ref{lem: nub-pm}.

\subsubsection{Symmetry in the spring dimer}
The spring dimer symmetry is 
\[
\Scal_{\Kb}
:= -RJ
= -JR
\]
with $R$ defined in \eqref{eqn: R} and $J$ defined in \eqref{eqn: J}.
The subscript is meant to emphasize the role of the linear spring coefficient ratio $\kappa$ in the spring dimer analysis.

Again, we just check that $\G_c(\Scal_{\Kb}\phib,\omega) = \G_c(\phib,\omega)$ in the case $w = m^{-1} = 1$, as the other symmetry properties from Definition \ref{defn: S} are evident.
With $\G_c = c^2\T + \P$ and $\T$ defined in \eqref{eqn: T} and $\P$ in \eqref{eqn: P}, we have
\begin{multline*}
\frac{2}{\omega^2}\T(\Scal_{\Kb}\phib,\omega)
= \ip{-(RJ\phib)''}{-RJ\phib}
= \ip{RJ\phib''}{RJ\phib}
= \ip{J\phib''}{J\phib}
= \ip{J^2\phib''}{\phib}
= \ip{\phib''}{\phib} \\
= \frac{2}{\omega^2}\T(\phib,\omega).
\end{multline*}
Here we have again used properties of $R$ from Lemma \ref{lem: R} and also $J^*= J^{-1} = J$.

Next,
\[
\P(\Scal_{\Kb}\phib,\omega)
= \ip{\Vb(\Delta_+(\omega)\Scal_{\Kb}\phib)}{\oneb},
\]
with $\Vb(\pb) = (\V_1(p_1),\V_2(p_2))$ for $\pb = (p_1,p_2)$, $\oneb = (1,1)$, and $\Delta_+(\omega)$ defined in \eqref{eqn: Delta-pm}.
We have
\[
\Delta_+(\omega)\Scal_{\Kb}\phib
= -\Delta_+(\omega)RJ
= RJ\Delta_+(\omega)J
\]
by \eqref{eqn: Delta+-1} and \eqref{eqn: Delta+-2}, and so
\[
\P(\Scal_{\Kb}\phib,\omega)
= \ip{\Vb(RJ\Delta_+(\omega)J\phib)}{\oneb}
= \ip{\Vb(J\Delta_+(\omega)J\phib)}{\oneb}.
\]
We compute
\[
J\Delta_+(\omega)J = \Lambda(\omega)\Delta_+(\omega),
\qquad
\Lambda(\omega) := \begin{bmatrix*} S^{\omega} &0 \\ 0 &S^{-\omega} \end{bmatrix*},
\]
and, for $\pb = (p_1,p_2) \in L_{\per}^2(\R^2)$,
\[
\ip{\Vb(\Lambda(\omega)\pb)}{\oneb}
= \int_{-\pi}^{\pi} \V_1(S^{\omega}p_1) + \int_{-\pi}^{\pi} \V_2(S^{-\omega}p_2)
= \int_{-\pi}^{\pi} \big(\V_1(p_1) + \V_2(p_2)\big)
= \ip{\Vb(\pb)}{\oneb}.
\]
We conclude
\[
\P(\Scal_{\Kb}\phib,\omega)
= \ip{\Vb(\Lambda(\omega)\Delta_+(\omega)\phib)}{\oneb}
= \ip{\Vb(\Delta_+(\omega)\phib)}{\oneb}
= \P(\phib,\omega).
\]

Unlike in the mass dimer, it is not always the case that $\Scal_{\Kb}\nub_1^c = \nub_1^c$ for the spring dimer.
Indeed, the situation is rather more complicated here, as we outline below.
It is for this reason that we developed Lemma \ref{lem: nub-pm}, which is unnecessarily elaborate for the mass dimer.

\begin{lemma}\label{lem: sd nub sym}
Assume $w = 1$.

\begin{enumerate}[label={\bf(\roman*)}]

\item
$\Scal_{\Kb}\nub_1^c = \nub_1^c$ if and only if $\omega_c = j\pi$ for some even $j \in \Z$.

\item
$\Scal_{\Kb}\nub_1^c = -\nub_1^c$ if and only if $\omega_c = j\pi$ for some odd $j \in \Z$.

\end{enumerate}
\end{lemma}

We prove this lemma in Appendix \ref{app: proof of lemma sd nub sym}.
A consequence is that outside the isolated situations $\omega_c = j\pi$ for some $j \in \Z$, we must use the third, more complicated case of Lemma \ref{lem: nub-pm} to obtain symmetric eigenfunctions for the spring dimer.

We conclude this discussion of symmetry by noting that not all solutions to the traveling wave problem are symmetric.
Indeed, since the traveling wave problem is shift invariant (part \ref{part: Phib shift invariant} of Corollary \ref{cor: shift}), any solution $\phib$ to $\Phib_c(\phib,\omega) = 0$ generates other solutions $S^{\theta}\phib$ for $\theta \in \R$.
Still working in the spring dimer, suppose that $\phib$ is symmetric with respect to $\Scal_{\Kb}$, so $\Scal_{\Kb}\phib = \phib$.
We compute 
\[
\Scal_{\Kb}S^{\theta}\phib
= -JRS^{\theta}\phib
= -JS^{-\theta}R\phib
= S^{-\theta}(-JR)\phib
= S^{-\theta}\Scal_{\Kb}\phib
= S^{-\theta}\phib.
\]
Typically $S^{-\theta}\phib \ne \phib$ unless $\theta$ is an even integer multiple of $\pi$.
Thus the shifted solution need not be symmetric.

\section{Quantitative Results}\label{sec: quant}

Our previous proofs have fixed the wave speed $c$ to be greater than the speed of sound and yielded families of periodic solutions parametrized in ``amplitude,'' where the range of the amplitude has been allowed to depend on $c$.
Here we develop tools to track dependence on $c$ and its variation from the speed of sound.
Such quantitative results have been essential to all of the existing nanopteron proofs that incorporate periodic solutions, and we expect the same to be necessary in any future constructions.

\subsection{An abstract quantitative bifurcation theorem}\label{sec: quantitative}
We first prove a very abstract bifurcation result from which our quantitative result for lattice periodics follows easily.
This result subsumes all of the existing quantitative periodic constructions for lattices and does not strictly depend on the long wave structure of the problem considered more broadly here.
That is, we claim that any of the prior quantitative periodic proofs follows from Theorem \ref{thm: quant} below.

We rely on the following fixed-point theorem, which was proved as \cite[Lem.\@ C.1]{faver-wright}

\begin{lemma}\label{lem: abstract fp}
For $0 < \ep < \ep_0$, let $\X^{\rho}$ be a Banach space and let $\F_{\ep} \colon \X^{\rho} \times \R \to \X^{\rho}$ be a family of maps.
Suppose that for some $C_0$, $a_0$, $b_0 > 0$, if $x$, $\grave{x} \in \X^{\rho}$ and $a \in \R$ with $\norm{x}_{\X^{\rho}}$, $\norm{\grave{x}}_{\X^{\rho}} \le b_0$ and $|a| \le a_0$, then
\begin{equation}\label{eqn: abstract fp map}
\norm{\F_{\ep}(x,a)}_{\X^{\rho}} \le C_0\big(|a| + \norm{x}_{\X^{\rho}}^2\big)
\end{equation}
and
\begin{equation}\label{eqn: abstract fp lip}
\norm{\F_{\ep}(x,a)-\F_{\ep}(\grave{x},a)}_{\X^{\rho}} \le C_0\big(\norm{x}_{\X^{\rho}} + \norm{\grave{x}}_{\X^{\rho}} + |a|\big)\norm{x-\grave{x}}_{\X^{\rho}}
\end{equation}
for all $0 < \ep < \ep_0$.
Then there exist $a_1$, $r_1 > 0$ such that for each $|a| \le a_1$ and $0 < \ep < \ep_0$, there is a unique $x_{\ep}^a \in \X^{\rho}$ with $x_{\ep}^a = \F_{\ep}(x_{\ep}^a,a)$ and $\norm{x_{\ep}^a}_{\X^{\rho}} \le r_1$.

Moreover, suppose that there is $L_0 > 0$ such that 
\begin{equation}\label{eqn: abstract fp lip a}
\norm{\F_{\ep}(x,a)-\F_{\ep}(x,\grave{a})}_{\X^{\rho}} 
\le L_0|a-\grave{a}|
\end{equation}
for all $x \in \X^{\rho}$ with $\norm{x}_{\X^{\rho}} \le b_0$, $a$, $\grave{a} \in \R$ with $|a|$, $|\grave{a}| \le a_0$, and $0 < \ep < \ep_0$.
Then there is $L_1 > 0$ such that
\begin{equation}\label{eqn: abstract fp lip a concl}
\norm{x_{\ep}^a - x_{\ep}^{\grave{a}}}_{\X^{\rho}}
\le L_1|a-\grave{a}|
\end{equation}
for all $0 < \ep < \ep_0$ and $a$, $\grave{a} \in \R$ with $|a|$, $|\grave{a}| \le a_1$.
\end{lemma}

Here is our primary abstract result.
It is very technical.
We discuss the application of this result to our long wave problem in Section \ref{sec: app of abstract quant} below, but for now we encourage the reader to think of the map $\Phib_{\ep}$ in the theorem as $\Phib_{\cep}$ from \eqref{eqn: Phib} with $\cep^2 = c_{\star}^2+\ep^2$ and to think of the spaces $\X^r$ below as $\set{\phib \in H_{\per}^r(\R^2)}{\ip{\phib}{\nub_0} = 0}$ with $\rho = 2$.

\begin{theorem}\label{thm: quant}
Let $\{\X^r\}_{r \ge 0}$ be a family of Hilbert spaces such that $\X^{r+s}$ is continuously embedded in $\X^r$ for each $s \ge 0$.
Denote the norm on $\X^{\rho}$ by $\norm{\cdot}_{\rho}$ and the inner product on $\X^0$ by $\ip{\cdot}{\cdot}$.
Suppose that for $0 < \ep < \ep_0$ and some $\rho > 0$, there is a map
\[
\Phib_{\ep} 
\colon \X^{\rho} \times \R \to \X^0
\colon (\phib,\omega) \mapsto \Phib_{\ep}(\phib,\omega)
\]
with the following properties.

\begin{enumerate}[label={\bf(\roman*)}, ref={(\roman*)}]


\item\label{hypo: trivial}
{}[Branch of trivial solutions]
$\Phib_{\ep}(0,\omega) = 0$ for all $\omega \in \R$.

\item\label{hypo: regularity}
{}[Regularity]
The partial derivatives $D_{\phib}\Phib_{\ep}$ and $D_{\phib\phib}^2\Phib_{\ep}$ exist and are continuous from $\X^{\rho} \times \R$ to $\X^0$, and the partial derivative $D_{\phib\omega}^2\Phib_{\ep}(0,\cdot)$ exists and is continuous on $\R$.

\item\label{hypo: (co)kernel}
{}[(Co)kernel dimensionality]
There is $\omega_{\ep} \in \R$ such that
\[
\ker(D_{\phib}\Phib_{\ep}(0,\omega_{\ep})) = \spn(\nub_1^{\ep},\nub_2^{\ep})
\quadword{and}
\ker(D_{\phib}\Phib_{\ep}(0,\omega_{\ep})^*) = \spn(\mub_1^{\ep},\mub_2^{\ep})
\]
for some vectors $\nub_1^{\ep}$, $\nub_2^{\ep}$, $\mub_1^{\ep}$, $\mub_2^{\ep} \in \X^0$ with $\norm{\nub^{\ep}}_0 = \norm{\mub_1^{\ep}}_0 = 1$ (the cases $\nub_2^{\ep} = 0$ and $\mub_2^{\ep} = 0$ are allowed).

\item\label{hypo: transversality}
{}[Uniform transversality]
$\inf_{0 < \ep < \ep_0} |\ip{D_{\phib\omega}^2\Phib_{\ep}(0,\omega_{\ep})\nub^{\ep}}{\mub_1^{\ep}}| > 0$.

\item\label{hypo: coercivity}
{}[Uniform coercivity]
For each $r \ge 0$, there is $C_r > 0$ such that if $\psib \in \X^{r+\rho}$ and $\etab \in \X^r$ with 
\[
D_{\phib}\Phib_{\ep}(0,\omega_{\ep})\psib = \etab,
\quad
\ip{\psib}{\nub^{\ep}} = \ip{\psib}{\mub_1^{\ep}} = \ip{\psib}{\mub_2^{\ep}} = 0,
\quad
\text{and}
\quad
\ip{\etab}{\nub^{\ep}} = \ip{\etab}{\mub_1^{\ep}} = \ip{\etab}{\mub_2^{\ep}} = 0,
\]
then $\norm{\psib}_{r+\rho} \le C_r\norm{\etab}_r$.

\item\label{hypo: bootstrapping}
{}[Bootstrapping]
If $\phib \in \X^{\rho}$ such that $D_{\phib}\Phib_{\ep}(0,\omega_{\ep})\phib \in \X^r$ for some $r \ge 0$, then $\phib \in \X^{r+\rho}$.

\item\label{hypo: uniform estimates}
{}[Uniform mapping and Lipschitz estimates]
There is $b_0 > 0$ such that the following estimates hold for each $r \ge 0$ (not necessarily uniformly in $r$):
\[
\sup_{0 < \ep < \ep_0} 
\norm{D_{\phib\omega}^2\Phib_{\ep}(0,\omega_{\ep})}_{\X^{r+\rho} \to \X^r} 
< \infty,
\]

\[
\sup_{\substack{0 < \ep < \ep_0 \\ |\omega-\omega_{\ep}| < b_0, \ |\grave{\omega}-\omega_{\ep}| < b_0 \\ \omega \ne \grave{\omega}}}
\frac{\norm{D_{\phib\omega}^2\Phib_{\ep}(0,\omega)-D_{\phib\omega}^2\Phib_{\ep}(0,\grave{\omega})}_{\X^{r+\rho} \to \X^r}}{|\omega-\grave{\omega}|}
< \infty,
\]

\[
\sup_{\substack{0 < \ep < \ep_0 \\ \norm{\phib}_{r+\rho} + |\omega-\omega_{\ep}| < b_0}} 
\norm{D_{\phib\phib}^2\Phib_{\ep}(\phib,\omega)}_{\X^{r+\rho} \times \X^{r+\rho} \to \X^r} 
< \infty,
\]
and

\[
\sup_{\substack{0 < \ep < \ep_0 \\ \norm{\phib}_{r+\rho} + |\omega-\omega_{\ep}| < b_0, \ \norm{\grave{\phib}}_{r+\rho} + |\grave{\omega}-\omega_{\ep}| < b_0 \\ (\phib,\omega) \ne (\grave{\phib},\grave{\omega})}}
\frac{\norm{D_{\phib\phib}^2\Phib_{\ep}(\phib,\omega)-D_{\phib\phib}^2\Phib_{\ep}(\grave{\phib},\grave{\omega})}_{\X^{r+\rho} \times \X^{r+\rho} \to \X^r}}{\norm{\phib-\grave{\phib}}_{r+\rho} + |\omega-\grave{\omega}|}
< \infty.
\]

\item\label{hypo: miracle}
If $\mub_2^{\ep} \ne 0$, then there are a Banach space $\W_{\ep}$ with $\X^{\rho} \subseteq \W_{\ep} \subseteq \X^0$ and a nonzero linear operator $\T_{\ep} \colon \W_{\ep} \to \X^0$ with the following properties.

\begin{enumerate}[label=$\bullet$]

\item
$\ip{\Phib_{\ep}(\phib,\omega)}{\T_{\ep}\phib} = 0$ for all $\phib \in \X^{\rho}$ and $\omega \in \R$.

\item
$\T_{\ep}\nub^{\ep} = \mub_1^{\ep}$.

\item
There is $\tau_{\ep} \in \R\setminus\{0\}$ such that $\T_{\ep}\mub_1^{\ep} = \pm\tau_{\ep}\mub_2^{\ep}$ and $\T_{\ep}\mub_2^{\ep} = \mp\tau_{\ep}\mub_1^{\ep}$.
\end{enumerate}
\end{enumerate}

Then there is $a_{\star} > 0$ such that for $|a| < a_{\star}$ and $0 < \ep < \ep_0$, there exist $\phib_{\ep}^a \in \cap_{r=0}^{\infty} \X^r$ and $\omega_{\ep}^a \in \R$ such that $\Phib_{\ep}(\phib_{\ep}^a,\omega_{\ep}^a) = 0$ with
\[
\phib_{\ep}^a = a(\nub^{\ep}+\psib_{\ep}^a), \qquad \ip{\psib_{\ep}^a}{\nub^{\ep}} = 0, \qquad \psib_{\ep}^0 = 0
\quadword{and}
\omega_{\ep}^a = \omega_{\ep} + \xi_{\ep}^a, \qquad \xi_{\ep}^0 = 0.
\]
The following mapping and Lipschitz estimates also hold for each $r$ (not necessarily uniformly in $r$):
\begin{equation}\label{eqn: psib xi map lip}
\sup_{\substack{0 < \ep < \ep_0 \\ |a| < a_{\star}}} \norm{\psib_{\ep}^a}_r + |\xi_{\ep}^a| < \infty
\quadword{and}
\sup_{\substack{0 < \ep < \ep_0 \\ |a| < a_{\star}, \ |\grave{a}| < a_{\star} \\ a \ne \grave{a}}} \frac{\norm{\psib_{\ep}^a-\psib_{\ep}^{\grave{a}}}_r}{|a-\grave{a}|} + \frac{|\xi_{\ep}^a-\xi_{\ep}^{\grave{a}}|}{|a-\grave{a}|} < \infty.
\end{equation}
\end{theorem}

\begin{proof}
We break the proof into several steps.

\begin{enumerate}[label={\bf\arabic*.}]

\item
{\it{The Lyapunov--Schmidt reduction.}}
Since $\Phib_{\ep}(0,\omega) = 0$ for all $\omega$ by Hypothesis \ref{hypo: trivial} and $D_{\phib\phib}^2\Phib_{\ep}$ exists and is continuous on $\X^{\rho} \times \R$ by Hypothesis \ref{hypo: regularity}, the fundamental theorem of calculus implies
\[
\Phib_{\ep}(\phib,\omega)
= D_{\phib}(0,\omega)\phib + \int_0^1\int_0^1 tD_{\phib\phib}^2\Phib_{\ep}(st\phib,\omega)[\phib,\phib] \ds \dt
\]
for all $\phib \in \X^{\rho}$ and $\omega \in \R$.
Next, since $D_{\phib\omega}^2\Phib_{\ep}(0,\cdot)$ exists and is continuous on $\R$ by Hypothesis \ref{hypo: regularity} again, another application of the fundamental theorem of calculus yields
\[
D_{\phib}\Phib_{\ep}(0,\omega+\xi)
= D_{\phib}\Phib_{\ep}(0,\omega) + \xi{D}_{\phib\omega}^2\Phib_{\ep}(0,\omega) + \int_0^1 \xi\big(D_{\phib\omega}^2\Phib_{\ep}(0,\omega+t\xi)-D_{\phib\omega}^2\Phib_{\ep}(0,\omega)\big) \dt
\]
for all $\omega$, $\xi \in \R$.
Together with $D_{\phib}\Phib_{\ep}(0,\omega_{\ep})\nub^{\ep} = 0$ from Hypothesis \ref{hypo: (co)kernel}, these two expansions give
\begin{equation}\label{eqn: abstract Phib expn}
\Phib_{\ep}(a(\nub^{\ep}+\psib),\omega_{\ep}+\xi)
= aD_{\phib}\Phib_{\ep}(0,\omega_{\ep})\psib + a\xi{D}_{\phib\omega}^2\Phib_{\ep}(0,\omega_{\ep})\nub^{\ep}-a\rhs_{\ep}(\psib,\xi,a)
\end{equation}
for all $\psib \in \X^{\rho}$ with $\ip{\psib}{\nub_1^{\ep}} = \ip{\psib}{\nub_2^{\ep}} = 0$ and $\xi$, $a \in \R$, where
\begin{equation}\label{eqn: abstract R}
\begin{aligned}
\rhs_{\ep}(\psib,\xi,a)
&:= -\xi{D}_{\phib\omega}^2\Phib_{\ep}(0,\omega_{\ep})\psib \\
&-a\int_0^1\int_0^1 tD_{\phib\phib}^2\Phib_{\ep}(st{a}(\nub^{\ep}+\psib),\omega_{\ep}+\xi)[\nub^{\ep}+\psib,\nub^{\ep}+\psib] \ds \dt \\
&-\xi\int_0^1 \big(D_{\phib\omega}^2\Phib_{\ep}(0,\omega_{\ep}+t\xi)-D_{\phib\omega}^2\Phib_{\ep}(0,\omega_{\ep})\big)(\nub^{\ep}+\psib) \dt.
\end{aligned}
\end{equation}

We will use the expansion \eqref{eqn: abstract Phib expn} to obtain a pair of fixed-point equations for $\psib$ and $\xi$.
Put
\begin{equation}\label{eqn: Pi-ep}
\Pi_{\ep}\phib
:= \ip{\phib}{\mub_1^{\ep}}\mub_1^{\ep} + \ip{\phib}{\mub_2^{\ep}}\mub_2^{\ep}.
\end{equation}
Then $\Phib_{\ep}(a(\nub^{\ep}+\psib),\omega_{\ep}+\xi) = 0$ if and only if
\begin{subnumcases}{}
(\ind_{\X^0}-\Pi_{\ep})\Phib_{\ep}(a(\nub^{\ep}+\psib),\omega_{\ep}+\xi) = 0 \label{eqn: abstract ID} \\
\Pi_{\ep}\Phib_{\ep}(a(\nub^{\ep}+\psib),\omega_{\ep}+\xi) = 0. \label{eqn: abstract FD}
\end{subnumcases}

\item
{\it{The preliminary equation for $\psib$.}}
It follows from the expansion \eqref{eqn: abstract Phib expn} that \eqref{eqn: abstract ID} is equivalent to
\begin{equation}\label{eqn: abstract inf dim}
(\ind_{\X^0}-\Pi_{\ep})D_{\phib}\Phib_{\ep}(0,\omega_{\ep})\psib 
= (\ind_{\X^0}-\Pi_{\ep})\big(\rhs_{\ep}(\psib,\xi,a)-\xi{D}_{\phib\omega}^2\Phib_{\ep}(0,\omega_{\ep})\psib\big).
\end{equation}
Put
\begin{equation}\label{eqn: X-ep}
\X_{\ep}^{\infty} := \set{\psib \in \X^{\rho}}{\ip{\psib}{\nub_1^{\ep}} = \ip{\psib}{\nub_2^{\ep}} = 0}
\quadword{and}
\Y_{\ep}^{\infty} := (\ind_{\X^0}-\Pi_{\ep})(\X^0).
\end{equation}
Since, by Hypothesis \ref{hypo: (co)kernel}, $D_{\phib}\Phib_{\ep}(0,\omega_{\ep})$ has trivial kernel on $\X_{\ep}^{\infty}$ and trivial cokernel in $\Y_{\ep}^{\infty}$, for each $\etab \in \Y_{\ep}^{\infty}$, there is a unique $\psib \in \X_{\ep}^{\infty}$ such that $D_{\phib}(0,\omega_{\ep})\psib = \etab$.
We write $\psib := D_{\phib}(0,\omega_{\ep})^{-1}\etab$.
With this notation, \eqref{eqn: abstract inf dim} is equivalent to
\begin{equation}\label{eqn: psib fp1}
\psib
= D_{\phib}\Phib_{\ep}(0,\omega_{\ep})^{-1}(\ind_{\X^0}-\Pi_{\ep})\big(\rhs_{\ep}(\psib,\xi,a)-\xi{D}_{\phib\omega}^2\Phib_{\ep}(0,\omega_{\ep})\nub^{\ep}\big).
\end{equation}
This is our preliminary fixed-point equation for $\psib$, but it will need some subsequent modification, as the term $\xi{D}_{\phib\omega}^2\Phib_{\ep}(0,\omega_{\ep})\nub^{\ep}$ is formally $\O(1)$ in $\xi$ and thus not suitably small for contractive purposes.

\item
{\it{The preliminary equation for $\xi$.}}
We now turn our attention to the second, finite-dimensional equation \eqref{eqn: abstract FD}.
From Hypothesis \ref{hypo: (co)kernel}, we have
\begin{equation}\label{eqn: abstract ortho1}
\ip{D_{\phib}\Phib_{\ep}(0,\omega_{\ep})\psib}{\mub_j^{\ep}}
= \ip{\psib}{D_{\phib}\Phib_{\ep}(0,\omega_{\ep})^*\mub_j^{\ep}}
= 0, \ j = 1, \ 2,
\end{equation}
If $\mub_2^{\ep} \ne 0$, the argument in Appendix \ref{app: other transverse via gradient} that proved Lemma \ref{lem: other transverse} can be adapted (take $\T_{\ep} = \partial_x$) using the properties of $\T_{\ep}$ in Hypothesis \ref{hypo: miracle} to show
\begin{equation}\label{eqn: abstract ortho2}
\ip{D_{\phib\omega}^2\Phib_{\ep}(0,\omega_{\ep})\nub^{\ep}}{\mub_2^{\ep}} 
= 0.
\end{equation}
The calculations \eqref{eqn: abstract ortho1} and \eqref{eqn: abstract ortho2} then imply that \eqref{eqn: abstract FD} is equivalent to the two equations
\begin{subnumcases}{}
\xi\ip{D_{\phib\omega}^2\Phib_{\ep}(0,\omega_{\ep})\nub^{\ep}}{\mub_1^{\ep}} = \ip{\rhs_{\ep}(\psib,\xi,a)}{\mub_1^{\ep}} \label{eqn: abstract pre xi} \\
\ip{\rhs_{\ep}(\psib,\xi,a)}{\mub_2^{\ep}} = 0 \label{eqn: abstract extra}.
\end{subnumcases}
Hypothesis \ref{hypo: transversality} implies that \eqref{eqn: abstract pre xi} is equivalent to
\begin{equation}\label{eqn: xi fp1}
\xi = \P_{\ep}\rhs_{\ep}(\psib,\xi,a),
\qquad
\P_{\ep}\etab := \frac{\ip{\etab}{\mub_1^{\ep}}}{\ip{D_{\phib\omega}^2\Phib_{\ep}(0,\omega_{\ep})\nub^{\ep}}{\mub_1^{\ep}}}
\end{equation}
This is our preliminary fixed-point equation for $\xi$, but, like the preliminary equation for $\psib$, it too needs some adjustment.
The problem here is that estimates on $\rhs_{\ep}(\psib,\xi,a)$ in $\norm{\cdot}_r$ will depend on estimates in $\psib$ in $\norm{\cdot}_{r+\rho}$, and so we will not get estimates within the same norm for contractive purposes.

\item
{\it{The final fixed-point system.}}
Put
\begin{equation}\label{eqn: abstract Psib}
\Psib_{\ep}(\psib,\xi,a)
:= D_{\phib}\Phib_{\ep}(0,\omega_{\ep})^{-1}(\ind_{\X^0}-\Pi_{\ep})\big[\rhs_{\ep}(\psib,\xi,a)-\big(\P_{\ep}\rhs_{\ep}(\psib,\xi,a)\big)D_{\phib\omega}^2\Phib_{\ep}(0,\omega_{\ep})\nub^{\ep}\big]
\end{equation}
and
\begin{equation}\label{eqn: abstract Xi}
\Xi_{\ep}(\psib,\xi,a)
:= \P_{\ep}\rhs_{\ep}(\Psib_{\ep}(\psib,\xi,a),\xi,a).
\end{equation}
Then the preliminary fixed-fixed point equations \eqref{eqn: psib fp1} and \eqref{eqn: xi fp1} for $\psib \in \X_{\ep}^{\infty}$ and $\xi \in \R$ are equivalent to 
\begin{equation}\label{eqn: abstract fp final}
\begin{cases}
\psib = \Psib_{\ep}(\psib,\xi,a) \\
\xi = \Xi_{\ep}(\psib,\xi,a),
\end{cases}
\end{equation}
and this system will turn out to have the right contraction estimates.

\item
{\it{Solving the third equation \eqref{eqn: abstract extra}.}}
Before we solve \eqref{eqn: abstract fp final} with a quantitative contraction mapping argument that is uniform in $\ep$ and $a$, we need to be sure that solutions $\psib$ and $\xi$ to \eqref{eqn: abstract fp final} really do yield solutions to our original problem $\Phib_{\ep}(a(\nub^{\ep}+\psib),\omega_{\ep}+\xi) = 0$.
That is, we need to show that solutions to \eqref{eqn: abstract fp final} also meet the third equation \eqref{eqn: abstract extra}.
Certainly this third equation is met if $\mub_2^{\ep} = 0$, so assume $\mub_2^{\ep} \ne 0$ and invoke Hypothesis \ref{hypo: miracle}.

We first redo the proof of Lemma \ref{lem: Pi partial commute} with $\partial_x$ replaced by $\T_{\ep}$ to show that $\T_{\ep}$ and $\Pi_{\ep}$ commute.
Next, we use the equivalence of \eqref{eqn: abstract Psib} and \eqref{eqn: abstract inf dim} to replicate the calculation in \eqref{eqn: the miracle calculation} and conclude that if $\psib = \Psib_{\ep}(\psib,\xi,a)$, then
\[
\pm{a}\tau_{\ep}\ip{\Phib_{\ep}(a(\nub^{\ep}+\psib),\omega_{\ep}+\xi)}{\mub_2^{\ep}}
= 0.
\]
Since $\tau_{\ep} \ne 0$, we have 
\begin{equation}\label{eqn: abstract ortho3}
\ip{\Phib_{\ep}(a(\nub^{\ep}+\psib),\omega_{\ep}+\xi)}{\mub_2^{\ep}}
= 0
\end{equation}
for all $\psib$, $\xi$, and $a \ne 0$ with $\psib = \Psib_{\ep}(\psib,\xi,a)$.
Finally, for $a \ne 0$, by \eqref{eqn: abstract Phib expn} we have
\[
\rhs_{\ep}(\psib,\xi,a)
= D_{\phib}\Phib_{\ep}(0,\omega_{\ep})\psib + \xi{D}_{\phib\omega}^2\Phib_{\ep}(0,\omega_{\ep})\nub^{\ep}-a^{-1}\Phib_{\ep}(a(\nub^{\ep}+\psib),\omega_{\ep}+\xi).
\]
Combining \eqref{eqn: abstract ortho1}, \eqref{eqn: abstract ortho2}, and \eqref{eqn: abstract ortho3} yields $\ip{\rhs_2^{\ep}(\psib,\xi,a)}{\mub_2^{\ep}} = 0$.
This is \eqref{eqn: abstract extra}.

\item
{\it{Applying Lemma \ref{lem: abstract fp}.}}
To solve the fixed-point problem \eqref{eqn: abstract fp final} and consequently our original problem, we will apply this lemma to the family of maps
\[
\F_{\ep}
\colon (\X_{\ep}^{\infty} \times \R) \times \R \to \X_{\ep}^{\infty} \times \R
\colon (\psib,\xi,a) \mapsto \big(\Psib_{\ep}(\psib,\xi,a),\Xi_{\ep}(\psib,\xi,a)\big)
\]
with $\X_{\ep}^{\infty}$ defined in \eqref{eqn: X-ep}.
We put $\norm{(\psib,\xi)}_r := \norm{\psib}_r + |\xi|$.

All of our estimates for $\F_{\ep}$ are ultimately based on estimates for $\rhs_{\ep}$ from \eqref{eqn: abstract R}.
We provide these estimates in the arbitrary norm $\norm{\cdot}_r$ for the sake of ``bootstrapping'' later.
Let $b_0$ be as in Hypothesis \ref{hypo: uniform estimates} and $r \ge 0$.
The estimates from that hypothesis provide $C_r > 0$ such that, if $0 < \ep < \ep_0$, $\norm{\psib}_{r+\rho}$, $\norm{\grave{\psib}}_{r+\rho}$, $|a|$, $|\grave{a}| \le b_0/2$ and $|\xi|$, $|\grave{\xi}| \le b_0$, the following mapping and Lipschitz estimates hold:
\[
\norm{\rhs_{\ep}(\psib,\xi,a)}_r
\le C_r\big(\norm{\psib}_{r+\rho}^2 + |\xi|^2 + |a|\big),
\]
\[
\norm{\rhs_{\ep}(\psib,\xi,a)-\rhs_{\ep}(\grave{\psib},\grave{\xi},a)}_r
\le C_r\big(\norm{\psib}_{r+\rho} + \norm{\grave{\psib}}_{r+\rho} + |\xi| + |\grave{\xi}| + |a|\big)\big(\norm{\psib-\grave{\psib}}_{r+\rho} + |\xi-\grave{\xi}|\big),
\]
and
\[
\norm{\rhs_{\ep}(\psib,\xi,a)-\rhs_{\ep}(\psib,\xi,\grave{a})}_r 
\le C_r|a-\grave{a}|.
\]

With $r=0$, the transversality estimate from Hypothesis \ref{hypo: transversality} and the ``smoothing'' estimate from Hypothesis \ref{hypo: coercivity} then imply 
\begin{equation}\label{eqn: Psib mapping}
\norm{\Psib_{\ep}(\psib,\xi,a)}_{\rho}
\le C\big(\norm{\psib}_{\rho}^2 + |\xi|^2 + |a|\big),
\end{equation}
\begin{equation}\label{eqn: Psib lip}
\norm{\Psib_{\ep}(\psib,\xi,a)-\Psib_{\ep}(\grave{\psib},\grave{\xi},a)}_{\rho}
\le C\big(\norm{\psib}_{\rho} + \norm{\grave{\psib}}_{\rho} + |\xi| + |\grave{\xi}| + |a|\big)\big(\norm{\psib-\grave{\psib}}_{\rho} + |\xi-\grave{\xi}|\big),
\end{equation}
and
\begin{equation}\label{eqn: Psib lip a}
\norm{\Psib_{\ep}(\psib,\xi,a)-\Psib_{\ep}(\psib,\xi,\grave{a})}_{\rho} 
\le C|a-\grave{a}|.
\end{equation}
Since 
\[
|\Xi_{\ep}(\psib,\xi,a)| 
\le C\norm{\rhs_{\ep}(\Psib_{\ep}(\psib,\xi,a))}_0
\]
and
\[
|\Xi_{\ep}(\psib,\xi,a)-\Xi_{\ep}(\grave{\psib},\grave{\xi},\grave{a})| 
\le C\norm{\rhs_{\ep}(\Psib_{\ep}(\psib,\xi,a))-\rhs_{\ep}(\Psib_{\ep}(\grave{\psib},\grave{\xi},\grave{a})}_0,
\]
the estimates \eqref{eqn: Psib mapping}, \eqref{eqn: Psib lip}, and \eqref{eqn: Psib lip a} hold with $\Psib_{\ep}$ replaced by $\Xi_{\ep}$ (and the norm $\norm{\cdot}_{\rho}$ on the left replaced by absolute value).
It follows that on the space $\X_{\ep}^{\infty} \times \R$, the map $\F_{\ep}$ meets the estimates \eqref{eqn: abstract fp map}, \eqref{eqn: abstract fp lip}, and \eqref{eqn: abstract fp lip a} from Lemma \ref{lem: abstract fp}.
By that lemma there are solutions $(\psib_{\ep}^a,\xi_{\ep}^a)$ meeting $(\psib_{\ep}^a,\xi_{\ep}^a) = \F_{\ep}(\psib_{\ep}^a,\xi_{\ep}^a,a)$ and the mapping and Lipschitz estimates in \eqref{eqn: psib xi map lip} for $r = \rho$.

We then ``bootstrap'' on $\psib$ using the equation $\psib = \Psib_{\ep}(\psib,\xi,a)$, the definition of $\Psib_{\ep}$ in \eqref{eqn: abstract Psib}, and Hypothesis \ref{hypo: bootstrapping} to conclude that $\psib \in \X^{n\rho}$ for any integer $n \ge 1$.
Using the estimates on $\rhs_{\ep}$ above, which are valid for any $r$, and inducting, we obtain the estimates in \eqref{eqn: psib xi map lip} for $r = n\rho$.
Interpolating, we conclude that $\psib \in \X^r$ for all $r$ and obtain the estimates in \eqref{eqn: psib xi map lip} for $r$ arbitrary.
\qedhere
\end{enumerate}
\end{proof}

\subsection{Application to periodic traveling waves in dimer FPUT}\label{sec: app of abstract quant}
For long wave solutions, we are interested in rescaling the profiles as $\phib(x) = \ep^2\varphib(\ep{x})$, where $\ep > 0$ measures the distance between the speed of sound $c_{\star}$ from \eqref{eqn: speed of sound} and the chosen wave speed $c$ via $c^2 = c_{\star}^2+\ep^2$.
In \cite{faver-wright, faver-spring-dimer}, the traveling wave problem was solved under this rescaling; here we obtain those long wave solutions as a consequence of Theorem \ref{thm: quant} by introducing a rescaling of the amplitude parameter.

Specifically, let $\ep_0 = 1$ and let $\cep$ satisfy $\cep^2 = c_{\star}^2+\ep^2$.
Let 
\[
\X^r 
:= \set{\phib \in H_{\per}^r(\R^2)}{\ip{\phib}{\nub_0} = 0}
\]
and set
\[
\nub_1^{\ep} := \nub_1^{\cep},
\qquad
\nub_2^{\ep} := \nub_1^{\cep},
\qquad
\mub_1^{\ep} := \nub_1^{\cep},
\quadword{and}
\mub_2^{\ep} := \nub_2^{\cep}.
\]
Assume now that the spring potentials satisfy $\V_1$, $\V_2 \in \Cal^{\infty}(\R)$.
Then the map $\Phib_{\cep}$ from \eqref{eqn: Phib} meets all of the hypotheses of Theorem \ref{thm: quant}.
More precisely, Hypothesis \ref{hypo: regularity} follows from Lemma \ref{lem: Phib regularity}, Hypothesis \ref{hypo: (co)kernel} from Corollary \ref{cor: eigenfunctions}, Hypothesis \ref{hypo: transversality} from Corollary \ref{cor: transversality}, and Hypothesis \ref{hypo: coercivity} from Corollary \ref{cor: coercive}.
The mapping and Lipschitz estimates in Hypothesis \ref{hypo: uniform estimates} follow from the regularity properties of shift operators in Appendix \ref{app: shift operator calculus} and composition operators in Appendix \ref{app: composition operator calculus} and the uniform bounds on $\omega_c$ in $c$ from \eqref{eqn: omegac-ineq}.

We thus obtain solutions $\phib = \phib_{\cep}^a$ and $\omega = \omega_{\cep}^a$ to $\Phib_{\cep}(\phib_{\cep}^a,\omega_{\cep}^a) = 0$ for $0 < \ep < \ep_0$ and $|a| \le a_{\per}$ for some $a_{\per} > 0$.
Returning to our original position coordinates, we see that
\[
\pb_{\ep}^a(x) 
:= a\phib_{\cep}^a(\omega_{\cep}^ax)
\]
solves the original traveling wave problem \eqref{eqn: position tw system}.

Now we expose the long wave scaling.
Write $a$ in the form $a = \alpha\ep^2$ for $|\alpha| \le a_{\per}$ and $0 < \ep < 1$; this ensures $|a| < a_{\per}$ and $0 < \ep < \ep_0$.
Put
\[
\varphib_{\ep}^{\alpha} := \phib_{\ep}^{\alpha\ep^2}
\quadword{and}
\Omega_{\ep}^{\alpha} := \frac{\omega_{\cep}^{\alpha\ep^2}}{\ep}.
\]
Then the solutions to \eqref{eqn: position tw system} have the form 
\[
\pb_{\ep}^a(x) 
= \alpha\ep^2\varphib_{\ep}^{\alpha}(\ep\Omega_{\ep}^{\alpha}x),
\]
which reveals the long wave scaling.
Moreover, these solutions have the same mapping and Lipschitz estimates previously established in \cite[Thm.\@ 4.1]{faver-wright} and \cite[Thm.\@ 3.1]{faver-spring-dimer}.
For the frequency, the mapping and Lipschitz estimates from \eqref{eqn: psib xi map lip} and the bounds on $\omega_{\cep}$ from \eqref{eqn: omegac-ineq} give
\[
\sup_{\substack{0 < \ep < \ep_0 \\ |\alpha| < a_{\per}}} |\ep\Omega_{\ep}^{\alpha}|
= \sup_{\substack{0 < \ep < \ep_0 \\ |\alpha| < a_{\per}}} |\omega_{\cep}^{\alpha\ep^2}|
< \infty
\]
and, for $0 < \ep < \ep_0$ and $|\alpha|$, $|\grave{\alpha}| < a_{\per}$
\[
|\Omega_{\ep}^{\alpha} - \Omega_{\ep}^{\grave{\alpha}}|
= \frac{|\omega_{\cep}^{\alpha\ep^2}-\omega_{\cep}^{\grave{\alpha}\ep^2}|}{\ep}
\le \frac{C|\alpha-\grave{\alpha}|\ep^2}{\ep}
= C\ep|\alpha-\grave{\alpha}|,
\]
where $C > 0$ is independent of $\ep$, $\alpha$, and $\grave{\alpha}$.
This Lipschitz estimate is an improvement on the original $\O(1)$ Lipschitz estimates from \cite[Thm.\@ 4.1]{faver-wright} and \cite[Thm.\@ 3.1]{faver-spring-dimer}.

For the profile, we first introduce the norm
\[
\norm{\phib}_{\Cal_{\per}^r}
:= \norm{\phib}_{L^{\infty}} + \norm{\partial_x^r[\phib]}_{L^{\infty}}
\]
for $r$-times continuously differentiable, $2\pi$-periodic functions $\phib$.
The mapping and Lipschitz estimates
\[
\sup_{\substack{0 < \ep < \ep_0 \\ |\alpha| < a_{\per}}} \norm{\varphib_{\ep}^{\alpha}}_{\Cal_{\per}^r} < \infty
\quadword{and}
\sup_{0 < \ep < \ep_0} \norm{\varphib_{\ep}^{\alpha} - \varphib_{\ep}^{\grave{\alpha}}}_{\Cal_{\per}^r} < C_r|\alpha-\grave{\alpha}|
\]
follow again from \eqref{eqn: psib xi map lip} and the Sobolev embedding.

\appendix
\section{Fourier Analysis}\label{app: Fourier}

\subsection{Vectors and matrices}\label{app: vectors and matrices}
The following is wholly standard, but we include it in the hopes of completeness and clarity.
For $\vb$, $\wb \in \C^n$, we set
\[
\vb\cdot\wb := \sum_{k=1}^n v_k\overline{w_k}
\quadword{and}
|\vb|_2 := \big(\vb\cdot\vb\big)^{1/2}.
\]
Also, we define $\overline{\vb} \in \C^n$ to be the vector whose entries are the conjugates of those in $\vb \in \C^n$, and likewise if $A \in \C^{m \times n}$ (where $\C^{m \times n}$ is the space of all $m \times n$ matrices with entries in $\C$), then $\overline{A} \in \C^{m \times n}$ is the matrix whose entries are the conjugates of those in $A$.
We denote by $A^* \in \C^{n \times m}$ the conjugate transpose of $A$.


For a matrix $A \in \C^{m \times n}$, we put
\[
|A|_2 = \max_{\substack{\vb \in \C^n \\ |\vb|_2 = 1}} |A\vb|_2
\quadword{and}
|A|_{\infty} = \max_{\substack{1 \le i \le m \\ 1 \le j \le n}} |A_{ij}|
\]
with $A_{ij}$ as the entries of $A$.
Then we have the inequalities
\begin{equation}\label{eqn: matrix norm ineq}
|A\vb|_2 \le |A|_2|\vb|_2, \ \vb \in \C^n,
\quadword{and}
|A|_2 \le \sqrt{mn}|A|_{\infty}.
\end{equation}

Let $I_n \in \C^{n \times n}$ be the identity matrix.
If $|A|_2 < 1$, then $I_n-A$ is invertible by the Neumann series, and 
\[
|(I_2-A)^{-1}|_2
\le \frac{1}{1-|A|_2}.
\]

\subsection{Periodic Sobolev spaces}\label{app: per sob space}
This material is developed in \cite[Sec.\@ 8.1]{kress}, \cite[]{hunter-nachtergaele}, and \cite[App.\@ C.2]{faver-dissertation}
Let $L_{\per}^2(\C^n)$ be the completion of 
\[
\Cal_{\per}^{\infty}(\C^n)
:= \set{\phi \in \Cal^{\infty}([-\pi,\pi],\C^n)}{\phi(-\pi) = \phi(\pi)}
\]
under the norm 
\[
\norm{\phib}_{L_{\per}^2(\C^n)} := \big(\ip{\phib}{\phib}_{L_{\per}^2}\big)^{1/2},
\qquad
\ip{\phib}{\etab}_{L_{\per}^2(\C^n)} := \frac{1}{2\pi}\int_{-\pi}^{\pi} \phib(x)\cdot\etab(x) \dx.
\]

For $k \in \Z$, the $k$th Fourier coefficient of $\phib \in L_{\per}^2(\C^n)$ is
\[
\hat{\phib}(k) := \frac{1}{\sqrt{2\pi}}\int_{-\pi}^{\pi} e^{-ikx}\phib(x) \dx.
\]
For $r \in \R$ and $\phib$, $\etab \in L_{\per}^2(\C^n)$, let
\[
\ip{\phib}{\etab}_{H_{\per}^r(\C^n)} := \sum_{k=-\infty}^{\infty} (1+k^2)^r\big(\hat{\phib}(k)\cdot\hat{\etab}(k)\big),
\quadword{and}
\norm{\phib}_{H_{\per}^r(\C^n)} := \big(\ip{\phib}{\phib}_{H_{\per}^r}\big)^{1/2}
\]
Finally, we put
\[
H_{\per}^r(\C^n)
:= \set{\phib \in L_{\per}^2(\C^n)}{\norm{\phib}_{H_{\per}^r(\C^n)} < \infty}.
\]

Since we will primarily use the $L_{\per}^2$-inner product, we abbreviate it as
\[
\ip{\phib}{\etab} 
:= \ip{\phib}{\etab}_{L_{\per}^2(\C^n)}
= \int_{-\pi}^{\pi} \phib(x)\cdot\etab(x)\dx.
\]
We will employ two elementary identities involving this inner product.

First, with $\phib$, $\etab \in L_{\per}^2(\C^n)$, we substitute to obtain
\begin{equation}\label{eqn: adjoint of shift}
\ip{S^{\theta}\phib}{\etab} 
= \ip{\phib}{S^{-\theta}\etab}.
\end{equation}
Second, with $\phib$, $\etab \in H_{\per}^1(\C^n)$, we integrate by parts to find
\begin{equation}\label{eqn: integration by parts}
\ip{\phib'}{\etab} 
= -\ip{\phib}{\etab'}.
\end{equation}

\subsection{Fourier multipliers}\label{app: fm}
Let $\tM \colon \R \to \C^{m \times n}$ be measurable.
A bounded linear operator $\M \colon H_{\per}^r(\C^n) \to H_{\per}^s(\C^m)$ is a Fourier multiplier with symbol $\tM$ if the identity
\[
\hat{\M\phib}(k)
= \tM(k)\hat{\phib}(k)
\]
holds for all $\phib \in H_{\per}^r(\C^n)$ and $k \in \Z$.
In this case, the operator norm of $\M$ is
\begin{equation}\label{eqn: fm op norm}
\norm{\M}_{H_{\per}^r(\C^n) \to H_{\per}^s(\C^m)}
= \sup_{k \in \Z} (1+k^2)^{(s-r)/2}|\tM(k)|_2.
\end{equation}
Conversely, if $\tM \colon \R \to \C^{m \times n}$ is such that the supremum in \eqref{eqn: fm op norm} is finite, then defining
\[
(\M\phib)(x)
:= \sum_{k=-\infty}^{\infty} e^{ikx}\tM(k)\hat{\phib}(k)
\]
gives a Fourier multiplier $\M \in \b(H_{\per}^r(\C^n),H_{\per}^s(\C^m))$ with symbol $\tM$.
This and \eqref{eqn: fm op norm} are proved in \cite[Lem.\@ D.2.1]{faver-dissertation}.

The adjoint of $\M$ is the bounded linear operator $\M^* \colon H_{\per}^s(\C^m) \to H_{\per}^r(\C^n)$ satisfying
\[
\ip{\M\phib}{\etab}_{H_{\per}^s(\C^m)}
= \ip{\phib}{\M^*\etab}_{H_{\per}^r(\C^n)}
\]
for all $\phib \in H_{\per}^r(\C^n)$ and $\etab \in H_{\per}^s(\C^m)$.
We can calculate $\M^*$ explicitly via the formula
\begin{equation}\label{eqn: fm adj}
\hat{\M^*\etab}(k)
:= (1+k^2)^{s-r}\tM(k)^*\hat{\etab}(k),
\end{equation}
where $\tM(k)^*$ is the conjugate transpose of $\tM(k)$.

\subsection{Differentiating the shift operator}\label{app: shift operator calculus}
We prove that the map
\[
\R \to \b(H_{\per}^{r+2}(\C^n),H_{\per}^r(\C^n))
\colon \omega \mapsto S^{\omega}
\]
is differentiable and that its derivative is Lipschitz continuous on $\R$.
This is proved more generally in \cite[Thm.\@ D.3.1]{faver-dissertation} for a ``scaled'' Fourier multiplier, but we include the calculation here for completeness and because all Fourier multipliers that we consider ultimately boil down to shifts.
The derivative at $\omega \in \R$ is the operator $(S^{\omega})'$ given by
\begin{equation}\label{eqn: shift derivative}
\hat{(S^{\omega})'\phib}(k)
:= k(ie^{i\omega{k}})\hat{\phib}(k).
\end{equation}
For $\phib \in H_{\per}^{r+2}(\C^n)$, $\omega \in \R$, and $h \ne 0$, we compute
\[
\longnorm{\left(\frac{S^{\omega+h}-S^{\omega}-h(S^{\omega})'}{h}\right)\phib}_{H_{\per}^r(\C^n)}^2
= \sum_{k=-\infty}^{\infty} (1+k^2)^{-2}\left|\frac{e^{ihk}-1-ihk}{h}\right|^2(1+k^2)^{r+2}|\hat{\phib}(k)|_2^2.
\]
Two applications of the fundamental theorem of calculus yield
\begin{equation}\label{eqn: two FTC}
e^{ihk}-1-ihk
= (ihk)^2\int_0^1\int_0^1 te^{ihkts} \ds \dt,
\end{equation}
from which we bound
\begin{equation}\label{eqn: k4}
(1+k^2)^{-2}\left|\frac{e^{ihk}-1-ihk}{h}\right|^2
\le \frac{(1+k^2)^{-2}h^4k^4}{h^2}\int_0^1 t \dt
= h^2\left(\frac{k^4}{2(1+k^2)^2}\right).
\end{equation}
It follows that 
\[
\longnorm{\left(\frac{S^{\omega+h}-S^{\omega}-h(S^{\omega})'}{h}\right)\phib}_{H_{\per}^r(\C^n)}^2
\le Ch^2\norm{\phib}_{H_{\per}^{r+2}(\C^n)}^2,
\]
from which we have differentiability.
The mismatch in regularity between the domain and codomain ($H_{\per}^{r+2}(\C^n)$ vs.\@ $H_{\per}^r(\C^n)$) arises because of the factor of $k^2$ in \eqref{eqn: two FTC}; squaring that $k^2$ in \eqref{eqn: k4} requires us to introduce the factor of $(1+k^2)^{-2}$ to compensate.
This agrees with the regularity requirements in \cite[Thm.\@ D.3.1]{faver-dissertation}.

Now we check Lipschitz continuity for the derivative and calculate
\[
\norm{\big((S^{\omega})'-(S^{\grave{\omega}})'\big)\phib}_{H_{\per}^r(\C^n)}^2
= \sum_{k=-\infty}^{\infty} (1+k^2)^{-2}\big|k(ie^{i\omega{k}})-k(ie^{i\grave{\omega}k})\big|^2(1+k^2)^{r+2}|\hat{\phib}(k)|_2^2.
\]
Since
\begin{equation}\label{eqn: only one k}
e^{i\omega{k}}-e^{i\grave{\omega}k}
= ik(\omega-\grave{\omega})\int_0^1 e^{i\grave{\omega}k+ik(\omega-\grave{\omega})t} \dt,
\end{equation}
we bound
\[
(1+k^2)^{-2}\big|k(ie^{i\omega{k}})-k(ie^{i\grave{\omega}k})\big|^2
\le |\omega-\grave{\omega}|^2\left(\frac{k^2}{(1+k^2)^2}\right),
\]
and this yields
\[
\norm{\big((S^{\omega})'-(S^{\grave{\omega}})'\big)\phib}_{H_{\per}^r(\C^n)}^2
\le C|\omega-\grave{\omega}|^2\norm{\phib}_{H_{\per}^{r+2}(\C^n)}^2.
\]
This is the Lipschitz continuity for the derivative.
Here we did not strictly need the domain to be $H_{\per}^{r+2}(\C^n)$ and could have viewed $S^{\omega}$ as an operator from $H_{\per}^{r+1}(\C^n)$ to $H_{\per}^r(\C^n)$, as we only have one power of $k$ emerging from \eqref{eqn: only one k}.
This too agrees with the regularity requirements in \cite[Thm.\@ D.3.1]{faver-dissertation}.

\subsection{Composition operators in periodic Sobolev spaces}\label{app: composition operator calculus}
Let $\V \in \Cal^7(\R)$ with $\V'(0) = 0$.
We briefly sketch the argument that the composition operator
\[
\Ncal
\colon H_{\per}^2(\R) \to H_{\per}^2(\R)
\colon \phi \mapsto \V'\circ\phi
\]
is well-defined and twice-differentiable, and its second derivative is (locally) Lipschitz continuous.
First, since $\V'(0) = 0$, we have
\[
\V'(r) = r\int_0^1 \V''(tr)\dt
\quadword{and therefore}
(\Ncal(\phi))(x) = \phi(x)\int_0^1 \V''(t\phi(x)) \dt.
\]
Next, differentiating under the integral, we can express $\partial_x^2[\Ncal(\phi)]$ as a sum of products of derivatives of $\phi$ up to second order and, by the periodic Sobolev embedding \cite[Thm.\@ 7.9]{hunter-nachtergaele}, continuous and periodic functions (involving integrals of the form $\medint_0^1 \V^{(k)}(t\phi) \dt$ for $k=2$, $3$, $4$).
It follows from \cite[Cor.\@ 8.8]{kress} that $\Ncal(\phi) \in H_{\per}^2(\R)$. 
Last, differentiability of $\Ncal$ is straightforward to establish using the fundamental theorem of calculus; the proof is similar to the composition operator work in \cite[Lem.\@ A.2]{friesecke-pego2}.
We obtain $D_{\phi}\Ncal(\phi)\eta = (\V''\circ\phi)\eta$ and $D_{\phi\phi}^2\Ncal(\phi)[\eta,\grave{\eta}] = (\V'''\circ\phi)\eta\grave{\eta}$, and (local) Lipschitz continuity follows from the fundamental theorem again.
For that, using $(\V'''\circ\phi)\eta\grave{\eta} = \phi\eta\grave{\eta}\medint_0^1 \V^{(4)}(t\phi) \dt$ and estimating in the $H_{\per}^2(\R)$-norm, we need up to {\it{seven}} continuous derivatives on $\V$.

If we assume $\V \in \Cal^{\infty}(\R)$, then the composition operator $\Ncal$ is also infinitely differentiable on $H_{\per}^2(\R)$ and so (more importantly, for the purposes of Theorem \ref{thm: quant}) by the Sobolev embedding all of its derivatives are locally bounded and locally Lipschitz.
This can be proved using the composition operator techniques in \cite[App.\@ B]{faver-spring-dimer}, and we omit the details.

\section{Proofs for Linear Analysis}

\subsection{The proof of Corollary \ref{cor: eigenfunctions}}\label{app: ker coker Lc-omegac}
If $\L_c[\omega_c]\phib = 0$ and $\hat{\phib}(k) \ne 0$, then by the arguments preceding the statement of Theorem \ref{thm: eigenvalues}, the scalar $c^2(\omega_ck)^2$ must be an eigenvalue of $M^{-1}\tD(\omega_ck)$, and so $c^2(\omega_ck)^2 = \tlambda_{\pm}(\omega_ck)$.
By Theorem \ref{thm: eigenvalues}, this can happen only if $k = 0$ or $k\pm1$, and so
\[
\phib(x)
= e^{-ix}\hat{\phib}(-1) + \hat{\phib}(0) + e^{ix}\hat{\phib}(1).
\]
We study each of these Fourier modes separately.
Throughout, we are assuming that at least one of $w = m^{-1}$ or $\kappa$ is greater than $1$.

\subsubsection{The eigenfunction at $k=0$}
We solve $M^{-1}\tD(0)\vb = 0$ for $\vb = (v_1,v_2)$.
By definition of $\tD$ in \eqref{eqn: tD}, we have
\[
M^{-1}\tD(0)
= \begin{bmatrix*}
(1+\kappa) &-(1+\kappa) \\
-w(\kappa+1) &w(1+\kappa)
\end{bmatrix*},
\]
and so the vector $\vb$ must be a scalar multiple of 
\[
\nub_0
:= \frac{1}{\sqrt{2}}\begin{pmatrix*}
1 \\ 1
\end{pmatrix*}.
\]
Then $\hat{\phib}(0) = a_0\nub_0$ for some $a_0 \in \R$ (since $\phib$ is real-valued, $\hat{\phib}(0)$ must be real, too).

\subsubsection{The eigenfunction at $k=1$}
We solve $M^{-1}\tD(\omega_c)\vb = c^2\omega_c^2\vb$ for $\vb = (v_1,v_2)$.
Then, using again the definition of $\tD$ in \eqref{eqn: tD}, we need
\[
\begin{bmatrix*}
(1+\kappa) &-(e^{i\omega_c}+\kappa{e}^{-i\omega_c}) \\
-w(\kappa{e}^{i\omega_c}+e^{-i\omega_c}) &w(1+\kappa)
\end{bmatrix*}
\begin{pmatrix*}
v_1 \\
v_2
\end{pmatrix*}
= c^2\omega_c^2\begin{pmatrix*}
v_1 \\ v_2
\end{pmatrix*}.
\]
The first component here reads
\[
(1+\kappa-c^2\omega_c^2)v_1-(e^{i\omega_c}+\kappa{e}^{-i\omega_c})v_2
= 0.
\]
Assume for the moment that $1+\kappa-c^2\omega_c^2 \ne 0$; we prove this below in Appendix \ref{app: the thing is nonzero}.
\[
v_1
= \frac{e^{i\omega_c}+\kappa{e}^{-i\omega_c}}{1+\kappa-c^2\omega_c^2}v_1,
\]
so $\vb$ must be a scalar multiple of 
\begin{equation}\label{eqn: mubc}
\mub_c
:= \begin{pmatrix*}
e^{i\omega_c}+\kappa{e}^{-i\omega_c} \\
1+\kappa-c^2\omega_c^2
\end{pmatrix*}.
\end{equation}
Then $\hat{\phib}(1) = a_1\mub_c$ for some $a_1 \in \C$.
(Here we are not guaranteed $a_1 \in \R$.)

\subsubsection{The proof that $1+\kappa-c^2\omega_c^2 \ne 0$}\label{app: the thing is nonzero}
We use the identity $c^2\omega_c^2 = \tlambda_+(\omega_c)$ and the definition of $\tlambda_+$ in \eqref{eqn: tlambda} to compute
\begin{equation}\label{eqn: thing nonzero aux}
1+\kappa-c^2\omega_c^2
= 1+\kappa-\tlambda_+(\omega_c)
= -\left(\frac{(1+\kappa)(w-1)+\trho(\omega_c)}{2}\right),
\end{equation}
where $\trho$ is defined in \eqref{eqn: trho}.
In particular, $\trho(\omega_c) \ge 0$.
Thus for $w > 1$, we have $1+\kappa-c^2\omega_c^2 < 0$.
When $w = 1$, and consequently $\kappa > 1$, \eqref{eqn: thing nonzero aux} simplifies to
\[
1+\kappa-c^2\omega_c^2
- \frac{\trho(\omega_c)}{2}
- \sqrt{(1-\kappa)^2 + 4\kappa\cos^2(\omega_c)}
< 1-\kappa.
\]
The resulting estimate
\begin{equation}\label{eqn: thing nonzero est}
|1+\kappa-c^2\omega_c^2|
\ge \begin{cases}
(1+\kappa)(w-1)/2, \ w > 1 \\
\kappa-1, \ w = 1
\end{cases}
\end{equation}
will be useful in subsequent proofs, since it is uniform in $c$.

\subsubsection{A basis for the kernel of $\L_c[\omega_c]$}\label{sec: kernel calc}
We are assuming $\L_c[\omega_c]\phib = 0$ and so far know that 
\[
\phib(x)
= e^{-ix}\hat{\phib}(-1) + \hat{\phib}(0) + e^{ix}\hat{\phib}(1).
\]
Since we always assume that $\phib$ is real-valued, we have
\[
\hat{\phib}(-1)
= \overline{\hat{\phib}(1)},
\]
and then
\[
\phib(x)
= \hat{\phib}(0) + 2\re[e^{ix}\hat{\phib}(1)].
\]
Write $\hat{\phib}(1) = a_1\mub_c$ with $\mub_c$ defined in \eqref{eqn: mubc} and suppose $a_1 = a_r+ia_i$ for $a_r$, $a_i \in \R$.
Then
\[
\re\big[e^{ix}\hat{\phib}(1)\big]
= \re\big[(a_r+ia_i)\big(\re[e^{ix}\mub_c] + i\im[e^{ix}\mub_c]\big)\big]
= a_r\re[e^{ix}\mub_c] - a_i\im[e^{ix}\mub_c].
\]
Thus
\[
\phib(x)
= \hat{\phib}(0) + 2\re[e^{ix}\hat{\phib}(1)]
= a_0\nub_0+a_r\re[e^{ix}\mub_c] - a_i\im[e^{ix}\mub_c].
\]

It follows that the vectors
\begin{equation}\label{eqn: nub1 nub2}
\nub_0,
\qquad
\nub_1^c(x) := \frac{1}{\Nu_c}\re[e^{ix}\mub_c],
\quadword{and}
\nub_2^c(x) := \frac{1}{\Nu_c}\im[e^{ix}\mub_c],
\qquad
\Nu_c := |\mub_c|,
\end{equation}
span the kernel of $\L_c[\omega_c]$.
We check orthonormality as follows and obtain linear independence, so they are a basis for the kernel.
First, that$\ip{\nub_0}{\nub_1^c} = \ip{\nub_0}{\nub_2^c} = 0$ follows directly from the formulas above and the identity
\[
\int_{-\pi}^{\pi} e^{\pm{ix}} \dx
= 0.
\]
Next, for any $\phib \in L_{\per}^2(\R^2)$, we compute
\begin{equation}\label{eqn: ortho equiv aux1}
\ip{\phib}{\nub_1^c} 
= 2\re\big[\hat{\phib}(1)\cdot\hat{\nub_1^c}(1)\big],
\end{equation}
\[
(S^{-\pi/2}\nub_1^c)(x)
= \frac{1}{N_c}\re[-ie^{ix}\mub_c]
= \frac{1}{N_c}\im[e^{ix}\mub_c]
= \nub_2^c(x),
\]
and
\begin{equation}\label{eqn: ortho equiv aux2}
\ip{\phib}{\nub_2^c} 
= \ip{\phib}{S^{-\pi/2}\nub_1^c} 
= 2\re\big[\hat{\phib}(1)\cdot(-i\hat{\nub_1^c}(1))\big] 
= 2\im[\hat{\phib}(1)\cdot\hat{\nub_1^c}(1)\big].
\end{equation}
Combining \eqref{eqn: ortho equiv aux1} and \eqref{eqn: ortho equiv aux2}, incidentally, proves the orthogonality equivalence condition \eqref{eqn: nub ortho equiv}.
From \eqref{eqn: ortho equiv aux2}, we have
\[
\ip{\nub_1^c}{\nub_2^c}
= 2\im[\hat{\nub_1^c}(1)\cdot\hat{\nub_1^c}(1)]
= 2\im[|\nub_1^c|_2^2]
= 0.
\]
This concludes the orthonormality proof.
Last, the derivative identities \eqref{eqn: nub derivatives} follow directly from the formulas \eqref{eqn: nub1 nub2}.

\subsubsection{The kernel of $\L_c[\omega_c]^*$}\label{app: ker Lc-omegac-star}
As discussed in Appendix \ref{app: fm}, the adjoint operator $\L_c[\omega_c]^* \colon \L_{\per}^2(\R^2) \to H_{\per}^2(\R^2)$ satisfies
\[
\hat{\L_c[\omega_c]^*\etab}(k)
= (1+k^2)^{-2}(-c^2(\omega_ck)^2M+\tD(\omega_ck)^*)\hat{\etab}(k).
\]
Here $\tD(K)^*$ is the conjugate transpose of the matrix $\tD(K)$ defined in \eqref{eqn: tD}.
Happily, $\tD(K)$ is symmetry, so $\tD(K)^* = \tD(K)$, and therefore
\begin{equation}\label{eqn: ker of Lc-omegac-star}
\hat{\L_c[\omega_c]^*\etab}(k)
= (1+k^2)^{-2}(-c^2(\omega_ck)^2M+\tD(\omega_ck))\hat{\etab}(k)
= (1+k^2)^{-2}\hat{\L_c[\omega_c]\etab}(k).
\end{equation}
Thus if $\L_c[\omega_c]\etab = 0$, then $\hat{\L_c[\omega_c]\etab}(k) = 0$ for all $k$, and so $\L_c[\omega_c]\etab = 0$.
Consequently, the kernel of $\L_c[\omega_c]^*$ is contained in the span of $\nub_0$, $\nub_1^c$, and $\nub_2^c$, and the reverse containment is also obvious from \eqref{eqn: ker of Lc-omegac-star}.

\subsection{The proof of Corollary \ref{cor: transversality}}\label{app: transversality}
We compute the exact value of $\ip{\L_c'[\omega_c]\nub_1^c}{\nub_1^c}$, where, from the definitions of $M$ in \eqref{eqn: M} and $\tD$ in \eqref{eqn: tD}, the symbol of $\L_c'[\omega_c]$ is
\[
\tL_c'(\omega_ck)
= -2c^2\omega_ckM+\tD'(\omega_ck)
= \begin{bmatrix*}
-2c^2\omega_ck &-i(e^{i\omega_ck}-\kappa{e}^{-i\omega_ck}) \\
-i(\kappa{e}^{i\omega_ck}-e^{-i\omega_ck}) &-2c^2w^{-1}\omega_ck
\end{bmatrix*}.
\]
Our goal is to use the inequality
\begin{equation}\label{eqn: transversality app}
\inf_{|c| > c_{\star}} 2c^2\omega_c-\tlambda_+'(\omega_c)
> 0
\end{equation}
from Theorem \ref{thm: eigenvalues} and recognize $\ip{\L_c'[\omega_c]\nub_1^c}{\nub_1^c}$ as the product of $2c^2\omega_c-\tlambda_+'(\omega_c)$ and a quantity that is uniformly bounded in $c$ away from $0$.

By \eqref{eqn: ortho equiv aux1}, we have
\begin{equation}\label{eqn: app transv 2}
\ip{\L_c'[\omega_c]\nub_1^c}{\nub_1^c}
= 2\re\big[\hat{\L_c'[\omega_c]\nub_1^c}{1}\cdot\hat{\nub_1^c}(1)\big]
= 2\re\big[\tL_c(\omega_c)\hat{\nub_1^c}(1)\cdot\hat{\nub_1^c}(1)\big].
\end{equation}
The formula \eqref{eqn: nub1} for $\nub_1^c$ gives
\begin{equation}\label{eqn: app transv 3}
\tL_c'(\omega_c)\hat{\nub_1^c}(1)\cdot\hat{\nub_1^c}(1)
= \frac{1}{\Nu_c^2}\begin{bmatrix*}
-2c^2\omega_c &-i(e^{i\omega_c}-\kappa{e}^{-i\omega_c}) \\
-i(\kappa{e}^{i\omega_c}-e^{-i\omega_c}) &-2c^2w^{-1}\omega_c
\end{bmatrix*}
\begin{pmatrix*} 
e^{i\omega_c}+\kappa{e}^{-i\omega_c} \\
1+\kappa-c^2\omega_c^2
\end{pmatrix*}
\cdot \begin{pmatrix*} 
e^{i\omega_c}+\kappa{e}^{-i\omega_c} \\
1+\kappa-c^2\omega_c^2
\end{pmatrix*}.
\end{equation}
Some preparation and attention to detail will simplify what would otherwise be a burdensome calculation into a slightly less burdensome calculation.
Suppressing dependence on $c$, we put
\begin{equation}\label{eqn: z and v}
z_1 = -2c^2\omega_c,
\quad
z_2 = e^{i\omega_c}-\kappa{e}^{-i\omega_c},
\quad
v_1 = e^{i\omega_c}+\kappa{e}^{-i\omega_c}, 
\quad\text{and}\quad
v_2 = 1+\kappa-c^2\omega_c^2. 
\end{equation}
In particular,
\[
\overline{z_2} = e^{-i\omega_c}-\kappa{e}^{i\omega_c}
\quadword{and so}
-i(\kappa{e}^{i\omega_c}-e^{-i\omega_c})
= i(e^{-i\omega_c}-\kappa{e}^{i\omega_c})
= i(\overline{z_2}).
\]
Then \eqref{eqn: app transv 3} is equivalent to
\begin{equation}\label{eqn: transversality real}
\begin{aligned}
\Nu_c^2\tL_c'(\omega_c)\hat{\nub_1^c}(1)\cdot\hat{\nub_1^c}(1)
&= \begin{bmatrix*}
z_1 &-iz_2 \\
i(\overline{z_2}) &w^{-1}z_1
\end{bmatrix*}
\begin{pmatrix*}
v_1 \\
v_2
\end{pmatrix*}
\cdot
\begin{pmatrix*}
v_1 \\
v_2
\end{pmatrix*} \\
&= (z_1v_1-iz_2v_2)\overline{v_1}+(i(\overline{z_2})v_1+w^{-1}z_1v_2)\overline{v_2} \\
&= z_1|v_1|^2-iz_2\overline{v_1}v_2+i(\overline{z_2})v_1\overline{v_2} + w^{-1}z_1v_2^2 \\
&= z_1\big(|v_1|^2+w^{-1}v_2^2)+\overline{i(\overline{z_2})v_1\overline{v_2}}+i(\overline{z_2})v_1v_2 \\
&= z_1\big(|v_1|^2+w^{-1}v_2^2)+2\re\big[i(\overline{z_2})v_1v_2\big].
\end{aligned}
\end{equation}
This immediately shows that $\tL_c'(\omega_c)\hat{\nub_1^c}(1)\cdot\hat{\nub_1^c}(1)$ is real, and so, after (re)introducing what turns out to be a helpful factor of $w$, \eqref{eqn: app transv 2} reads
\begin{equation}\label{eqn: app transv 4}
\frac{w\Nu_c^2\ip{\L_c'[\omega_c]\nub_1^c}{\nub_1^c}}{2}
= w\re\big[\tL_c'(\omega_c)\hat{\nub_1^c}(1)\cdot\hat{\nub_1^c}(1)\big]
= -2c^2\omega_c\big(w|v_1|^2+v_2^2)+2w\re\big[i(\overline{z_2})v_1v_2\big].
\end{equation}

The first term on the right in \eqref{eqn: app transv 4} contains a factor of $2c^2\omega_c$, which appears in our favorite estimate \eqref{eqn: transversality app}.
Now we work on the second term in \eqref{eqn: app transv 4} to expose a factor of $\tlambda_+'(\omega_c)$, which also appears in that estimate.
We have
\begin{align*}
\overline{z_2}v_1
&= (e^{-i\omega_c}-\kappa{e}^{i\omega_c})(e^{i\omega_c}+\kappa{e}^{-i\omega_c}) \\
&= 1+\kappa{e}^{-2i\omega_c}-\kappa{e}^{2i\omega_c}-\kappa^2 \\
&= 1-\kappa^2-2i\kappa\sin(2\omega_c),
\end{align*}
and so
\[
i(\overline{z_2})v_1v_2
= i(1-\kappa^2-2i\kappa\sin(2\omega_c))v_2
= i(1-\kappa^2)v_2+2\kappa\sin(2\omega_c)v_2.
\]
Thus
\[
2\re\big[i(\overline{z_2})v_1v_2\big]
= 2\re\big[i(1-\kappa^2)v_2+2\kappa\sin(2\omega_c)v_2\big]
= 4\kappa\sin(2\omega_c)v_2
\]
since $v_2 \in \R$.
By definition of $\tlambda_+$ in \eqref{eqn: tlambda}, we compute
\[
\tlambda_+'(\omega_c)
= -\frac{4\kappa{w}\sin(2\omega_c)}{\trho(\omega_c)},
\]
where $\trho$ is defined in \eqref{eqn: trho}.
Thus
\[
4\kappa{w}\sin(2\omega_c)v_2
= -\trho(\omega_c)v_2\tlambda_+'(\omega_c),
\]
and so \eqref{eqn: app transv 4} becomes
\begin{equation}\label{eqn: app transv 5}
\frac{\Nu_c^2w\ip{\L_c'[\omega_c]\nub_1^c}{\nub_1^c}}{2}
= -2c^2\omega_c\big(w|v_1|^2+v_2^2)-\trho(\omega_c)v_2\tlambda_+'(\omega_c).
\end{equation}

The first term in \eqref{eqn: app transv 5} now needs our attention.
We compute
\begin{equation}\label{eqn: transv v1}
|v_1|^2
= |e^{i\omega_c}+\kappa{e}^{-i\omega_c}|^2
= (1-\kappa)^2+4\kappa\cos^2(\omega_c).
\end{equation}
Next, in \eqref{eqn: thing nonzero aux}, we calculated
\begin{equation}\label{eqn: transv v2}
v_2
= 1+\kappa-c^2\omega_c^2
= \frac{(1+\kappa)(1-w)-\trho(\omega_c)}{2}.
\end{equation}
We use these expansions for $|v_1|^2$ and $v_2$ as well as the expansion
\[
\trho(\omega_c)^2
= (1+w)^2(1-\kappa)^2+4\kappa(1-w)^2+16\kappa{w}\cos^2(\omega_c)
\]
from the definition of $\trho$ in \eqref{eqn: trho} to compute, laboriously,
\begin{multline}\label{eqn: transv extra factor}
w|v_1|^2+v_2^2
= w(1-\kappa)^2+4\kappa{w}\cos^2(\omega_c)
+ \left(\frac{(1+\kappa)(1-w)-\trho(\omega_c)}{2}\right)^2 \\
= -\frac{\trho(\omega_c)\big[(1+\kappa)(1-w)-\trho(\omega_c)\big]}{2}
= -\trho(\omega_c)v_2.
\end{multline}

Back to \eqref{eqn: app transv 5}, we now see that
\[
\frac{w\Nu_c^2\ip{\L_c'[\omega_c]\nub_1^c}{\nub_1^c}}{2}
= -2c^2\omega_c\big(-\trho(\omega_c)v_2\big)-\trho(\omega_c)v_2\tlambda_+'(\omega_c)
= \trho(\omega_c)v_2\big(2c^2\omega_c-\lambda_+'(\omega_c)\big).
\]
That is,
\[
\ip{\L_c'[\omega_c]\nub_1^c}{\nub_1^c}
= \frac{2\trho(\omega_c)v_2}{w\Nu_c^2}\big(2c^2\omega_c-\lambda_+'(\omega_c)\big)
\]

All that remains is to check that the product $\rho(\omega_c)v_2/\Nu_c$ is uniformly bounded in $c$ away from $0$.
First, the definition of $\trho$ in \eqref{eqn: trho} implies
\[
|\trho(\omega_c)|
\ge \sqrt{(1+w)^2(1-\kappa)^2 + 4\kappa(1-w)^2}.
\]
Since at least one of $\kappa$ or $w$ is greater than $1$, this quantity is positive.
Next, the estimate \eqref{eqn: thing nonzero est} gives a positive lower bound on $v_2$ that is independent of $c$.
Finally, the definition \eqref{eqn: Nu-c} of $\Nu_c$ gives
\[
|\Nu_c|
\ge \begin{cases}
\sqrt{2}(\kappa-1), \ \kappa > 1 \\
\sqrt{2}(w-1), \ w > 1.
\end{cases}
\]
We conclude that $\trho(\omega_c)v_2/\Nu_c$ is uniformly bounded in $c$ away from $0$.

\subsection{The proof of Corollary \ref{cor: coercive}}\label{app: coercive}
Assume that $\L_c[\omega_c]\psib = \etab$, where $\psib = (\psi_1,\psi_2) \in H_{\per}^{r+2}(\R^2)$ and $\etab = (\eta_1,\eta_2) \in H_{\per}^r(\R^2)$ with
\[
\ip{\psib}{\nub_0} = \ip{\psib}{\nub_1^c} = \ip{\psib}{\nub_2^c} = 0
\quadword{and}
\ip{\etab}{\nub_0} = \ip{\etab}{\nub_1^c} = \ip{\etab}{\nub_2^c} = 0.
\]
We will solve for $\psib$ in terms of $\etab$ and uniformly estimate $\norm{\psib}_{H_{\per}^{r+2}}$ in terms of $c$ and $\norm{\etab}_{H_{\per}^r}$.
Since $\L_c[\omega_c]\psib = \etab$, we have $\tL_c(\omega_ck)\hat{\psib}(k) = \hat{\etab}(k)$ for each $k \in \Z$, and so we really need to solve
\begin{equation}\label{eqn: coercive Fourier side}
\big(-c^2\omega_c^2k^2M+\tD(\omega_ck)\big)\hat{\phib}(k)
= \hat{\etab}(k)
\end{equation}
for each $k \in \Z$, where $\tD$ is defined in \eqref{eqn: tD}.
We treat the cases $k = 0$, $k=\pm1$, and $|k| \ge 2$ separately.
This is the same strategy as the proofs of the coercive estimates for the mass dimer small mass limit \cite[Lem.\@ B.1]{hoffman-wright}, the mass dimer equal mass limit \cite[Lem.\@ C.2]{faver-hupkes-equal-mass}, and the MiM small mass limit \cite[Prop.\@ 5]{faver-mim-nanopteron}.

Before proceeding, we point out some consequences of the orthogonality conditions above for $k=0$ and $k=1$ that make the entire argument possible.
(This is essentially an exercise in solving $2\times2$ linear systems, but we need to be careful with our material parameters $w$ and $\kappa$ and our wave speed $c$.)
Suppose that $\phib = (\phi_1,\phi_2) \in L_{\per}^2(\R^2)$ with
\[
\ip{\phib}{\nub_0} 
= \ip{\phib}{\nub_1^c} 
= \ip{\phib}{\nub_2^c} 
= 0.
\]
We use these orthogonality conditions to derive formulas for $\hat{\phi}_2(k)$ in terms of $\hat{\phi}_1(k)$ for $k=0$ and $k=1$.

First, the condition $\ip{\phib}{\nub_0} = 0$ immediately implies 
\begin{equation}\label{eqn: phi2 in terms of phi1 0}
\hat{\phi}_2(0) 
= -\hat{\phi}_1(0).
\end{equation}
Next, the orthogonality condition \eqref{eqn: nub ortho equiv} implies
\[
\hat{\phib}(1)\cdot\hat{\nub_1^c}(1)
= 0,
\]
and from the definition of $\nub_1^c$ in \eqref{eqn: nub1}, this reads
\[
\hat{\phi}_1(1)(e^{-i\omega_c}+\kappa{e}^{i\omega_c}) + (1+\kappa-c^2\omega_c^2)\hat{\phi}_2(1)
= 0.
\]
Since $1+\kappa-c^2\omega_c^2 \ne 0$ by the work in Appendix \ref{app: the thing is nonzero}, we have
\begin{equation}\label{eqn: phi2 in terms of phi1}
\hat{\phi}_2(1)
= -\frac{e^{-i\omega_c}+\kappa{e}^{i\omega_c}}{1+\kappa-c^2\omega_c^2}\hat{\phi}_1(1).
\end{equation}

\subsubsection{The case $k=0$}
Here the first component of \eqref{eqn: coercive Fourier side} reads
\[
\hat{\psi}_1(0) - \hat{\psi}_2(0)
= \frac{\hat{\eta}_1(0)}{1+\kappa},
\]
and from \eqref{eqn: phi2 in terms of phi1 0} this is
\[
2\hat{\psi}_1(0)
= \hat{\eta}_1(0).
\]
Thus
\[
\hat{\psi}_1(0)
= \frac{\hat{\eta}_1(0)}{2(1+\kappa)}.
\]
It follows from this equality and \eqref{eqn: phi2 in terms of phi1 0} that
\[
|\hat{\psib}(0)|_2
\le \frac{|\hat{\etab}(0)|_2}{2(1+\kappa)}.
\]

\subsubsection{The case $k=\pm1$}
We only need to estimate $|\hat{\psib}(1)|$, as $|\hat{\psib}(1)| = |\hat{\psib}(-1)|$ since $\psib$ is $\R^2$-valued.
At $k=1$ the first component of \eqref{eqn: coercive Fourier side} reads
\[
(1+\kappa-c^2\omega_c^2)\hat{\psi}_1(1)-(e^{i\omega_c}+\kappa{e}^{-i\omega_c})\hat{\psi}_2(1)
= \hat{\eta}_1(1).
\]
We use the identity \eqref{eqn: phi2 in terms of phi1} to remove $\hat{\psi}_2(1)$ from this equation and write it in terms of $\hat{\psi}_1(1)$ alone.
We find
\[
\left(1+\kappa-c^2\omega_c^2+\frac{|e^{i\omega_c}+\kappa{e}^{-i\omega_c}|^2}{1+\kappa-c^2\omega_c^2}\right)\hat{\psi}_1(1)
= \hat{\eta}_1(1).
\]
We use \eqref{eqn: transv v1} and rearrange this into
\begin{equation}\label{eqn: the big coercive one}
\big((1+\kappa-c^2\omega_c^2)^2+(1-\kappa)^2+4\kappa\cos^2(\omega_c)\big)\hat{\psi}_1(1)
= (1+\kappa-c^2\omega_c^2)\hat{\eta}_1(1).
\end{equation}

Since
\[
(1+\kappa-c^2\omega_c^2)^2+(1-\kappa)^2+4\kappa\cos^2(\omega_c)
\ge (1+\kappa-c^2\omega_c^2)^2,
\]
the uniform lower bound on $1+\kappa-c^2\omega_c^2$ from \eqref{eqn: thing nonzero est} and the upper bound
\[
|1+\kappa-c^2\omega_c^2|
\le 1+\kappa+(1+\kappa)(1+w)
\]
from \eqref{eqn: omegac-ineq}, we can derive from \eqref{eqn: the big coercive one} the estimate
\[
|\hat{\psi}_1(1)|_2f
\le C|\hat{\eta}_1(1)|_2,
\]
where $C$ depends on $\kappa$ and $w$ but not on $c$, $\psib$, or $\etab$.
The identity \eqref{eqn: phi2 in terms of phi1} and the uniform lower bound on $1+\kappa-c^2\omega_c^2$ imply 
\[
|\hat{\psi}_2(1)|_2
\le C|\hat{\eta}_1(1)|_2
\]
 as well.
 A final invocation of \eqref{eqn: phi2 in terms of phi1} allows us to estimate $|\hat{\eta}_1(1)|_2 \le C|\hat{\etab}(1)|_2$.

\subsubsection{The case $|k| \ge 2$}
Since $k \ne 0$, we may rewrite \eqref{eqn: coercive Fourier side} as
\begin{equation}\label{eqn: Neumann}
\left(I_2 - \frac{1}{c^2\omega_c^2k^2}M^{-1}\tD(\omega_ck)\right)\hat{\psib}(k)
= -\frac{1}{c^2\omega_c^2k^2}M^{-1}\hat{\etab}(k).
\end{equation}
Here $I_2$ is the $2\times2$ identity matrix.
We will use the Neumann series to solve \eqref{eqn: Neumann} for $\hat{\psib}(k)$ in terms of $\hat{\etab}(k)$ with uniform estimates in $c$.

The following estimates use our conventions for matrix norms from Appendix \ref{app: vectors and matrices}.
First, the definition of $\tD$ in \eqref{eqn: tD} yields the estimate
\[
|M^{-1}\tD(\omega_ck)|_{\infty}
\le (1+\kappa)w.
\]
Since $|k| \ge 2$, we have
\[
\frac{1}{c^2\omega_c^2k^2}|M^{-1}\tD(\omega_ck)|_2
\le \frac{2}{c^2\omega_c^2k^2}|M^{-1}\tD(\omega_ck)|_{\infty}
\le \frac{2(1+\kappa)w}{4c^2\omega_c^2}.
\]
Next, the inequality \eqref{eqn: omegac-ineq} on $\omega_c$ and the definition of $\tlambda_+$ in \eqref{eqn: tlambda} imply
\begin{equation}\label{eqn: 1/csq omega sq}
\frac{1}{c^2\omega_c^2}
\le \frac{1}{\tlambda_+(\pi/2)}
\le \frac{2}{(1+\kappa)(1+w)}.
\end{equation}
Thus
\[
\frac{1}{c^2\omega_c^2k^2}|M^{-1}\tD(\omega_ck)|_2
\le \frac{4(1+\kappa)w}{4(1+\kappa)(1+w)}
= \frac{w}{1+w}
< 1.
\]

We may therefore use the Neumann series to solve \eqref{eqn: Neumann} for $\hat{\psib}(k)$ in terms of $\hat{\etab}(k)$, and we obtain
\[
|\hat{\psib}(k)|_2
\le \frac{1}{1-w/(1+w)}\left(\frac{|M^{-1}|_2}{c^2\omega_c^2k^2}\right)|\hat{\etab}(k)|_2
\le \frac{1}{1-w/(1+w)}\left(\frac{4w}{(1+\kappa)(1+w)}\right)\frac{|\hat{\etab}(k)|_2}{k^2}.
\]
The second inequality follows from \eqref{eqn: 1/csq omega sq} and the estimate $|M^{-1}|_2 \le 2|M^{-1}|_{\infty} = 2w$ from \eqref{eqn: matrix norm ineq}.
This, along with the uniform estimates in $c$ on $|\hat{\psib}(k)|_2$ for $k=0$, $1$ from the previous sections, gives the coercive estimate $\norm{\psib}_{H_{\per}^{r+2}} \le C\norm{\etab}_{H_{\per}^r}$.
The constant $C$ depends on $\kappa$ and $w$ but is independent of $r$.

\subsection{The proof of Lemma \ref{lem: Phib regularity}}\label{app: proof of lem Phib regularity}
Continuity and differentiability of $\Phib_c$ in $\phib$ follow from the composition operator calculus in Appendix \ref{app: composition operator calculus} and in $\omega$ from the shift operator calculus in Appendix \ref{app: shift operator calculus}.
A second appeal to these appendices gives the same results for $D_{\phib}\Phib_c$.
In each case, we are only taking one derivative with respect to $\omega$, and that is all that Appendix \ref{app: shift operator calculus} guarantees when we consider $S^{\omega}$ as a map from $H_{\per}^2(\R)$ to $H_{\per}^0 = L_{\per}^2$.

\subsection{The proof of Lemma \ref{lem: other transverse}}\label{app: other transverse}

\subsubsection{A proof using the gradient formulation}\label{app: other transverse via gradient}
We claim that 
\begin{equation}\label{eqn: Lc ortho cond}
\ip{\L_c[\omega]\phib}{\phib'} 
= 0
\end{equation}
for all $\phib \in H_{\per}^2(\R^2)$ and $\omega \in \R$ and prove this claim below.
Assuming this to be true, we differentiate \eqref{eqn: Lc ortho cond} with respect to $\omega$ and obtain
\[
\ip{\L_c'[\omega]\phib}{\phib'} 
= 0
\]
for all $\phib \in H_{\per}^2(\R^2)$ and $\omega \in \R$.
In particular,
\[
\ip{\L_c'[\omega_c]\nub_1^c}{\nub_2^c}
= -\ip{\L_c'[\omega_c]\nub_1^c}{\partial_x\nub_1^c}
= 0.
\]

Now we prove the claim \eqref{eqn: Lc ortho cond}.
The proofs of the derivative orthogonality condition $\ip{\Phib_c(\phib,\omega)}{\phib'} = 0$ in both part \ref{part: deriv ortho cond} of Corollary \ref{cor: shift} and in Lemma \ref{lem: deriv ortho} did not rely on the precise structure of the spring potentials $\V_1$ and $\V_2$, provided that they were continuously differentiable.
So, assume here that both are linear with $\V_1(r) = \V_1'(0)r$ and $\V_2(r) = \V_2'(0)r$.
Then $\Phib_c(\phib,\omega) = D_{\phib}\Phib_c(0,\omega)\phib = \L_c[\omega]\phib$, and \eqref{eqn: Lc ortho cond} follows from the original derivative orthogonality condition.

\subsubsection{A proof via direct calculation}
The same reasoning that led to \eqref{eqn: app transv 2} implies
\[
\ip{\L_c'[\omega_c]\nub_1^c}{\nub_2^c}
= 2\re\big[\tL_c'(\omega_c)\hat{\nub_1^c}(1)\cdot\hat{\nub_2^c}(1)\big].
\]
Now, from the definitions of $\nub_1^c$ in \eqref{eqn: nub1} and $\nub_2^c$ in \eqref{eqn: nub2}, we have
\[
\hat{\nub_2^c}(1)
= -i\hat{\nub_1^c}(1),
\]
and so
\[
\tL_c'(\omega_c)\hat{\nub_1^c}(1)\cdot\hat{\nub_2^c}(1)
= \tL_c'(\omega_c)\hat{\nub_1^c}(1)\cdot\big(-i\hat{\nub_1^c}(1)\big)
= i\big(\tL_c'(\omega_c)\hat{\nub_1^c}(1)\cdot\hat{\nub_1^c}(1)\big),
\]
thus
\[
\ip{\L_c'[\omega_c]\nub_1^c}{\nub_2^c}
= 2\re\big[i\big(\tL_c'(\omega_c)\hat{\nub_1^c}(1)\cdot\hat{\nub_1^c}(1)\big)\big].
\]
But in \eqref{eqn: transversality real}, we calculated that $\tL_c'(\omega_c)\hat{\nub_1^c}(1)\cdot\hat{\nub_1^c}(1)$ is real, and so $\ip{\L_c'[\omega_c]\nub_1^c}{\nub_2^c} = 0$.

\subsection{The proof of Lemma \ref{lem: sd nub sym}}\label{app: proof of lemma sd nub sym}
We have $\Scal_{\Kb}\nub_1^c = \pm\nub_1^c$ if and only if $\hat{\Scal_{\Kb}\nub_1^c}(1) = \pm\hat{\nub_1^c}(1)$ and $\hat{\Scal_{\Kb}\nub_1^c}(-1) = \pm\hat{\nub_1^c}(-1)$.
Since $\Scal_{\Kb}\nub_1^c$ and $\nub_1^c$ are real-valued, the second equality automatically holds if the first does.  
Thus $\Scal_{\Kb}\nub_1^c = \pm\nub_1^c$ if and only if $\hat{\Scal_{\Kb}\nub_1^c}(1) = \pm\hat{\nub_1^c}(1)$.
We compute
\[
\hat{\Scal_{\Kb}\nub_1^c}(1)
= -J\hat{R\nub_1^c}(1)
= -J\hat{\nub_1^c}(-1).
\]
From the definition of $\nub_1^c$ in \eqref{eqn: nub1}, where it is not at this time at all apparent that taking $w=1$ matters, we have
\[
-\Nu_cJ\hat{\nub_1^c}(-1)
= -J\begin{pmatrix*}
e^{-i\omega_c}+\kappa{e}^{i\omega_c} \\
1+\kappa-c^2\omega_c^2
\end{pmatrix*}
= -\begin{pmatrix*}
1+\kappa-c^2\omega_c^2 \\
e^{-i\omega_c}+\kappa{e}^{i\omega_c}
\end{pmatrix*}
\]
Thus $\hat{\Scal_{\Kb}\nub_1^c}(1) = \pm\hat{\nub_1^c}(1)$ if and only if
\[
-\begin{pmatrix*}
1+\kappa-c^2\omega_c^2 \\
e^{-i\omega_c}+\kappa{e}^{i\omega_c}
\end{pmatrix*}
= \pm\begin{pmatrix*}
e^{i\omega_c}+\kappa{e}^{-i\omega_c} \\
1+\kappa-c^2\omega_c^2
\end{pmatrix*},
\]
from which it follows that $\hat{\Scal_{\Kb}\nub_1^c}(1) = \pm\hat{\nub_1^c}(1)$ is equivalent to 
\[
e^{i\omega_c}+\kappa{e}^{-i\omega_c}
= \pm(1+\kappa-c^2\omega_c^2).
\]

Taking real and imaginary parts, we have $\hat{\Scal_{\Kb}\nub_1^c}(1) = \pm\hat{\nub_1^c}(1)$ if and only if
\begin{subnumcases}{}
(1+\kappa)\cos(\omega_c) = \pm(1+\kappa-c^2\omega_c^2) \label{eqn: spring dimer re aux} \\
(1-\kappa)\sin(\omega_c) = 0. \label{eqn: spring dimer im aux}
\end{subnumcases}
Since we are working with a spring dimer and $\kappa \ne 1$, \eqref{eqn: spring dimer im aux} is equivalent to $\omega_c = j\pi$ for some $j \in \Z$.

We use \eqref{eqn: thing nonzero aux} with $w = 1$ and $\omega_c = j\pi$, $j \in \Z$, and the definition of $\trho$ in \eqref{eqn: trho} to compute
\[
1+\kappa-c^2\omega_c^2
= -\frac{\trho(\omega_c)}{2}
= 1+\kappa,
\]
and so \eqref{eqn: spring dimer re aux} is equivalent to
\[
(1+\kappa)(-1)^j 
= \pm(1+\kappa).
\]
Thus $\hat{\Scal_{\Kb}\nub_1^c}(1) = \pm\hat{\nub_1^c}(1)$ if and only if $(-1)^j = \pm1$, so $\hat{\Scal_{\Kb}\nub_1^c}(1) = \hat{\nub_1^c}(1)$ if and only if $j$ is even, while $\hat{\Scal_{\Kb}\nub_1^c}(1) = -\hat{\nub_1^c}(1)$ if and only if $j$ is odd.

\bibliographystyle{siam}
\bibliography{position_bib}{}

\begin{thebibliography}{10}

\bibitem{akers-ambrose-sulon}
{\sc B.~Akers, D.~Ambrose, and D.~Sulon}, {\em Periodic traveling interfacial
  hydroelastic waves with or without mass ii: Multiple bifurcations and
  ripples}, European Journal of Applied Mathematics, 30 (2019), pp.~756--790.

\bibitem{ambrosetti-prodi}
{\sc A.~Ambrosetti and G.~Prodi}, {\em A primer of nonlinear analysis},
  Cambridge University Press, 1995.

\bibitem{amick-toland}
{\sc C.~J. Amick and J.~F. Toland}, {\em Solitary waves with surface tension.
  {I}. {T}rajectories homoclinic to periodic orbits in four dimensions}, Arch.
  Rational Mech. Anal., 118 (1992), pp.~37--69.

\bibitem{baldi-toland}
{\sc P.~Baldi and J.~F. Toland}, {\em Bifurcation and secondary bifurcation of
  heavy periodic hydroelastic travelling waves}, Interfaces Free Bound., 12
  (2010), pp.~1--22.

\bibitem{beale}
{\sc J.~T. Beale}, {\em Exact solitary water waves with capillary ripples at
  infinity}, Comm. Pure Appl. Math., 44 (1991), pp.~211--257.

\bibitem{betti-pelinovsky}
{\sc M.~Betti and D.~E. Pelinovsky}, {\em Periodic traveling waves in diatomic
  granular chains}, J. Nonlinear Sci., 23 (2013), pp.~689--730.

\bibitem{boyd}
{\sc J.~P. Boyd}, {\em Weakly Nonlocal Solitary Waves and Beyond-All-Orders
  Asymptotics}, vol.~442 of Mathematics and Its Applications, Kluwer Academic
  Publishers, Dordrecht, The Netherlands, 1998.

\bibitem{chirilus-bruckner-etal}
{\sc M.~Chirilus-Bruckner, C.~Chong, O.~Prill, and G.~Schneider}, {\em Rigorous
  description of macroscopic wave packets in infinite periodic chains of
  coupled oscillators by modulation equations}, Discrete Contin. Dyn. Syst.
  Ser. S, 5 (2012), pp.~879--901.

\bibitem{chong-kev-book}
{\sc C.~Chong and P.~G. Kevrekidis}, {\em Coherent structures in granular
  crystals: from experiment and modelling to computation and mathematical
  analysis}, Springer Briefs in Physics, Springer, Cham, 2018.

\bibitem{cpkd}
{\sc C.~Chong, M.~A. Porter, P.~G. Kevrekidis, and C.~Daraio}, {\em Nonlinear
  coherent structures in granular crystals}, Journal of Physics: Condensed
  Matter, 29 (2017), p.~413003.

\bibitem{crandall-rabinowitz}
{\sc M.~G. Crandall and P.~H. Rabinowitz}, {\em Bifurcation from simple
  eigenvalues}, J. Functional Analysis, 8 (1971), pp.~321--340.

\bibitem{dauxois}
{\sc T.~Dauxois}, {\em {F}ermi, {P}asta, {U}lam, and a mysterious lady},
  Physics Today, 61 (2008), pp.~55--57.

\bibitem{deng-sun-fput-md}
{\sc S.~Deng and S.-M. Sun}, {\em Existence of generalized solitary waves for a
  diatomic {F}ermi--pasta--{U}lam--{T}singou lattice}.
\newblock arXiv preprint arXiv:2408.02017.

\bibitem{faver-dissertation}
{\sc T.~E. Faver}, {\em Nanopteron-stegoton traveling waves in mass and spring
  dimer {F}ermi--{P}asta--{U}lam--{T}singou lattices}, PhD thesis, Drexel
  University, Philadelphia, PA, May 2018.

\bibitem{faver-spring-dimer}
{\sc T.~E. Faver}, {\em Nanopteron-stegoton traveling waves in spring dimer
  {F}ermi-{P}asta-{U}lam-{T}singou lattices}, Quarterly of Applied Mathematics,
  78 (2020), pp.~363--429.

\bibitem{faver-mim-nanopteron}
{\sc T.~E. Faver}, {\em Small mass nanopteron traveling waves in mass-in-mass
  lattices with cubic {FPUT} potential}, Journal of Dynamics and Differential
  Equations,  (Published online 7 July 2020).

\bibitem{faver-hupkes-equal-mass}
{\sc T.~E. Faver and H.~J. Hupkes}, {\em Micropteron traveling waves in
  diatomic {F}ermi--{P}asta--{U}lam--{T}singou lattices under the equal mass
  limit}, Physica D: Nonlinear Phenomena, 410 (2020), p.~132538.

\bibitem{faver-hupkes}
\leavevmode\vrule height 2pt depth -1.6pt width 23pt, {\em Micropteron
  traveling waves in diatomic {F}ermi-{P}asta-{U}lam-{T}singou lattices under
  the equal mass limit}, Physica D: Nonlinear Phenomena, 410 (2020).

\bibitem{faver-hupkes-spatial-dynamics}
\leavevmode\vrule height 2pt depth -1.6pt width 23pt, {\em Mass and spring
  dimer {F}ermi--{P}asta--{U}lam--{T}singou nanopterons with exponentially
  small, nonvanishing ripples}, Studies in Applied Mathematics, 150 (2023).

\bibitem{faver-wright}
{\sc T.~E. Faver and J.~D. Wright}, {\em Exact diatomic
  {F}ermi--{P}asta--{U}lam--{T}singou solitary waves with optical band ripples
  at infinity}, SIAM Journal on Mathematical Analysis, 50 (2018), pp.~182--250.

\bibitem{fput-original}
{\sc E.~Fermi, J.~Pasta, and S.~Ulam}, {\em Studies of nonlinear problems},
  Lect. Appl. Math., 12 (1955), pp.~143--56.

\bibitem{fml}
{\sc G.~Friesecke and A.~Mikikits-Leitner}, {\em Cnoidal waves on
  {F}ermi-{P}asta-{U}lam lattices}, J Dyn. Diff. Equat., 27 (2015).

\bibitem{friesecke-pego1}
{\sc G.~Friesecke and R.~L. Pego}, {\em Solitary waves on {FPU} lattices. {I}.
  {Q}ualitative properties, renormalization and continuum limit}, Nonlinearity,
  12 (1999), pp.~1601--1627.

\bibitem{friesecke-pego2}
\leavevmode\vrule height 2pt depth -1.6pt width 23pt, {\em Solitary waves on
  {FPU} lattices. {II}. {L}inear implies nonlinear stability}, Nonlinearity, 15
  (2002), pp.~1343--1359.

\bibitem{friesecke-wattis}
{\sc G.~Friesecke and J.~A.~D. Wattis}, {\em Existence theorem for solitary
  waves on lattices}, Comm. Math. Phys., 161 (1994), pp.~391--418.

\bibitem{gmwz}
{\sc J.~Gaison, S.~Moskow, J.~D. Wright, and Q.~Zhang}, {\em Approximation of
  polyatomic {FPU} lattices by {K}d{V} equations}, Multiscale Model. Simul., 12
  (2014), pp.~953--995.

\bibitem{hoffman-wright}
{\sc A.~Hoffman and J.~D. Wright}, {\em Nanopteron solutions of diatomic
  {F}ermi--{P}asta--{U}lam--{T}singou lattices with small mass-ratio}, Physica
  D: Nonlinear Phenomena, 358 (2017), pp.~33--59.

\bibitem{hunter-nachtergaele}
{\sc J.~K. Hunter and B.~Nachtergaele}, {\em Applied Analysis}, World
  Scientific, Hackensack, NJ, 2001.

\bibitem{kielhofer}
{\sc H.~Kielh\"{o}fer}, {\em Bifurcation theory: An introduction with
  applications to partial differential equations}, vol.~159 of Applied
  Mathematical Sciences, Springer, New York, second~ed., 2012.

\bibitem{kress}
{\sc R.~Kress}, {\em Linear Integral Equations}, vol.~82 of Applied
  Mathematical Sciences, Springer, 3~ed., 2014.

\bibitem{kromer-healey-kielhofer}
{\sc S.~Kr\"{o}mer, T.~J. Healey, and H.~Kielh\"{o}fer}, {\em Bifurcation with
  a two-dimensional kernel}, Journal of Differential Equations, 220 (2006),
  pp.~234--258.

\bibitem{liu-shi-wang}
{\sc P.~Liu, J.~Shi, and Y.~Wang}, {\em A double saddle-node bifurcation
  theorem}, Comunicatins on Pure and Applied Analysis, 12 (2013),
  pp.~2923--2933.

\bibitem{lombardi}
{\sc E.~Lombardi}, {\em Oscillatory Integrals and Phenomena Beyond all
  Algebraic Orders with Applications to Homoclinic Orbits in Reversible
  Systems}, vol.~1741 of Lecture Notes in Mathematics, Springer-Verlag Berlin
  Heidelberg, 2000.

\bibitem{pankov}
{\sc A.~Pankov}, {\em Travelling Waves and Periodic Oscillations in
  Fermi-Pasta-Ulam Lattices}, Imperial College Press, Singapore, 2005.

\bibitem{qin}
{\sc W.-X. Qin}, {\em Wave propagation in diatomic lattices}, SIAM J. Math.
  Anal., 47 (2015), pp.~477--497.

\bibitem{schneider-wayne}
{\sc G.~Schneider and C.~E. Wayne}, {\em Counter-propagating waves on fluid
  surfaces and the continuum limit of the {F}ermi-{P}asta-{U}lam model}, in
  International Conference on Differential Equations, B.~Fiedler,
  K.~Gr\"{o}ger, and J.~Sprekels, eds., World Scientific, 2000, pp.~390--404.

\bibitem{vainchtein-survey}
{\sc A.~Vainchtein}, {\em Solitary waves in {F}{P}{U}-type lattices}, Physica
  D: Nonlinear Phenomena,  (2022), p.~133252.

\bibitem{wright-scheel}
{\sc J.~D. Wright and A.~Scheel}, {\em Solitary waves and their linear
  stability in weakly coupled {K}d{V} equations}, Z. agnew Math. Phys, 58
  (2007).

\bibitem{zeidler}
{\sc E.~Zeidler}, {\em Applied functional analysis}, vol.~109 of Applied
  Mathematical Sciences, Springer-Verlag, New York, 1995.
\newblock Main principles and their applications.

\end{thebibliography}

\end{document}